\documentclass{article}
\usepackage[utf8]{inputenc}
\usepackage{amssymb}
\usepackage{amsmath}
\usepackage{amsthm}
\usepackage{accents}
\usepackage[margin=0pt]{subfig}
\usepackage{tikz}
\usepackage{empheq}
\usepackage{doi}
\usetikzlibrary{arrows.meta}
\usepackage{hyperref} 
\usepackage{enumitem}

\newtheorem{theorem}{Theorem}[section]
\newtheorem{lemma}[theorem]{Lemma}

\newtheorem{corollary}[theorem]{Corollary}
\newtheorem{hypothesis}[theorem]{Hypothesis}
\newtheorem{definition}[theorem]{Definition}
\newtheorem{remark}[theorem]{Remark}

\newcommand\R{\mathbb{R}}

\title{Propagation failure in heterogeneous FitzHugh--Nagumo systems via coupled upper and lower solutions}
\author{Mat\'ias Courdurier\footnotemark[1] \and Esteban Paduro\footnotemark[2]}
\date{November 8, 2026}
\begin{document}
\maketitle

\footnotetext[1]{Departamento de Matem\'atica, Facultad de Matem\'aticas, Pontificia Universidad Cat\'olica de Chile, 
Santiago, Chile.}

\footnotetext[2]{Instituto de Ingenier\'ia Matem\'atica y Computacional, Facultad de Matem\'aticas, Pontificia Universidad Cat\'olica de Chile, 
Santiago, Chile. Corresponding author. Email: {\tt esteban.paduro@uc.cl}
}

\begin{abstract}
Traveling waves are fundamental objects in reaction--diffusion systems; however, spatial heterogeneities can distort or inhibit their propagation.
Determining whether a given heterogeneity results in propagation failure remains challenging, particularly in systems with coupled components.
In this paper, we present explicit constructions of stationary coupled upper and lower solutions tailored to specific heterogeneities for a one-dimensional FitzHugh--Nagumo system. By leveraging the mixed quasimonotone structure of the system, these barriers yield sufficient conditions for the failure of propagation for families of trapped initial conditions.
The stationary constructions establish directional signal blocking in FitzHugh--Nagumo neuron models with localized 
reaction heterogeneities, including light-gated ionic conductances, and with abrupt geometric variations; the resulting stationary barriers persist for sufficiently small recovery diffusion.
Additionally, we introduce explicit almost stationary barriers, which provide conditional bounds on the expansion of componentwise threshold regions in homogeneous media.
\end{abstract}

\vspace{0.2 cm}

{\bf Keywords:} directional blocking, upper and lower solutions, reaction--diffusion systems, mixed quasimonotonicity, comparison principle, heterogeneous media, optogenetics

\vspace{0.2 cm} {\bf AMS subject classifications:} 35K57, 35B51, 35Q92

\section{Introduction}

The study of propagation and failure of propagation in reaction--diffusion equations has been a central focus in mathematical biology, physics, and applied analysis since the foundational work of 
Kolmogorov, Petrovskii, and Piskunov 
\cite{kolmogorov_study_1937} and 
Fisher 
\cite{fisher_wave_1937}. Classical theory primarily addresses conditions for the successful spread of an invading state throughout a medium and focuses on questions such as existence, stability, and multiplicity of traveling or transition fronts, both for scalar equations \cite{mckean_nagumos_1970,volpertTravelingWaveSolutions2000,xinFrontPropagationHeterogeneous2000,zlatosExistenceNonexistenceTransition2017a} and for systems \cite{rinzel_traveling_1973,hastingsSingleMultiplePulse1982,jonesStabilityTravellingWave1984b}. 
In scalar KPP--Fisher equations, the complementary problem of
characterizing the asymptotic spreading of compactly supported solutions in general heterogeneous media has been developed in \cite{berestycki_asymptotic_2022}.
In contrast, mechanisms underlying propagation failure---in which an initial perturbation does not invade or propagate through the medium---have been studied less systematically.

In this work, we study the failure of propagation in a heterogeneous reaction--diffusion system by identifying an obstruction to propagation through a comparison principle and explicit front-like coupled upper and lower solutions.

\subsection{Comparison principles and failure of propagation}

For scalar bistable equations, the comparison principle allows propagation failure to be established using stationary solutions that impede propagation \cite{fifeClinesInducedVariable1981,
pauwelussenNerveImpulsePropagation1981,
lewis_wave-block_2000,
dirr_pinning_2006}. In heterogeneous media, localized variations may generate stationary profiles and propagation failure. Examples include passive or reaction-free intervals \cite{sneyd_propagation_1997, lewis_wave-block_2000,caputoStoppingReactiondiffusionFront2021}, transitions between distinct reaction regimes \cite{LiangZhangZhang2026}, variations in diffusion or geometry \cite{ikedaPropagationExcitedState1989,berestyckiFrontBlockingPropagation2016}, and periodic or geometric pinning mechanisms \cite{dirr_pinning_2006,caputoBlocking2DBistable2026}. Such propagation failure can often be studied using stationary profiles and parabolic comparison arguments \cite{fifeClinesInducedVariable1981}. 
The existence of a transition front is itself not automatic in heterogeneous bistable media, and its nonexistence may reflect either stationary blocking or the development of multiple propagating interfaces \cite{zlatosExistenceNonexistenceTransition2017a}.

For reaction--diffusion systems, the scalar paradigm breaks down in a fundamental way. In general, systems do not satisfy a maximum or comparison principle, and propagation failure cannot be reduced to pointwise ordering arguments \cite{volpertTravelingWaveSolutions2000}. Still, structural conditions such as quasimonotonicity permit comparison with suitably coupled upper and lower solutions \cite{termanComparisonTheoremsReactionDiffusion1982, luMaximumPrinciplesComparison1993,
paoNonlinearReactiondiffusionSystems1982, paoNonlinearParabolicElliptic1993}. The nature of the solutions also matters: fronts connect two distinct equilibrium states of a system, whereas pulses are localized bumps that decay at infinity. In the scalar setting, comparison with a stationary front completely characterizes a family of front solutions that never cross it. In the systems considered here, the relevant propagating structures are pulses, but a pointwise comparison with a stationary pulse 
has a much weaker interpretation than comparison with a stationary front.

Propagation failure in systems may manifest through stationary fronts or localized stationary pulses \cite{chen_stability_2014}. For instance, in homogeneous double-diffusive FitzHugh--Nagumo (FHN)-type systems, positive localized stationary solutions were constructed in \cite{klaasenStationaryWaveSolutions1984, reineckePositiveSolutionRNto1999}. In contrast, 
\cite{chenStandingPulseSolutions2012} constructed standing pulse solutions with sign changes, whose stability was subsequently investigated in \cite{chen_stability_2014}. 
These stationary structures suggest that, in systems, propagation failure may correspond to pinning---convergence toward nontrivial localized stationary attractors---in addition to convergence to constant steady states.  

Spatial heterogeneity further enriches the dynamics of systems and can induce drift and deformation of pulses
\cite{nishiuraDynamicsTravelingPulses2007b,
chen_heterogeneity-induced_2018}.
Localized changes in the reaction dynamics, arising for instance from passive or weakly excitable regions, can support pinned or stationary structures, including heteroclinic profiles obtained by sub- and supersolution methods and fronts or pulses pinned by jump-type reaction defects \cite{kajiwara_sub-supersolution_2018, vanheijsterPinnedSolutionsHeterogeneous2019}.
Such heterogeneities need not be intrinsic to the medium; they may also be imposed externally. An important example is optogenetic stimulation, in which patterned illumination activates light-sensitive ion channels and can create localized, reversible changes in excitability or conduction \cite{watanabeOptogeneticManipulationAnatomical2017}.
Experimental and computational studies have examined light-induced control of wave dynamics in cardiac tissue models, including the real-time control, drift, and termination of spiral waves \cite{majumderOptogeneticsEnablesRealtime2018, hussainiDriftTerminationSpiral2021}.
These studies show that light can be used to steer and terminate excitation waves; to our knowledge, however, rigorous sufficient conditions under which such a light-induced change in excitability blocks propagation have not been established.

A distinct source of heterogeneity is spatial variation in transport or geometry. In models of nerve conduction, changes in axonal geometry affect pulse shape and conduction velocity \cite{goldstein_changes_1974}, while abrupt diameter changes and branching can lead to one-way transmission or blocking in reduced reaction--diffusion models \cite{keenerCausesPropagationFailure1987, pauwelussenOneWayTraffic1982, ikedaWaveBlockingPhenomenaBistable1989a, ikedaStabilityAnalysisStationary1991}.

Taken together, these results highlight a structural distinction. In scalar bistable equations, failure of propagation is closely tied to comparison principles and stationary solutions \cite{pauwelussenNerveImpulsePropagation1981,lewis_wave-block_2000}, whereas in systems, blocking may arise from geometric obstructions, localized stationary solutions, or timescale effects \cite{volpertTravelingWaveSolutions2000}. In the terminology of \cite{keenerCausesPropagationFailure1987}, geometric obstructions correspond primarily to geometric block, whereas recovery-dependent and rate-dependent mechanisms are forms of functional block.

Our approach combines comparison for mixed quasimonotone systems with explicit front-like coupled upper and lower solutions to identify obstructions to pulse propagation. It builds on the reformulation of the maximum principle in terms of comparison functions \cite{chuehPositivelyInvariantRegions1977}, and the theory of upper and lower solutions for quasimonotone systems \cite{termanComparisonTheoremsReactionDiffusion1982,luMaximumPrinciplesComparison1993,paoNonlinearParabolicElliptic1993}. Related constructions of upper and lower comparison functions have been used to establish propagation failure in \cite{pauwelussenOneWayTraffic1982,berestyckiFrontBlockingPropagation2016,almeida_wave_2022}. 

The mixed-quasimonotone comparison principle is classical in origin; our contribution is to combine an interface-compatible formulation of that principle with explicit coupled front-like barriers adapted to heterogeneous FitzHugh--Nagumo systems. The stationary constructions yield sufficient conditions for directional blocking of families of trapped initial data through a common profile mechanism, while the results for small recovery diffusion and for almost stationary barriers provide, respectively, a robustness theorem and conditional bounds on threshold-region expansion.

This work is organized as follows. In Section \ref{section_definitions_and_main_results}, we describe our framework and state our main results. 
Sections \ref{section_barrier_heterogeneous}--\ref{section_moving_barrier} develop the stationary and almost stationary barrier constructions. 
Section \ref{section_numerical_experiments} presents numerical simulations illustrating the heterogeneous and homogeneous barrier constructions and exploring parameter regimes beyond those covered by the theory.

\section{Framework and main results}\label{section_definitions_and_main_results}
\subsection{The FitzHugh--Nagumo model}
We consider the one-dimensional FHN system \cite{FitzHugh1961, nagumo_active_1962} as a model problem for studying propagation and directional blocking in heterogeneous reaction--diffusion systems. 
\begin{empheq}[left=\empheqlbrace]{equation}\label{FHN_pde_general}
\begin{aligned}
\partial_t v - \frac{1}{a(x)} \partial_x\bigl(b(x)\partial_x v\bigr) &= f(v,x) - w, &t\in (0,T), x\in \R,\\
\tau \partial_t w - D \partial_x^2 w & = v - \gamma w, &t\in (0,T), x\in \R, \\
\bigl(v(x,0),w(x,0)\bigr) &= \bigl(v^\mathrm{in}(x), w^\mathrm{in}(x)\bigr),  &x\in \R,\\
\lim_{\lvert x\rvert\to \infty} (\lvert v(x,t)\rvert+\lvert w(x,t)\rvert) &= 0, &t\in (0,T).\\
\end{aligned}
\end{empheq}
Following standard terminology, we call $v$ the activator variable and $w$ the recovery variable; correspondingly, the first and second equations in \eqref{FHN_pde_general} are the activator and recovery equations.
Here $\gamma, \tau>0$ and $D\geq0$. The recovery-diffusion coefficient $D$ may be constant or spatially dependent, as specified in each result, and is assumed to be either identically zero or piecewise smooth, with finitely many jump discontinuities, and bounded above and below by positive constants. We assume that there exist constants $m, M>0$ such that
\(0<m\leq a(x)\leq M\), \(0\leq b(x)\leq M\) for all \(x\in\mathbb R\). The functions $a$ and $b$ are piecewise smooth with finitely many jump discontinuities.
We assume that, for every fixed $v\in\mathbb R$, the map $x\mapsto f(v,x)$ is piecewise continuous with finitely many jump discontinuities.
For every $R>0$, there exists $L_R>0$ such that \(
\lvert f(v,x)-f(\tilde v,x)\rvert
\leq L_R\lvert v-\tilde v\rvert
\)
for all $\lvert v\rvert,\lvert\tilde v\rvert\leq R$ and all $x\in\mathbb R$. Finally, $f$ is locally bounded uniformly in $x$: 
\begin{equation}\label{local_uniform_boundedness}
\sup_{\substack{x\in\mathbb R\\ \lvert v\rvert\leq R}}
\lvert f(v,x)\rvert<\infty
\qquad\text{for every }R>0.
\end{equation}
A standard homogeneous choice is the cubic nonlinearity $f(v,x) = v(1-v)(v-\alpha)$ with $\alpha \in (0,1)$. 

For the construction of stationary barriers in heterogeneous media, we assume that \( f(v,x)=f_0(v) \) outside a compact set and that the homogeneous reaction term $f_0:\mathbb R\to\mathbb R$ satisfies:
\begin{equation}\label{condition_H1} 
\begin{split}
&\text{(i)}~f_0(0)=0,\\
&\text{(ii)}~-f_0'(0) >1/\gamma,\\
&\text{(iii)}~\text{There exists } v_0\geq0 \text{ such that }
\min\{-f_0(v),f_0(-v)\}\geq \frac{v}{\gamma} \text{ for every }v\geq v_0.\\
&\text{(iv)}~\text{There exists } k^*>0 \text{ such that }
\min\{-f_0(v),f_0(-v)\}> \frac{v}{\gamma} \text{ for every }v\in (0, k^*]. 
\end{split}
\end{equation}
For a cubic nonlinearity $f_0(v) = v(1-v)(v-\alpha)$, condition (i) is satisfied, and condition (ii) is equivalent to
$\alpha>1/\gamma$. Moreover, \(
\min\{-f_0(v),f_0(-v)\}
=v^3+O(v^2)\) as \(v\to+\infty\), so condition (iii) is also satisfied. In contrast, an ignition-type nonlinearity satisfying $f_0(v) = 0$ for $v\in [0, \theta]$, with $\theta>0$, does not satisfy condition~(ii). Conditions (i) and (ii) imply that condition (iv) holds for some sufficiently small \(k^*>0\). We state it separately because the constant \(k^*\) is used explicitly in the heterogeneous constructions.

\begin{remark}
Condition \eqref{condition_H1}.(ii) ensures strong linear damping near the rest state. More precisely, a simple trace-determinant calculation for the spatially homogeneous linearization at $(0,0)$ shows that every eigenvalue $\lambda$ satisfies \( \lvert\operatorname{Re}\lambda\rvert>\lvert\operatorname{Im}\lambda\rvert\) for every \(\tau>0\). Related conditions appear in \cite{Rauch1978,pauwelussenOneWayTraffic1982}.
\end{remark}

To formulate the comparison principle used throughout the paper, we embed \eqref{FHN_pde_general} into a more general class of mixed quasimonotone reaction--diffusion systems, for which comparison methods based on coupled upper and lower solutions are available.

\subsection{Solutions and barriers for mixed quasimonotone systems}
Let $\Omega = (\omega_1,\omega_2)\subset \R$ be an interval, let  $T>0$, and consider
\begin{empheq}[left=\empheqlbrace]{equation}\label{FHN_pde_general_quasimonotone}
\begin{aligned}
\mathcal{L}_1 v:= \partial_t v - \frac{1}{a_1(x)} \partial_x(b_1(x)\partial_x v) &= f_1(v,w,x), &t\in (0,T), x\in \Omega,\\
\mathcal{L}_2 w:= \partial_t w - \frac{1}{a_2(x)} \partial_x(b_2(x)\partial_x w) &= f_2(v,w,x) , &t\in (0,T), x\in \Omega.
\end{aligned}
\end{empheq}

We supplement \eqref{FHN_pde_general_quasimonotone} with prescribed initial data and homogeneous Neumann boundary conditions at each finite endpoint of $\Omega$: if $\omega_1>-\infty$, we impose 
\[ (b_1\partial_x v)(\omega_1^+, t)= (b_2\partial_x w)(\omega_1^+,t)=0,\] 
and if $\omega_2<\infty$, we impose 
\[(b_1\partial_x v)(\omega_2^-,t)= (b_2\partial_x w)(\omega_2^-,t)=0.\]
At an infinite endpoint, the corresponding components are assumed to converge to zero. We assume that, for $i=1,2$, \(0<m\leq a_i(x)\leq M\), \(0\leq b_i(x)\leq M\), and that $a_i$ and $b_i$ are piecewise smooth with finitely many jump discontinuities.

\begin{hypothesis}\label{hyp:mixed-quasimonotone}
The functions $f_1,f_2$ satisfy the following conditions:
\begin{enumerate}[label=(\roman*)]
\item For every fixed $(v,x)$, the map
      $w\mapsto f_1(v,w,x)$ is non-increasing.
\item For every fixed $(w,x)$, the map
      $v\mapsto f_2(v,w,x)$ is non-decreasing.
\item For every $(r,s)\in\mathbb R^2$, the maps $x\mapsto f_i(r,s,x)$, $i=1,2$,  are piecewise continuous.
\item For any compact $R\subset \R^2$ there exists a constant $L_R>0$, independent of $x$, such that
\[
\lvert f_i(r,s,x)-f_i(\tilde r,\tilde s,x)\rvert
\leq
L_R\bigl(\lvert r-\tilde r\rvert +\lvert s-\tilde s\rvert \bigr),\qquad i=1,2,
\] 
for all $(r,s), (\tilde{r},\tilde{s})\in R$.
\end{enumerate}
\end{hypothesis}
\noindent
Conditions (i) and (ii) express the mixed quasimonotonicity of the reaction term $(f_1,f_2)$.

System \eqref{FHN_pde_general} is of the form \eqref{FHN_pde_general_quasimonotone} with
\[
a_1=a,\qquad b_1=b,\qquad f_1(v,w,x)=f(v,x)-w,\qquad f_2(v,w,x)=\frac{v-\gamma w}{\tau},
\]
and $(a_2,b_2)=(1,0)$ if $D\equiv0$, or $(a_2,b_2)=(\tau/D,1)$ if $D$ is bounded above and below by positive constants (adjusting $m$ and $M$ if necessary). Hypothesis~\ref{hyp:mixed-quasimonotone} holds because $\partial_w f_1=-1$ and $\partial_v f_2=1/\tau>0$. In particular, at each jump of $D$ the transmission condition \eqref{transmission_conditions0} in Definition~\ref{definition_solution} for the recovery variable reads $[\partial_x w]=0$.

For a function $g$ with one-sided traces at $\xi$, write $[g]_\xi = g(\xi^+) - g(\xi^-)$. We use the following notion of a piecewise classical solution. 

\begin{definition}[Piecewise classical solution]
\label{definition_solution}
Let $\Omega=(\omega_1,\omega_2)$, where $-\infty\leq \omega_1<\omega_2\leq+\infty$.
A continuous function \(
(v,w):\Omega\times[0,T)\to\mathbb R^2 \)
is called a piecewise classical solution of
\eqref{FHN_pde_general_quasimonotone} if the following conditions hold.

\begin{enumerate}
\item There exist finitely many fixed interface points
\(\omega_1=p_0<p_1<\cdots<p_N=\omega_2\) such that $a_i,b_i$ are smooth and $x\mapsto f_i(\cdot,\cdot, x)$ is continuous on each $I_k=(p_{k-1},p_k)$, and both components belong to
\(C^{2,1}(I_k\times(0,T))
\cap
C^{1,0}(\overline{I_k}\times(0,T)),
\)
where the closure is taken in $\mathbb R$. Regularity up to the spatial boundary is required only when the corresponding endpoint of $I_k$ is finite. In particular, the one-sided spatial derivatives are well defined at each interior interface, but they are not required to coincide. 

\item The equations in \eqref{FHN_pde_general_quasimonotone} hold pointwise on $I_k\times(0,T)$ for every $k$.

\item At every interior interface $p_k$, the solution satisfies  the transmission conditions expressing continuity of the diffusive fluxes: 
\begin{equation}\label{transmission_conditions0}
\relax [b_1\partial_x v]_{p_k}=0,
\qquad
\relax [b_2\partial_x w]_{p_k}=0,
\qquad k=1,\dots,N-1.
\end{equation}
These conditions ensure that the divergence terms do not produce singular measures at the interfaces.

\item The initial and boundary conditions are satisfied, and at each infinite endpoint the prescribed decay condition holds.
\end{enumerate}
\end{definition}
 
We next introduce the coupled upper and lower solutions used in the comparison argument. Because the reaction term is mixed quasimonotone, the upper and lower components must be coupled across the two equations.

This definition is closely related to upper and lower solutions for mixed quasimonotone systems \cite{paoNonlinearReactiondiffusionSystems1982, paoNonlinearParabolicElliptic1993}, while allowing piecewise smooth profiles, transmission inequalities, and vanishing conductivities.

\begin{definition}[Barrier]\label{defi_barrier}
Let $\Omega=(\omega_1,\omega_2)$, where $-\infty\leq \omega_1<\omega_2\leq+\infty$. A continuous function $(\bar{v},\bar{w},\underaccent{\bar}{v},\underaccent{\bar}{w}): \Omega \times [0,T) \to \R^4$ is called a barrier for system \eqref{FHN_pde_general_quasimonotone} if $\underaccent{\bar}{v}\leq \bar{v}$, $\underaccent{\bar}{w}\leq \bar{w}$, and the following conditions hold.
\begin{enumerate}
\item There exist finitely many interface locations $p_k\in C^1([0,T))$, $k=1,\dots,N-1$, such that $\omega_1=p_0<p_1(t)<\dots<p_{N-1}(t)<p_N=\omega_2$, for every $t\in [0,T)$. Every discontinuity of $a_i$, $b_i$, or $x\mapsto f_i(\cdot,\cdot, x)$ is represented by a stationary interface. The remaining interfaces may depend on time.
Every moving interface lies in a region where the coefficients $a_i$, $b_i$ are smooth and $x\mapsto f_i(\cdot,\cdot, x)$ is continuous.

\item On each region \(\mathcal Q_k
= \{(x,t):p_{k-1}(t)<x<p_k(t),\ 0<t<T\}
\) the coefficients are smooth, and the four components belong to \(C^{2,1}(\mathcal Q_k)\). The components and their first spatial derivatives possess continuous one-sided traces on the finite lateral boundaries of $\mathcal Q_k$, and the components extend continuously at $t=0$.

\item The coupled differential inequalities
\begin{empheq}[left=\empheqlbrace]{equation}\label{system_sub_super_solutions}
\begin{aligned}
\mathcal{L}_1 \bar{v}  -f_1(\bar{v},\underaccent{\bar}{w},x) &\geq 0 \geq \mathcal{L}_1 \underaccent{\bar}{v} - f_1(\underaccent{\bar}{v}, \bar{w},x) ,\\
\mathcal{L}_2 \bar{w} - f_2(\bar{v}, \bar{w},x) &\geq 0 \geq \mathcal{L}_2 \underaccent{\bar}{w} - f_2(\underaccent{\bar}{v},\underaccent{\bar}{w},x) ,
\end{aligned}
\end{empheq}
hold pointwise in \(\mathcal Q_k\).

\item At each interior interface $\xi= p_k(t)$, $k=1, \dots,N-1$, we require the transmission inequalities
\begin{equation}\label{transmission_condition}
\relax [b_1 \partial_x \underaccent{\bar}{v}]_{\xi} \geq 0, \quad 
[b_2 \partial_x \underaccent{\bar}{w}]_{\xi} \geq 0, \quad 
[b_1 \partial_x \bar{v}]_{\xi} \leq 0, \quad [b_2 \partial_x \bar{w}]_{\xi} \leq 0.
\end{equation}

\item At finite endpoints, we require the following outgoing flux conditions. If $p_0>-\infty$, 
\[
\begin{aligned}
&(b_1\partial_x\bar v )(p_0^+,t)\leq0,
&(b_2\partial_x\bar w)(p_0^+,t)\leq0,\\
&(b_1\partial_x\underaccent{\bar}{v})(p_0^+,t)\geq0,
&(b_2\partial_x\underaccent{\bar}{w})(p_0^+,t)\geq0.
\end{aligned}
\]
And if $p_N<\infty$,
\[
\begin{aligned}
&(b_1\partial_x\bar v )(p_N^-,t)\geq0,
&(b_2\partial_x\bar w)(p_N^-,t)\geq0,\\
&(b_1\partial_x\underaccent{\bar}{v})(p_N^-,t)\leq0,
&(b_2\partial_x\underaccent{\bar}{w})(p_N^-,t)\leq0.
\end{aligned}
\]
\end{enumerate}
\end{definition}
The key advantage is that barriers satisfy only differential inequalities, allowing substantially more flexibility than exact stationary or traveling-wave solutions.

\begin{remark}
The signs in \eqref{transmission_condition} are precisely those for which the differential inequalities hold across the interfaces in the distributional sense without generating singular contributions of the wrong sign (with our sign convention $[g]_\xi = g(\xi^+) - g(\xi^-)$). 
In the comparison proof, we treat the weighted time derivative globally on a fixed spatial domain and integrate only the diffusion term piecewise. Consequently, moving interfaces contribute only through spatial flux jumps, and no interface-velocity terms appear.
\end{remark}
\begin{lemma}[Symmetric stationary barriers for $D=0$]
\label{lem:symmetric_stationary_barriers}
Consider system \eqref{FHN_pde_general} with $D=0$, and let $V:\Omega\to[0,\infty)$ be continuous and piecewise $C^2$, possess well-defined one-sided derivatives at its interfaces, and satisfy the regularity and interface requirements of Definition \ref{defi_barrier}.
Then \((\bar v, \bar w, \underaccent{\bar}{v}, \underaccent{\bar}{w}) = ( V, V/\gamma, -V, -V/\gamma )\) is a stationary barrier if and only if, on every subinterval where $V$ is $C^2$,
\begin{equation}
\label{system_trapping_reduced}
\frac{1}{a(x)}
\partial_x(b(x)\partial_x V)
\leq
\min\{-f(V,x),f(-V,x)\}
-\frac{V}{\gamma},
\end{equation}
and, at every interior interface $\xi$,
\begin{equation}
\label{transmission_condition_symmetric}
[b\partial_xV]_\xi\leq0.
\end{equation}
At finite endpoints, $V$ is required to satisfy $b \partial_x V\leq0$ at the finite left endpoint and $b \partial_x V\geq0$ at the finite right endpoint, using interior traces.
\end{lemma}
\begin{proof}
Under the symmetric ansatz \(
(\bar v,\bar w ,\underaccent{\bar}{v} , \underaccent{\bar}{w}) = ( V, V/\gamma, -V, -V/\gamma )\), the two recovery inequalities hold with equality because the barrier is stationary and $D=0$. The two activator inequalities reduce to
\(\frac{1}{a(x)}
\partial_x(b(x)\partial_x V)
\leq -f(V,x)-\frac{V}{\gamma}
\) and
\(
\frac{1}{a(x)}
\partial_x(b(x)\partial_x V)
\leq f(-V,x)-\frac{V}{\gamma},
\) which are equivalent to \eqref{system_trapping_reduced}. Finally, the transmission inequality for $\bar{v}$ is
$[b\partial_xV]_\xi\leq0$, while the one for $\underaccent{\bar}{v}$ follows from \(
[b\partial_x(-V)]_\xi=-[b\partial_xV]_\xi\geq0\). 
Since $D=0$ corresponds to $b_2\equiv0$, the transmission inequalities for $\bar{w}$ and $\underaccent{\bar}{w}$ are vacuous.
The boundary inequalities follow similarly.
\end{proof}

Lemma \ref{prop_first_energy_estimate} shows that the required boundary condition for the barrier is compatible with the boundary condition imposed for \eqref{FHN_pde_general_quasimonotone} (see \cite{paoNonlinearParabolicElliptic1993} for other boundary conditions). 

\begin{remark}
The coupling between the upper and lower components in
\eqref{system_sub_super_solutions} is essential, since a one-sided upper solution does not generally yield a comparison result for a mixed quasimonotone system. Every piecewise classical solution defines a degenerate barrier by taking the upper and lower components equal; thus, the comparison principle below simultaneously yields uniqueness 
among bounded piecewise classical solutions satisfying the regularity assumptions of Lemma \ref{prop_first_energy_estimate}.
Finally, translations and reflections preserve barriers whenever they preserve the coefficients and reaction terms.
\end{remark}

\subsection{The comparison principle}

The following lemma shows that a solution initially trapped between the components of a barrier remains trapped throughout its interval of existence. Related comparison principles for reaction--diffusion systems have been established on bounded domains \cite{paoNonlinearReactiondiffusionSystems1982, sleemanComparisonPrinciplesStrongly1987,luMaximumPrinciplesComparison1993} and unbounded domains \cite{termanComparisonTheoremsReactionDiffusion1982, grindrodComparisonPrinciplesAnalysis1984, tumaComparisonPrinciplesStrongly1988}. The formulation below is adapted to the setting of this paper: it permits nonnegative, possibly vanishing conductivities, together with piecewise smooth barriers and additional integrability assumptions needed for the energy argument. 
Because these properties depend on the underlying existence theory and on the spatial domain, they are stated separately in the following lemma. 
In the explicit ordinary differential equation (ODE) constructions below, however, each relevant conductivity ratio $\sigma_i=b_i/a_i$ is assumed to be either strictly positive or identically zero. 
Throughout the proof, we use the notation $z_+ = \max\{z,0\}$.

\begin{lemma}[Comparison principle]\label{prop_first_energy_estimate}
Let $\Omega=(\omega_1,\omega_2)$, where $-\infty\leq \omega_1<\omega_2\leq+\infty$, and let $T>0$. Assume that $0<m\leq a_i(x)\leq M$, $0\leq b_i(x)\leq M$, $i=1,2$, and the coefficients are piecewise smooth with finitely many jump discontinuities.
Suppose that the reaction term satisfies 
Hypothesis \ref{hyp:mixed-quasimonotone}.

Let $(\bar{v},\bar{w},\underaccent{\bar}{v},\underaccent{\bar}{w})$ be a barrier in $\Omega\times[0,T)$, and let $(v,w)$ be a piecewise classical solution of \eqref{FHN_pde_general_quasimonotone}. Assume that the two solution components and the four barrier components are uniformly bounded on $\Omega\times[0,T)$. Consider the differences
\begin{equation*}
q_1 = v-\bar{v},\quad q_2 =\underaccent{\bar}{w}-w,\quad q_3= \underaccent{\bar}{v}-v, \quad q_4 =w-\bar{w}.
\end{equation*}
Let \( \ell(1)=\ell(3)=1, \ell(2)=\ell(4)=2\). For $i = 1, 2, 3, 4$, assume that:
\begin{enumerate}[label=(A\arabic*),ref=A\arabic*]
\item \label{lemma_assumption_A1} For every \(S<T\) and all sufficiently large \(R>0\) with \( \Omega_R:=\Omega\cap(-R,R)\), one has
\begin{equation*}
q_i \in L^2(0,S;H^1(\Omega_R)), \qquad \partial_t\bigl( a_{\ell(i)} q_i\bigr)\in L^2\Bigl(0,S;\Bigl(H^{1}(\Omega_R)\Bigr)^*\Bigr).
\end{equation*}
\item \label{lemma_assumption_A2} $(q_i)_+ \in C([0,T); L^2(\Omega))$.
\item \label{lemma_assumption_A3} \(
\mathbf 1_{\{q_i >0\}}\partial_x q_i \in L^2(0,S;L^2(\Omega))\) for every $S<T$.
\end{enumerate}
If $\underaccent{\bar}{v}(x,0)\leq  v(x,0) \leq \bar{v}(x,0)$ and $\underaccent{\bar}{w}(x,0)\leq  w(x,0) \leq \bar{w}(x,0)$, then
\begin{equation*}
\underaccent{\bar}{v}(x,t)\leq  v(x,t) \leq \bar{v}(x,t),\qquad \underaccent{\bar}{w}(x,t)\leq  w(x,t) \leq \bar{w}(x,t),
\end{equation*}
for every $(x,t) \in \Omega\times [0,T)$.
\end{lemma}

\begin{proof}
Define the violation functions \(z_i=(q_i)_+\), $i= 1, \dots ,4$.
Let $C$ denote a generic positive constant that may change from line to line. Define
\[
E(t)
=
\int_\Omega
e(x,t)\,dx, \qquad e(x,t) = 
a_1z_1^2+a_2z_2^2+a_1z_3^2+a_2z_4^2.
\]
Fix $S<T$ and a sufficiently large $R$, set $U=\Omega_R$, and consider one of the differences \(q=q_i\), with \(a=a_{\ell(i)}\). By \ref{lemma_assumption_A1}, \(q\in L^2(0,S;H^1(U))\) and \(\partial_t(aq)\in L^2(0,S;H^1(U)^*)\). The Sobolev truncation theorem gives
\[
q_+\in L^2(0,S;H^1(U)),\qquad
\partial_xq_+=\mathbf 1_{\{q>0\}}\partial_xq
\quad\text{a.e.}
\]
Applying Lemma~\ref{lem:weighted_truncation_identity} yields
\begin{equation}\label{weighted_chain_rule}
\frac12\frac{d}{dt}\int_U a(x)q_+(x,t)^2\,dx
=
\langle \partial_t(aq)(t),q_+(t)\rangle_{H^1(U)^*,H^1(U)}
\end{equation}
for almost every \(t\in(0,S)\). Assumption \ref{lemma_assumption_A2} separately provides the global \(L^2\)-continuity of the violation functions.

If \(\Omega\) is bounded, then \(\Omega_R=\Omega\) for all sufficiently large \(R\). Applying the preceding chain rule to each \(q_i\) shows that $E(t)$ is absolutely continuous in $[0,S]$, for every $S<T$. On an unbounded domain, we will argue using cutoff arguments.

The initial ordering implies \(z_i(\cdot,0)=0\) for \(i=1,\ldots,4\), and hence \(E(0)=0\). By the continuity of the solution and the barrier, the statement of the lemma is equivalent to proving that \(E(t)=0\) for every \(t<T\). This will be obtained as an application of Gr\"onwall's inequality.

\noindent
\textbf{Case $\Omega$ bounded.} 
From the first equation in \eqref{FHN_pde_general_quasimonotone} and the first inequality in \eqref{system_sub_super_solutions} we obtain
\begin{equation}\label{system_difference}
a_1(x)\mathcal{L}_1 (v-\bar{v}) \leq a_1(x) \bigl(f_1(v,w,x) - f_1(\bar{v},\underaccent{\bar}{w},x)\bigr)=: J_1.
\end{equation}
By the global boundedness of the solution and the barrier, there exists $R_0>0$ such that all their components take values in $[-R_0,R_0]$. The uniform local Lipschitz assumption gives a constant $L_{R_0}>0$, independent of $x$, such that
\(
\lvert f_i(r,s,x)-f_i(\tilde r,\tilde s,x)\rvert
\leq
L_{R_0}\bigl(\lvert r-\tilde r\rvert+\lvert s-\tilde s\rvert\bigr)
\) 
whenever all arguments belong to $[-R_0,R_0]$. Hence $z_1 (f_1(v,w,x) - f_1(\bar{v},w,x))\leq  L_{R_0} z_1^2$.
Also, $f_1$ is non-increasing in $w$,
hence
$(f_1(\bar{v},w,x) - f_1(\bar{v},\underaccent{\bar}{w},x)) \leq L_{R_0} (\underaccent{\bar}{w} - w)_+ = L_{R_0} z_2$. This implies
\begin{align*}
z_1 J_1 &= a_1(x) z_1 (f_1(v,w,x) - f_1(\bar{v},\underaccent{\bar}{w},x)) \\
&= a_1(x) z_1 (f_1(v,w,x) - f_1(\bar{v},w,x)+ f_1(\bar{v},w,x)  - f_1(\bar{v},\underaccent{\bar}{w},x)) \\
&\leq a_1(x) L_{R_0} z_1^2 + a_1(x) L_{R_0} z_1 z_2 \\
&\leq \frac{3}{2} L_{R_0} a_1(x) z_1^2 + \frac{M}{2m}L_{R_0} a_2(x) z_2^2 \\
&\leq C_1\bigl(a_1(x) z_1^2 + a_2(x)z_2^2 \bigr).
\end{align*}
Analogous computations apply for each component, using the corresponding quasimonotonicity: the non-increasing dependence of $f_1$ on its second argument and the non-decreasing dependence of $f_2$ on its first argument. This gives
\[
\begin{aligned}
a_1 z_1 \mathcal{L}_1 q_1 \leq a_1 z_1\bigl(
f_1(v,w,x)
-f_1(\bar{v},\underaccent{\bar}{w},x)
\bigr)
&\leq C_1( a_1 z_1^2+a_2z_2^2)\leq C  e(x,t),\\
a_2 z_2 \mathcal{L}_2 q_2 \leq  a_2 z_2\bigl(
f_2(\underaccent{\bar}{v},\underaccent{\bar}{w},x)
-f_2(v,w,x)
\bigr)
&\leq C_2( a_2 z_2^2+a_1 z_3^2) \leq C  e(x,t),\\
a_1 z_3 \mathcal{L}_1 q_3 \leq  a_1 z_3\bigl(
f_1(\underaccent{\bar}{v},\bar w,x)
-f_1(v,w,x)
\bigr)
&\leq C_3( a_1 z_3^2+a_2 z_4^2)\leq C  e(x,t),\\
a_2 z_4 \mathcal{L}_2 q_4 \leq  a_2 z_4\bigl(
f_2(v,w,x)-f_2(\bar v,\bar w,x)
\bigr)
&\leq C_4( a_2 z_4^2+a_1 z_1^2)\leq C  e(x,t).
\end{aligned}
\]

We next estimate the left-hand side of \eqref{system_difference}. 
Let \( \omega_1=p_0<p_1(t)<\cdots<p_{N-1}(t)<p_N=\omega_2 \), with $p_k\in C^1([0,T))$, be a partition obtained by taking a common refinement of the partitions in Definitions~\ref{definition_solution} and \ref{defi_barrier}. 
This common refinement includes all fixed solution and coefficient interfaces, as well as all moving barrier interfaces. Set $I_k(t) = (p_{k-1}(t),p_k(t))$ for $k=1, \dots, N$. 

For almost every \(\tau\in(0,t)\), the differential inequality \eqref{system_difference} holds pointwise on each interval \(I_k(\tau)\). Multiplying it by the nonnegative function \(z_1(\tau)\), integrating over each \(I_k(\tau)\), summing over the common refinement, and then integrating in time gives
\begin{equation*}
J_2 + J_3 \leq \int_0^t\int_\Omega J_1z_1\,dx\,d\tau,
\end{equation*}
where
\[
J_2
=
\int_0^t
\sum_{k=1}^N
\int_{I_k(\tau)}
a_1\partial_\tau q_1 z_1
\,dx
\,d\tau,\qquad 
J_3
=
-\int_0^t
\sum_{k=1}^N
\int_{I_k(\tau)}
\partial_x(b_1\partial_x q_1)
z_1\,dx\,d\tau.
\]
We next identify the piecewise expression defining \(J_2\) with the global weighted distributional derivative. In the sense of distributions on \(\Omega\times(0,t)\),
\[
\partial_\tau(a_1q_1)
=
a_1(\partial_\tau q_1)_{\mathrm{pw}}
-
\sum_{k=1}^{N-1}
p_k'(\tau)[a_1q_1]_{p_k(\tau)}
\delta_{p_k(\tau)},
\]
where $(\partial_\tau q_1)_{\mathrm{pw}}$ denotes the piecewise classical time derivative, without the singular contributions supported on the interfaces. The interface terms vanish. Indeed, coefficient and solution interfaces are stationary, whereas at a moving barrier interface \(a_1\) is smooth and \(q_1=v-\bar v\) is continuous. Hence \( p_k'(\tau)[a_1q_1]_{p_k(\tau)}=0 \) at every interface, and therefore \( \partial_\tau(a_1q_1) = a_1(\partial_\tau q_1)_{\mathrm{pw}} \) distributionally. Since \(z_1\in L^2(0,t;H^1(\Omega))\), this distributional identity extends to \(z_1\) by approximation, and the weighted chain rule \eqref{weighted_chain_rule} gives
\begin{equation*}
J_2
=
\int_0^t
\langle
\partial_\tau( a_1 q_1)(\tau),z_1(\tau)
\rangle_{H^1(\Omega)^*,H^1(\Omega)}
\,d\tau
=
\frac12E_1(t)-\frac12E_1(0),
\end{equation*}
where \(
E_1(t)=\int_\Omega a_1z_1(x,t)^2\,dx
\). For $J_3$, we use that for almost every $t$, the function $z_1(t)\in H^1(\Omega)$, which gives a continuous representative in one dimension. The piecewise classical regularity permits integration by parts on each $I_k(t)$ of the common refinement. Additionally, the Sobolev truncation identity $\partial_x z_i = \mathbf 1_{\{q_i >0\}} \partial_x q_i$ a.e. gives
\begin{align*}
J_3&= 
- \int_0^t\sum_{k=1}^N \int_{I_k(\tau)}
\partial_x(b_1\partial_x q_1)
z_1\,dx\,d\tau\\
&=\int_0^t\sum_{k=1}^N \int_{I_k(\tau)}
b_1 \partial_x z_1
\partial_x q_1\,dx\,d\tau -
\int_0^t
\sum_{k=1}^N
b_1\partial_x q_1 z_1
\big\vert_{p_{k-1}(\tau)^+}^{p_k(\tau)^-} d\tau\\
&=
\int_0^t \int_\Omega b_1 \lvert\partial_x z_1\rvert^2 dx d\tau
+\int_0^t \sum_{k=1}^{N-1}  z_1(p_k(\tau),\tau) [b_1\partial_x q_1]_{p_k(\tau)} d\tau\\
&\qquad + \int_0^t \bigl(
z_1(p_0,\tau) (b_1 \partial_x q_1)(p_0^+,\tau) -z_1(p_N,\tau) (b_1 \partial_x q_1)(p_N^-,\tau)  \bigr)d\tau.
\end{align*}

All terms in \(J_3\) are nonnegative. The bulk term is nonnegative because \(b_1\geq0\). The endpoint terms have the required signs by the homogeneous Neumann condition for the solution and the outgoing flux inequalities for the barrier. At an interior interface, \[ [b_1\partial_xq_1] = [b_1\partial_xv]-[b_1\partial_x\bar v] = -[b_1\partial_x\bar v] \geq0, \] because the solution flux is continuous across every interface of the common refinement.
Putting everything together gives
\begin{equation*}
\frac{1}{2}E_1(t) - \frac{1}{2} E_1(0) = J_2 \leq J_2 + J_3
\leq \int_0^t \int_\Omega C_1( a_1(x) z_1^2+a_2(x)z_2^2) dx d\tau \leq C \int_0^t E(\tau) \,d\tau.
\end{equation*}
The constant \(C>0\) depends only on the uniform bounds of the solution and barrier, the local Lipschitz constants of \(f_1,f_2\), and the bounds on \(a_1,a_2\). Summing the four estimates, multiplying by $2$, and renaming the constant gives
\(E(t)\leq E(0)+C \int_0^tE(s)\,ds\). Since \(E(0)=0\) and $E(t)$ is continuous, Gr\"onwall's
inequality yields \(E(t)=0\) for every \(t\in[0,T)\). We conclude that each $z_i(\cdot,t)$, $i=1,\dots, 4$, vanishes almost everywhere in $\Omega$. Since the solution and barrier components are continuous, \(z_i(\cdot,t)\) is continuous in space. Hence \(z_i(x,t)=0\) for every \((x,t)\in\Omega\times[0,T)\).

\noindent
\textbf{Case $\Omega$ unbounded.} For an unbounded interval, we localize the preceding estimate. Let $\chi\in C_c^\infty(\mathbb R)$ satisfy $0\leq\chi\leq1$, $\chi\equiv1$ on $[-1,1]$, and $\chi\equiv0$ outside $[-2,2]$, and set $\chi_R(x)=\chi(x/R)$. Then
\[
\lvert \partial_x\chi_R\rvert\leq \frac{C}{R},
\qquad
\operatorname{supp}(\partial_x\chi_R)
\subset\{R\leq\lvert x\rvert \leq2R\}.
\]
Consider the localized energy \(
E_R(t)
=
\int_\Omega\chi_R^2
\bigl(
a_1z_1^2+a_2z_2^2+a_1z_3^2+a_2z_4^2
\bigr)\,dx
\). 
To justify the localized time-chain rule, set \(u_{i,R}=\chi_Rq_i\). Multiplication by \(\chi_R\) is bounded on \(H^1(\Omega_{3R})\), and \ref{lemma_assumption_A1} gives 
\[ u_{i,R}\in L^2(0,S;H^1(\Omega_{3R})), \qquad \partial_t\bigl(a_{\ell(i)}u_{i,R}\bigr) = \chi_R\partial_t\bigl(a_{\ell(i)}q_i\bigr) \in L^2(0,S;H^1(\Omega_{3R})^*). \] 
Here  \( \langle \chi_R\partial_t(a_{\ell(i)} q_i),\varphi \rangle := \langle \partial_t(a_{\ell(i)}q_i),\chi_R\varphi \rangle\), \(\varphi\in H^1(\Omega_{3R})\),
which is well defined because multiplication by \(\chi_R\) is bounded on \(H^1(\Omega_{3R})\).
Since \(\chi_R\geq0\), one has \((u_{i,R})_+=\chi_Rz_i\). Applying Lemma~\ref{lem:weighted_truncation_identity} to \(u_{i,R}\) gives
\[
\frac12\frac{d}{dt}\int_{\Omega_{3R}}a_{\ell(i)}\chi_R^2z_i^2\,dx
=
\bigl\langle
\partial_t(a_{\ell(i)}q_i),\chi_R^2z_i
\bigr\rangle_{H^1(\Omega_{3R})^*,H^1(\Omega_{3R})}
\]
for almost every \(t\in(0,S)\). Hence \(E_R\) is absolutely continuous on $[0,S]$ for every $S<T$. 

We repeat the bounded-domain estimate with test function \(\chi_R^2z_i\). The reaction terms are unchanged, while each interface term is multiplied by the nonnegative factor \(\chi_R(p_k(t))^2\) and therefore retains its sign. Since \(\chi_R\) is smooth, it produces no additional jump terms. Any finite physical endpoint of \(\Omega\) is treated as in the bounded case. By contrast, the localized test function vanishes near every artificial endpoint of \(\Omega_{3R}\), so no boundary contribution arises there.

The only new terms arise from differentiating the cutoff in the diffusion terms. In the estimate for $E_1(t)$, we obtain an additional term in $J_3$, which can be bounded as follows.
\begin{equation*}
\biggl\lvert 2\int_0^t\int_\Omega b_1\chi_R z_1
   \partial_x\chi_R\partial_x z_1\,dx\, d\tau\biggr\rvert 
   \leq \frac12\int_0^t\int_\Omega b_1\chi_R^2
       \lvert\partial_xz_1\rvert^2\,dx\,d\tau
+\frac{C}{R^2}\int_0^t
 \int_{\{R\leq\lvert x\rvert\leq2R\} \cap \Omega}z_1^2\,dx\, d\tau,
\end{equation*}
where we used Young's inequality and the uniform upper bound on $b_1$. The first term is absorbed into the nonnegative diffusion term in $J_3$. Applying the same argument to the other three components, integrating in time, we obtain, for every $t<T$,
\begin{equation}\label{gronwall_unbounded}
E_R(t)
\leq E_R(0)+ 
C\int_0^t E_R(s)\,ds
+ \varepsilon_R(t),\,\, \varepsilon_R(t) = \frac{\tilde{C}}{R^2}
 \sum_{i=1}^4
 \int_0^t\int_{\{R\leq\lvert x\rvert\leq2R\} \cap \Omega}
 z_i^2\,dx\,ds,
\end{equation}
where $\tilde{C}>0$ is independent of $R$.
Because $z_i \in L^2(0,t;L^2(\Omega))$, we have
\begin{equation*}
0\leq \varepsilon_R(t) \leq \frac{C}{R^2} \sum_{i=1}^4 \int_0^t \int_\Omega z_i^2 dx ds \to 0,\qquad \text{as }R\to \infty.
\end{equation*}
Next, fix $t<T$. Since $\varepsilon_R(s)\leq\varepsilon_R(t)$ for $0\leq s\leq t$ and $E_R(0)=0$, \eqref{gronwall_unbounded} implies
\begin{equation*}
E_R(s) \leq C\int_0^s E_R(\tau) d\tau + \varepsilon_R(t), \qquad 0\leq s \leq t .
\end{equation*}
Gr\"onwall's inequality therefore gives $E_R(t)\leq e^{C t}\varepsilon_R(t)$ for every $t<T$. All constants in the preceding estimate depend only on the coefficient bounds, the uniform bounds of the solution and barrier components, and the local Lipschitz constants of the reaction terms; in particular, they are independent of $R$. Moreover, because the coefficient and barrier partitions contain only finitely many interfaces, their common refinement is finite on the support of each cutoff. Since $\varepsilon_R(t) \to 0$ as $R \to \infty$, we conclude that $E_R(t)\to 0$ as $R \to \infty$ for every fixed $t<T$. On the other hand, $\chi_R\to1$ pointwise and $0\leq\chi_R\leq1$. For each fixed $t\in [0,T)$, we have
\begin{equation*}
0\leq \chi_R^2\Bigl(a_1 z_1^2 +a_2 z_2^2+a_1 z_3^2+a_2 z_4^2\Bigr)\leq M \sum_{i=1}^4 z_i^2\in L^1(\Omega).
\end{equation*}
The dominated convergence theorem therefore yields \(E_R(t)\to E(t)\) as \(R\to\infty\). Consequently, $E(t)=0$ for every $t<T$, and the conclusion follows as in the bounded case.
\end{proof}

\subsection{Existence and boundedness of solutions}

Lemma \ref{prop_first_energy_estimate} separates the comparison argument from the existence and global boundedness of solutions. This is convenient here because the comparison argument can be formulated directly for discontinuous coefficients and barriers satisfying transmission inequalities, without requiring a unified existence theory for all such problems.

For FitzHugh--Nagumo systems, one standard approach combines local well-posedness with invariant-region estimates. The latter provide uniform bounds and permit continuation of the local solution. This strategy was developed for constant-coefficient systems on unbounded domains in \cite{Rauch1978}; related results for bounded domains and piecewise-constant diffusion appear in \cite{schonbekBoundaryValueProblems1978, pauwelussenExistenceUniquenessNonlinear1981}. More recent variants address heterogeneous nonlinearities \cite{cerpa_approximation_2026} and problems on the half-line with time-dependent Neumann data \cite{rodriguezAnalysisStabilityFitzHughNagumo}. For the standard FitzHugh--Nagumo nonlinearity, the dominant cubic term supplies the constant upper and lower bounds needed in the invariant-region argument.

A second approach constructs solutions through monotone iteration between ordered upper and lower solutions \cite{amannExistencePositiveSolutions1971, sattingerMonotoneMethodsNonlinear1972}. For quasimonotone reaction--diffusion systems, it yields existence and comparison results simultaneously under suitable regularity assumptions \cite{paoNonlinearReactiondiffusionSystems1982, paoNonlinearParabolicElliptic1993}; related comparison arguments on unbounded domains appear in \cite{termanComparisonTheoremsReactionDiffusion1982}. The linear theory for parabolic transmission problems provides the analytic framework required to treat broader classes of discontinuous coefficients; see \cite{ladyzenskajaLinearQuasilinearEquations1968}.

\begin{remark}
The regularity assumptions in Lemma \ref{prop_first_energy_estimate} are imposed in addition to the pointwise and interface regularity in Definitions \ref{definition_solution} and \ref{defi_barrier}. The local $H^1$ assumption on the differences is used through the Sobolev truncation identity 
\[ \partial_x(q_i)_+ = \mathbf 1_{\{q_i>0\}}\partial_xq_i, \]
which makes the positive parts admissible in the energy estimate. 

If the corresponding conductivity is uniformly positive, this regularity is consistent with standard parabolic energy estimates, provided the initial data and forcing have the required regularity. If the conductivity vanishes identically, diffusion does not generate spatial regularity. Spatial regularity may nevertheless be propagated from the initial data by constructing the coupled solution in a common phase space $B\times B$, for example with $B=H^1(\Omega)$ or, in one dimension, $B=W^{k,p}(\Omega)$ with $kp>1$; see \cite{Rauch1978,cerpa_approximation_2026}. 

The global $L^2$ assumptions concern only the ordering violations, not the barrier components themselves. This permits front-like barriers that approach nonzero limits at infinity. For the stationary barriers constructed below, and for the almost stationary barriers by their temporal monotonicity, the initial ordering gives 
\[ 
\begin{aligned} 
(v-\bar v)_+&\leq(v-v^\mathrm{in})_+, & (\underaccent{\bar}{v}-v)_+&\leq(v^\mathrm{in}-v)_+,\\ 
(w-\bar w)_+&\leq(w-w^\mathrm{in})_+, & (\underaccent{\bar}{w}-w)_+&\leq(w^\mathrm{in}-w)_+. 
\end{aligned} 
\]
Hence $L^2$ control of $(v(t)-v^\mathrm{in},w(t)-w^\mathrm{in})$ implies pointwise-in-time $L^2$ control of the four violations. 

Finally, if 
\[ 
q_i\in L^2(0,T;H^1(\Omega)),\qquad \partial_t (a_{\ell(i)}q_i) \in L^2(0,T;H^1(\Omega)^*), 
\] 
then the weighted Lions--Magenes result gives a representative which is continuous in $L^2(\Omega,a_{\ell(i)}(x)\,dx)$. Since $m\leq a_{\ell(i)}\leq M$, this is equivalent to continuity in $L^2(\Omega)$.
\end{remark}

\subsection{Directional blocking}

\begin{definition}[Directional blocking]\label{defi_directional_blocking}
Let \((v,w)\) be a global solution of \eqref{FHN_pde_general}.
We say that the solution $(v,w)$ is directionally blocked toward $-\infty$ at threshold $\theta>0$ if there exists \(x_\theta\in\mathbb R\)
such that
\[
\lvert v(x,t)\rvert \leq \theta, 
\qquad
\lvert w(x,t)\rvert\leq \theta
\]
for every $x\leq x_\theta$, $t\geq0$.
\end{definition}

The definition is intentionally asymmetric in space. It excludes invasion toward $-\infty$ but does not prescribe the behavior toward $+\infty$.

\begin{corollary}[Directional blocking by a stationary barrier]
\label{cor_directional_blocking}
Let \((v,w)\) be a global solution of \eqref{FHN_pde_general}
whose initial data are trapped between the components of a stationary barrier $(\bar v,\bar w,\underaccent{\bar}{v} ,\underaccent{\bar}{w})$ satisfying the assumptions of Lemma
\ref{prop_first_energy_estimate}. Suppose, in addition, that
\[
\lim_{x\to-\infty}
\bigl(
\lvert \bar v(x)\rvert+
\lvert \underaccent{\bar}{v}(x)\rvert+
\lvert \bar w(x)\rvert+
\lvert \underaccent{\bar}{w}(x)\rvert
\bigr)=0.
\]
Then $(v,w)$ is directionally blocked toward $-\infty$ at every threshold.
\end{corollary}

\begin{proof}
Fix $\theta>0$. By the decay of the stationary barrier as $x\to-\infty$, there exists $x_\theta\in\mathbb R$ such that
\(
\lvert \bar v(x)\rvert,\ 
\lvert \underaccent{\bar}{v}(x)\rvert,\ 
\lvert \bar w(x)\rvert,\ 
\lvert \underaccent{\bar}{w}(x)\rvert
\leq\theta\) for all \(x\leq x_\theta\).
For arbitrary $T>0$, Lemma \ref{prop_first_energy_estimate} gives
\(
\underaccent{\bar}{v}(x)\leq v(x,t)\leq\bar v(x)\),
\(
\underaccent{\bar}{w}(x)\leq w(x,t)\leq\bar w(x)
\) for $0\leq t\leq T$. Hence
\(\lvert v(x,t)\rvert\leq\theta\), \(\lvert w(x,t)\rvert\leq\theta\)
whenever $x\leq x_\theta$ and $0\leq t\leq T$.
Since the estimate is valid for every $T>0$, the conclusion follows for all $t\geq0$.
\end{proof}

\subsection{Stationary barriers in heterogeneous media}
The main results of this paper concern the construction of barriers adapted to heterogeneities in the nonlinearity and the diffusion. The results below are intended to be combined with Lemma \ref{prop_first_energy_estimate}. For applications, it is desirable that a barrier encompass a broad family of initial conditions and possess a meaningful asymptotic structure. Thus, whenever a barrier is constructed, every global solution satisfying the assumptions of Lemma \ref{prop_first_energy_estimate} and whose initial data are trapped between the barrier components remains trapped for all time. If, in addition, the barrier converges to zero as $x\to-\infty$, Corollary \ref{cor_directional_blocking} yields directional blocking.

We consider heterogeneities representing qualitatively distinct blocking mechanisms, and the constructions extend to related configurations with minor modifications. Cases \ref{thm_heterogeneity_type1} and \ref{thm_heterogeneity_type2} concern localized changes in the reaction dynamics, representing regions of depressed excitability or localized light-gated ionic conductances. In the latter interpretation, conductance-based descriptions of optogenetic stimulation motivate a reaction term proportional to \(-p(v-E_{\mathrm{rev}})\), where \(E_{\mathrm{rev}}\) is the reversal potential of the light-sensitive channel \cite{williamsComputationalOptogeneticsEmpiricallyDerived2013}; this choice is discussed after Theorem~\ref{thm_barrier_heterogeneous}. Cases \ref{thm_heterogeneity_type3}--\ref{thm_heterogeneity_type5} concern instead abrupt or localized changes in axonal geometry. 

These heterogeneities are partially motivated by previous studies of localized reaction defects \cite{sneyd_propagation_1997,lewis_wave-block_2000}, optogenetically induced conduction block \cite{watanabeOptogeneticManipulationAnatomical2017, williamsComputationalOptogeneticsEmpiricallyDerived2013}, and propagation across changes in axonal geometry \cite{pauwelussenNerveImpulsePropagation1981, pauwelussenOneWayTraffic1982, ikedaPropagationExcitedState1989, ikedaStabilityAnalysisStationary1991}. 
Our proofs provide new insights into the nature of the failure of propagation.

Another important aspect of the construction of barriers is that, because of the additional variables and the inequality relations, system \eqref{system_sub_super_solutions} admits more general solutions than \eqref{FHN_pde_general_quasimonotone}. More specifically, we seek front-like barriers for \eqref{system_sub_super_solutions}, even when profiles with this asymptotic structure are not admissible solutions for \eqref{FHN_pde_general_quasimonotone} or \eqref{FHN_pde_general}.

\begin{theorem}[Stationary barriers in heterogeneous media]
\label{thm_barrier_heterogeneous}
Let $\gamma>0$, $D=0$, and let $f_0:\mathbb R\to\mathbb R$ be locally Lipschitz continuous, differentiable in a neighborhood of zero, and satisfy \eqref{condition_H1}. Assume that one of the conditions \ref{thm_heterogeneity_type1}--
\ref{thm_heterogeneity_type5} holds.

\begin{enumerate}[label=(\roman*)]
\item $a(x)$, $b(x)$ are constant positive functions, $f(v,x) = f_0(v)$ for $x\in [0, L]^c$, and $f(v,x) = f_0(v) + p g(v)$ for $x\in [0,L]$, where \(g:\mathbb R\to\mathbb R\) is locally Lipschitz continuous and there exists \(\eta\in[0,k^*)\) such that \(vg(v)<0\) whenever \(\lvert v\rvert>\eta\).\label{thm_heterogeneity_type1}
\item $a(x)$, $b(x)$ are constant positive functions, $L = p$, $f(v,x) = f_0(v)$ for $x\in [0,L]^c$, and 
\[\min\{-f(v,x),f(-v,x)\}\geq v/\gamma\]
for every $v\geq 0$ and $x\in [0,L]$.\label{thm_heterogeneity_type2}
\item $a(x) = r(x)$, $b(x) = r(x)^2$ where $r(x) = r_0 >0$ for $x\leq 0$ and $r(x)= \frac{1}{1+p} r_0$ for $x> 0$. $f(v,x) = f_0(v).$\label{thm_heterogeneity_type3}
\item $a(x) = r(x)$, $b(x) = r(x)^2$ where $r(x)=r_0>0$ for $x\in [0,L]^c$ and $r(x) = \frac{1}{1+ p}r_0$ for $x\in[0,L]$. $f(v,x) = f_0(v).$\label{thm_heterogeneity_type4}
\item $a(x) = r(x)$, $b(x) = r(x)^2$ where $r(x)= r_0>0$ for $x\in [0,L]^c $ and $r(x)  = (1+p)r_0$ for $x\in[0,L]$. $f(v,x) = f_0(v).$\label{thm_heterogeneity_type5}
\end{enumerate}

For each of these cases, there exists a threshold $p^*\geq0$, depending on the fixed data of that case ($f_0$, $\gamma$, $a$, $b$, $g$, $\eta$, or $r_0$) with the following property. For every $K_0>0$ and every $p\geq p^*$, the remaining geometric parameters can be chosen as specified in the corresponding case: $L$ is chosen in cases \ref{thm_heterogeneity_type1}, \ref{thm_heterogeneity_type4}, and \ref{thm_heterogeneity_type5}, while $L=p$ in \ref{thm_heterogeneity_type2}. There is no additional geometric parameter in case \ref{thm_heterogeneity_type3}. With these choices, for every $\tau>0$, system \eqref{FHN_pde_general} admits a stationary barrier of the form
\[
(\bar v,\bar w,\underaccent{\bar}{v},\underaccent{\bar}{w})
=
\Biggl(
V,\frac{V}{\gamma},-V,-\frac{V}{\gamma}
\Biggr).
\]
The profile $V$ is strictly positive and nondecreasing, and there exist $K\geq K_0$ and $x_0\in\mathbb R$ such that \(
V(x)=K\) for all \(x\geq x_0\). The profile may be chosen so that \(V(x)\to 0\) exponentially as \(x\to-\infty\). Alternatively, for every sufficiently small $V_->0$, it may be chosen
so that \(V(x)=V_-\) for all sufficiently negative $x$.
\end{theorem}

The proof is presented in Section \ref{section_barrier_heterogeneous}, and Figure~\ref{fig_localized_heterogeneity} shows barriers obtained from the construction. Numerical experiments illustrating the results are presented in Section \ref{section_numerical_experiments}. In the statement of the theorem, the threshold \(p^*\) is independent of \(K_0\). 

The affine choice \(g(v)=E_{\mathrm{rev}}-v\) in case~\ref{thm_heterogeneity_type1} provides a simplified conductance-based representation of localized optogenetic stimulation. In this interpretation, \(p\) is the effective light-gated conductance, \(E_{\mathrm{rev}}\) is the channel reversal potential, and \([0,L]\) is the illuminated region. Since \(vg(v)<0\) whenever \(\lvert v\rvert>\lvert E_{\mathrm{rev}}\rvert\), this choice satisfies case~\ref{thm_heterogeneity_type1} with \(\eta=\lvert E_{\mathrm{rev}}\rvert\), and is therefore covered whenever \(\min\{-f_0(v),f_0(-v)\}>v/\gamma\) for all \(v\in(0,\lvert E_{\mathrm{rev}}\rvert]\), since then \eqref{condition_H1}.(iv) holds with some \(k^*>\lvert E_{\mathrm{rev}}\rvert\). This ensures that the construction can enter the heterogeneous region above the sign-change interval of \(g\). For the parameter values \(\alpha=0.3\) and \(\gamma=5\) used in Experiments~1--3 of Section~\ref{section_numerical_experiments}, the largest admissible \(k^*\) is approximately \(0.08\), so this condition restricts the result to light-gated conductances whose reversal potential is close to the rest state (shunting conductances); depolarizing channels with reversal potential above threshold are not covered by the present construction. The formulation is quasi-steady and does not attempt to resolve the channel-state dynamics.

Case~\ref{thm_heterogeneity_type1} with \(g(v)=-v\) is a two-component analog of the leaky gap considered by \cite{lewis_wave-block_2000} for the scalar bistable equation. In that setting, blocking of a front is equivalent to the existence of a nonuniform stationary solution, which appears through a saddle-node bifurcation at a critical gap length; for models with recovery, the scalar critical length is expected to bound the critical length from above, but this is not derived rigorously. In the present construction, no stationary solution of \eqref{FHN_pde_general} is required: the barrier satisfies only differential inequalities and yields directional blocking of pulses for the full activator--recovery system. As in the scalar case, blocking reflects a trade-off between the intensity and the length of the depressed region. Since \(v_2''\geq\kappa p\) on \((0,L)\) in the proof, the heterogeneous interval used in the construction satisfies \(L\leq\sqrt{2\delta/(\kappa p)}\), so stronger depression requires a shorter region; Figure~\ref{fig4a} shows the same qualitative trade-off numerically.

\begin{proof}[Sketch of the proof of Theorem
\ref{thm_barrier_heterogeneous}]
By Lemma~\ref{lem:symmetric_stationary_barriers}, it suffices to construct a nonnegative, nondecreasing profile \(V\) satisfying \eqref{system_trapping_reduced}, together with the transmission inequality \eqref{transmission_condition_symmetric} at every interface.
We construct \(V\) from left to right. On \((-\infty,0]\), a homogeneous tail raises the profile from a sufficiently small nonnegative level to a fixed height \(k>0\), determined by $f_0$, $\gamma$, and, in case~\ref{thm_heterogeneity_type1}, $\eta$, with \(V(0)=k\) and \(V'(0^-)>0\). Across the heterogeneous interval, we consider a maximal increasing solution \(v_2\) of
\begin{empheq}[left=\empheqlbrace]{equation*}
\begin{aligned}
&\frac{1}{a(x)}
    \bigl(b(x)v_2'(x)\bigr)'
    \leq\min\{-f(v_2(x),x),f(-v_2(x),x)\}
    -\frac{v_2(x)}{\gamma},
    &&x\in(0,L),\\
&v_2(0)=k,
    \qquad
    b(0^+)v_2'(0^+)=b(0^-)V'(0^-),
\end{aligned}
\end{empheq}
subject to \(v_2>0\) and \(v_2'>0\). Matching the fluxes at \(x=0\) ensures that the transmission inequality \eqref{transmission_condition_symmetric} holds there with equality. The heterogeneous construction leads to one of two alternatives:
\begin{enumerate}[label=(\alph*)]
\item\label{case:sketch_reaches_favorable}
the profile reaches the favorable height \(v_2(L^-) \geq v_0\) with \(v_2'(L^-)>0\);
\item\label{case:sketch_large_slope}
the profile leaves the heterogeneity with a sufficiently large outgoing slope, 
\[ V'(L^+)> \frac{M}{\sqrt{\sigma_3}}, \] 
where \(L=0\) in case~\ref{thm_heterogeneity_type3}, 
\(\sigma_3=b/a\) in the subsequent homogeneous region, 
and \(M>0\) is chosen so that the subsequent homogeneous continuation remains increasing until it reaches \(v_0\).
\end{enumerate}
Cases~\ref{thm_heterogeneity_type2} and
\ref{thm_heterogeneity_type4} yield
alternative~\ref{case:sketch_reaches_favorable}, whereas
cases~\ref{thm_heterogeneity_type1}, \ref{thm_heterogeneity_type3}, and
\ref{thm_heterogeneity_type5} yield
alternative~\ref{case:sketch_large_slope}.

In the first alternative, the profile enters the homogeneous continuation at or above \(v_0\) with positive slope. In the second, its outgoing slope is large enough for the homogeneous continuation to remain increasing until it reaches \(v_0\). In both cases, it subsequently reaches a height \(K\geq K_0\) and is extended constantly to the right. Joining the pieces so that the corresponding transmission inequalities hold produces the profile required by
Lemma~\ref{lem:symmetric_stationary_barriers}, and hence the stationary
barrier asserted in the theorem.
\end{proof}

\begin{remark}
The stationary barriers provided by Theorem \ref{thm_barrier_heterogeneous} are independent of $\tau$ since the stationary recovery equation with $D=0$ reduces to $v- \gamma w = 0$, so $\tau$ disappears. In the double-diffusive case ($D>0$), the parameter $\tau$ is known to be related to the stability of the stationary pulse solutions \cite{chenStandingPulseSolutions2012,chen_stability_2014}.    
\end{remark}

The construction can also be adapted to sufficiently large bounded domains with homogeneous Neumann boundary conditions by choosing a barrier that is constant for sufficiently negative $x$. The left endpoint may then be placed in this constant region. Since the barrier is also constant for sufficiently large positive $x$, the right endpoint may be treated in the same way.

The assumption that the system is piecewise homogeneous outside the heterogeneous region is mainly used to simplify the presentation, since the focus is to understand how localized or abrupt effects induce obstructions to propagation. The structure of the proof should also permit extensions to more general media under suitable boundedness and regularity conditions.

The following result concerns failure of propagation in the double-diffusive case. This result uses a stationary barrier constructed for the single-diffusive case and argues that the same barrier applies to problems with small recovery diffusion $D>0$. Note that the assumptions of Theorem~\ref{thm_double_diffusive} are compatible with case~\ref{thm_heterogeneity_type2} and with the subcases of case~\ref{thm_heterogeneity_type1} for which \(g(0)=0\). In particular, this includes the zero-reversal-potential choice \(g(v)=-v\), but not the affine choice \(g(v)=E_{\mathrm{rev}}-v\) with \(E_{\mathrm{rev}}\neq0\).

\begin{theorem}[Barrier for a heterogeneous double-diffusive system]\label{thm_double_diffusive}
Let $\gamma>0$ and let $\tau>0$ be arbitrary. Let $0<m\leq a(x),\delta(x)\leq M$ be piecewise smooth with finitely many jump discontinuities, and $b(x)\equiv b_0 >0$. Let $0<\tilde\gamma < \gamma$, and let $f:\mathbb R\times\mathbb R\to\mathbb R$ be locally Lipschitz continuous in its first argument, uniformly in its second argument. Assume also that $f(0,x)=0$ for every $x\in\mathbb R$.

Assume that system \eqref{FHN_pde_general} with $D = 0$ and $\tilde\gamma$ in place of $\gamma$ admits a stationary barrier of the form $(V,V/\tilde{\gamma},-V,-V/\tilde{\gamma})$, where $0\leq V \leq K$. 
Then there exists $d^*>0$, depending only on $K$, $f$, and $(\gamma-\tilde{\gamma})$, such that, if
\[\frac{a(x)\delta(x)}{b_0}\leq d^*
\qquad\text{for every }x\in\mathbb R,
\]
then the same tuple is a stationary barrier for system \eqref{FHN_pde_general} with its original coefficient $\gamma$ and with recovery-diffusion coefficient $D(x)=\delta(x)$.

\end{theorem}

The proof of this result is presented in Section \ref{section_proof_double_diffusive}, and numerical experiments are presented in Section \ref{section_numerical_experiments}.

\subsection{Almost stationary barriers in homogeneous media}

In the homogeneous case, we additionally construct time-dependent barriers that provide quantitative comparison bounds for suitably trapped solutions. Introducing a positive barrier speed leads to a three-dimensional ODE system with several characteristic decay rates. Unlike the stationary construction, this argument does not require condition~\eqref{condition_H1}.(ii), and it yields useful estimates as $\tau\to\infty$. 

After suitable translations and reflections, the almost stationary barriers confine the componentwise threshold-exceedance regions between two prescribed trajectories. 

For comparison, in the classical excitable regime $\alpha<1/2$, the FHN system admits an unstable slow traveling pulse with speed $O(\tau^{-1/2})$ \cite{flores_stability_1991}. The theorem below concerns the opposite regime $\alpha>1/2$ and provides almost stationary barriers with any speed $c(\tau^{-1})$ satisfying the stated scale separation.

\begin{theorem}[Almost stationary barrier in homogeneous media]\label{thm_barriers_conduction_block}
Let $\gamma>0$, $a(x)$ and $b(x)$ be constant positive functions, let $D=0$, and let $f(v,x) = v(1-v) (v-\alpha)$ with $\alpha \in (1/2,1)$. Fix a function 
$c:(0,\infty)\to(0,\infty)$ such that $c(\varepsilon) \to 0$ and $\frac{\varepsilon}{c(\varepsilon)} \to 0$ as $\varepsilon \to 0$ (such as $c(\varepsilon) = \varepsilon^{\beta}$ for $0<\beta <1$). 
There exists $\tau_0>0$ such that given $K_1$, $K_2 >0$, and any $\tau>\tau_0$, upon setting $\varepsilon = \tau^{-1}$, system \eqref{FHN_pde_general} 
admits a barrier moving to the left with speed $\sqrt{b/a}c(\varepsilon)$. For each fixed $t$, the functions 
\[ \bar v(\cdot,t),\qquad \bar w(\cdot,t),\qquad -\underaccent{\bar}{v}(\cdot,t),\qquad -\underaccent{\bar}{w}(\cdot,t) \] 
are positive, nondecreasing, and globally bounded. They decay exponentially to zero as 
\( x+\sqrt{b/a} c(\varepsilon)t\to-\infty. \) 
Moreover, there exists $\xi\in\mathbb R$ such that 
\[ \bar v(x,t),-\underaccent{\bar}{v}(x,t)\geq K_1, \qquad \bar w(x,t),-\underaccent{\bar}{w}(x,t)\geq K_2, \]
whenever 
\( x+\sqrt{b/a} c(\varepsilon)t\geq\xi\).
\end{theorem}

The proof is presented in Section \ref{section_moving_barrier}, and numerical experiments on the results are presented in Section \ref{section_numerical_experiments}. The comparison consequence of Theorem~\ref{thm_barriers_conduction_block} can be formulated in terms of componentwise threshold-exceedance sets. The following corollary records the resulting conditional containment bound.

\begin{corollary}[Conditional threshold-region bound]
\label{cor:threshold_spreading}
Let \(A,B>0\), and assume the hypotheses of Theorem~\ref{thm_barriers_conduction_block}. Choose the almost stationary barrier so that its limiting plateau levels are strictly larger than \(A\) and \(B\), respectively. Let \((\bar v_L,\bar w_L,\underaccent{\bar}{v}_L, \underaccent{\bar}{w}_L)\) be a translated left-moving barrier and let \((\bar v_R,\bar w_R,\underaccent{\bar}{v}_R, \underaccent{\bar}{w}_R)\) be a translated right-moving barrier obtained by spatial reflection.

Let \((v,w)\) satisfy the hypotheses of Lemma~\ref{prop_first_energy_estimate} on every finite time interval, and suppose that, at \(t=0\), it lies componentwise between the lower and upper components of both barriers. For \(t\geq0\), define
\begin{equation*}
\begin{aligned}
\ell_v(t)&:=\sup\{x\in\mathbb R:\bar v_L(x,t)\leq A\},& r_v(t)&:=\inf\{x\in\mathbb R:\bar v_R(x,t)\leq A\},\\
\ell_w(t)&:=\sup\{x\in\mathbb R:\bar w_L(x,t)\leq B\},&  
r_w(t)&:=\inf\{x\in\mathbb R:\bar w_R(x,t)\leq B\}.
\end{aligned}
\end{equation*}
Set
\(x_-(t):=\min\{\ell_v(t),\ell_w(t)\}\) and
\(x_+(t):=\max\{r_v(t),r_w(t)\}\). Then
\[
\{x\in\mathbb R:\lvert v(x,t)\rvert >A\ \text{or}\ \lvert w(x,t)\rvert >B\}
\subseteq [x_-(t),x_+(t)]
\]
for every \(t\geq0\). If \(x_-(t)>x_+(t)\), the threshold-exceedance
set is empty. Moreover, with
\(s:=\sqrt{b/a} c(\varepsilon)\), one has
\(x_-(t)=x_-(0)-st\) and \(x_+(t)=x_+(0)+st\). 
\end{corollary}

\begin{proof}
Fix \(T>0\). Applying Lemma~\ref{prop_first_energy_estimate} to the
left-moving barrier and using the symmetry of its lower and upper
components gives \(\lvert v(x,t)\rvert \leq\bar v_L(x,t)\) and
\(\lvert w(x,t)\rvert \leq\bar w_L(x,t)\) for \(x\in\mathbb R\) and
\(0\leq t\leq T\). Since \(\bar v_L(\cdot,t)\) and
\(\bar w_L(\cdot,t)\) are nondecreasing, it follows that
\(\lvert v(x,t)\rvert \leq A\) and \(\lvert w(x,t)\rvert \leq B\) whenever \(x<x_-(t)\).

The corresponding comparison with the reflected right-moving barrier,
whose upper components are nonincreasing, gives
\(\lvert v(x,t)\rvert \leq A\) and \(\lvert w(x,t)\rvert \leq B\) whenever \(x>x_+(t)\).
Consequently, the threshold-exceedance set is contained in
\([x_-(t),x_+(t)]\). Since \(T>0\) is arbitrary, this holds for every
\(t\geq0\).

Finally, the left-moving barrier depends on
\(\sqrt{a/b} x+c(\varepsilon)t\), while its reflection depends on
\(-\sqrt{a/b} x+c(\varepsilon)t\). Their threshold locations therefore
move with velocities \(-s\) and \(s\), respectively, where
\(s=\sqrt{b/a} c(\varepsilon)\).
\end{proof}

Theorem \ref{thm_barriers_conduction_block} is qualitatively related to the propagation-failure mechanism studied in \cite{Ratas2012}, where high-frequency stimulation is treated using averaging and geometric singular perturbation theory. In that setting, increasing the stimulation causes a traveling pulse to slow down, shrink, and eventually disappear as the effective parameter approaches $\alpha=1/2$ from below. By contrast, our construction applies for $\alpha>1/2$ and produces comparison barriers whose speed tends to zero as $\tau\to\infty$.

\section{Stationary barriers in heterogeneous media: Proof of Theorem~\ref{thm_barrier_heterogeneous}}\label{section_barrier_heterogeneous}

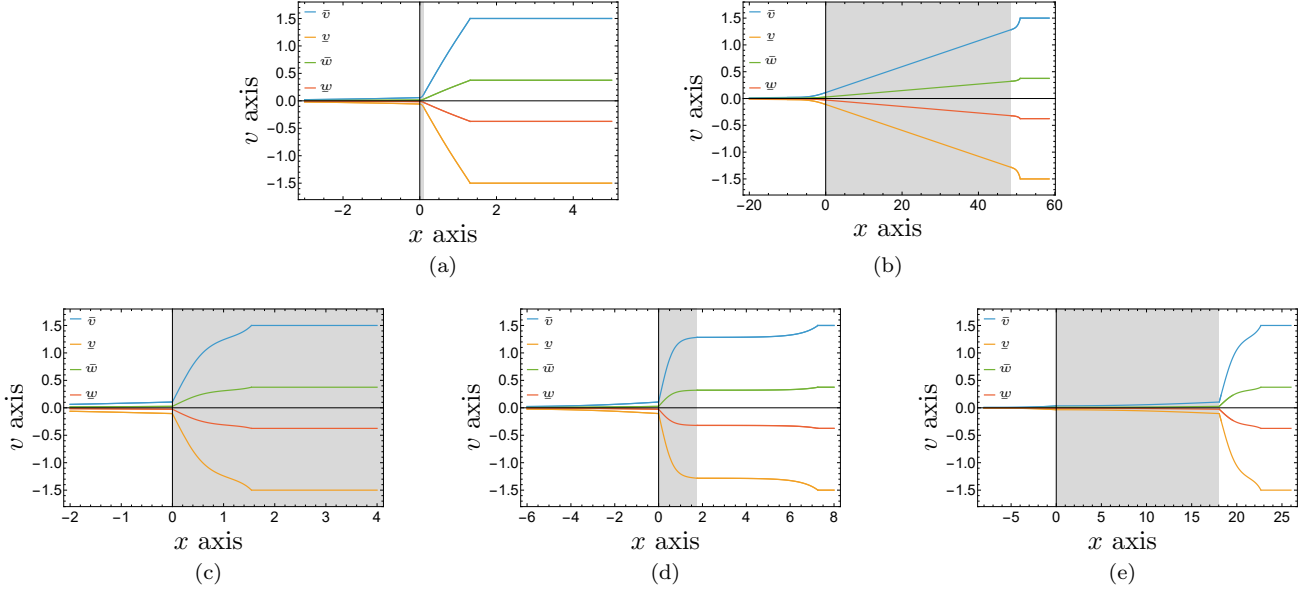
\begin{figure}
    \centering

    \newcommand{\barrierpanel}[2]{%
        \subfloat[\label{#2}]{%
            \begin{minipage}[c][3.2cm][t]{0.3\textwidth}
                \centering
                \begin{tikzpicture}
                    \node[
                        anchor=south west,
                        inner sep=0cm
                    ] (image1) at (0,0.5cm) {%
                        \includegraphics[
                            width=0.9\linewidth
                        ]{#1}%
                    };

                    \begin{scope}[
                        overlay,
                        shift={(image1.south west)},
                        x={(image1.south east)},
                        y={(image1.north west)}
                    ]
                        \node at (0.17,0.94)
                            {\tiny $\bar{v}$};
                        \node at (0.17,0.83)
                            {\tiny $\underaccent{\bar}{v}$};
                        \node at (0.17,0.725)
                            {\tiny $\bar{w}$};
                        \node at (0.17,0.6)
                            {\tiny $\underaccent{\bar}{w}$};
                        \node at (0.5,-0.06)
                            {$x$ axis};
                    \end{scope}
                \end{tikzpicture}
            \end{minipage}%
        }%
    }
    
    \barrierpanel{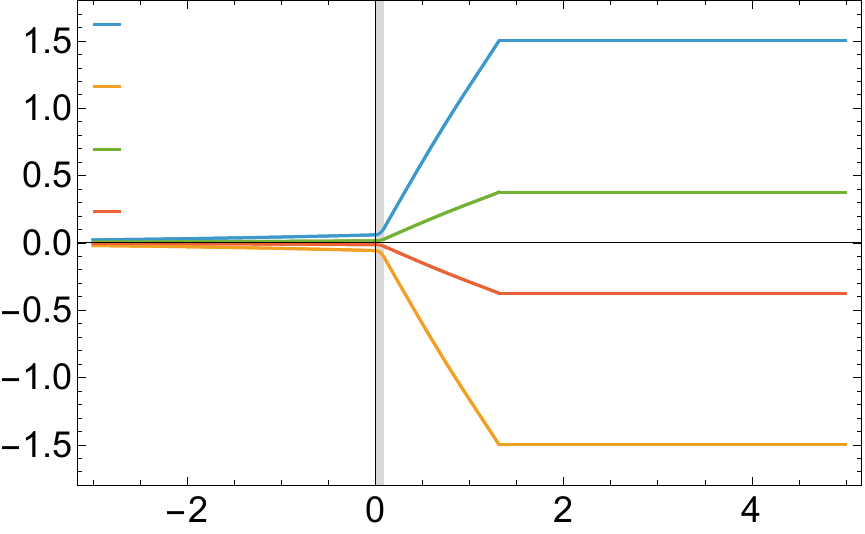}{fig1a}
    \qquad 
    \barrierpanel{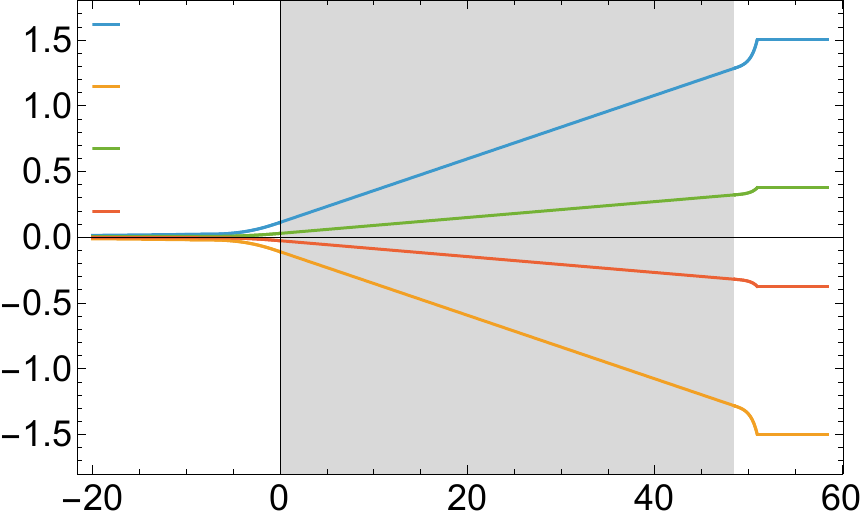}{fig1b}
    
    \par

    \barrierpanel{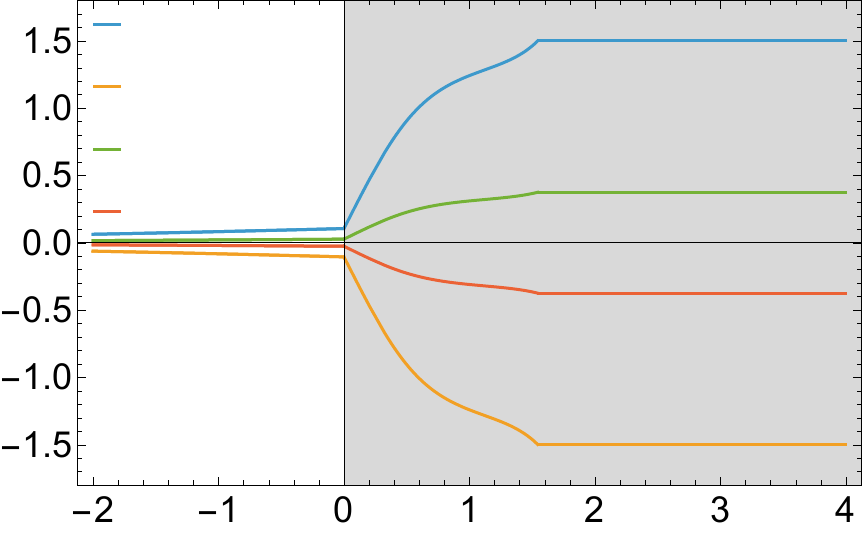}{fig1c}
    \hfill
    \barrierpanel{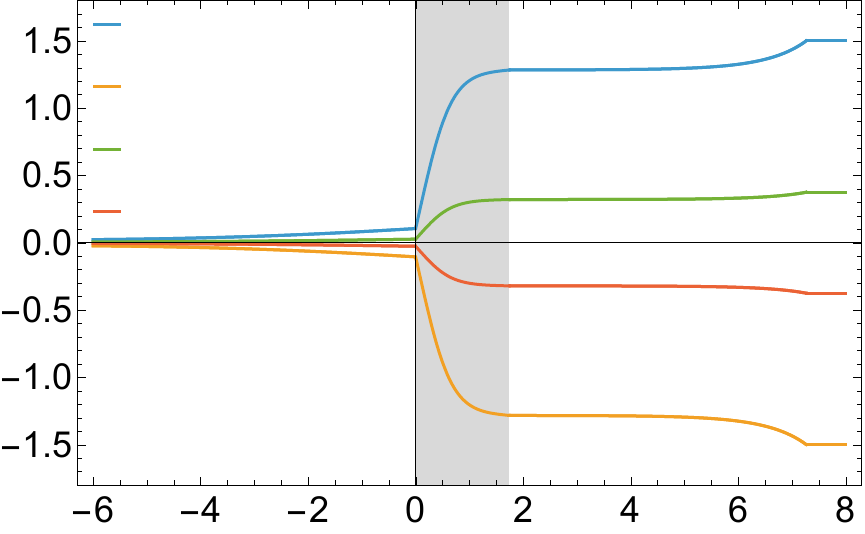}{fig1d}
    \hfill
    \barrierpanel{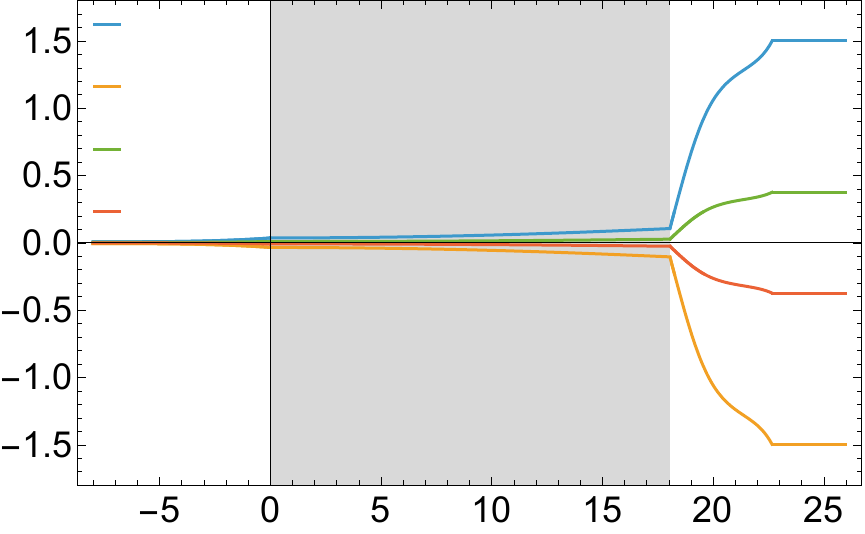}{fig1e}

    \caption{Barrier for a localized heterogeneity constructed in the
    proof of Theorem~\ref{thm_barrier_heterogeneous}. For each
    heterogeneity, the barrier follows the construction used in the
    proof. We consider $f_0(v)=v(1-v)(v-\alpha)$ with $\alpha=0.4$,
    $\gamma=4$, and $K=1.5$. The gray region indicates the
    heterogeneous region.}
    \label{fig_localized_heterogeneity}
\end{figure}

\begin{proof}

The goal is to construct a stationary barrier $(\bar v,\bar w,\underaccent{\bar}{v},\underaccent{\bar}{w})$ following the approach described in Lemma \ref{lem:symmetric_stationary_barriers} for each case of Theorem~\ref{thm_barrier_heterogeneous}. We will construct the required function $V$ in pieces, from left to right. Each piece of $V$ will be $C^2$, increasing, and will satisfy equation \eqref{system_trapping_reduced} on its corresponding interval. At each point where the pieces of $V$ are joined together, we will ensure that $V$ is continuous and that the transmission inequality \eqref{transmission_condition_symmetric} is satisfied.

If $K_0>0$ is arbitrarily prescribed, and $v_0\geq 0$ is given by condition \eqref{condition_H1}, we observe that \eqref{condition_H1}.(iii) remains valid if $v_0$ is increased; we may therefore assume \(v_0\geq k^*\), where \(k^*>0\) is the constant appearing in condition~\eqref{condition_H1}.(iv). We let $K=\max\{K_0,v_0\}$,
and the construction of $V$ will be such that it grows from a sufficiently small nonnegative value $V_-$ at $-\infty$, to the value $K$ at some point $x_0\in\R$. An important observation in this monotonically increasing construction is: as soon as $V(x_0)=K$ and $V'(x_0)\geq 0$ for some $x_0$, the construction can be completed by letting $V(x)=K$ for all $x\geq x_0$. This is true since with this choice the transmission inequality \eqref{transmission_condition_symmetric} is automatically satisfied at $x_0$ and at all points in $(x_0,\infty)$ where the coefficients are discontinuous, and additionally \eqref{system_trapping_reduced} is satisfied in all subintervals of $(x_0,\infty)$ where the coefficients are smooth, in all cases of Theorem \ref{thm_barrier_heterogeneous}.

We construct the profile $V$ in four stages: an initial tail, a climb through the heterogeneous region, and a monotonic continuation to the target height, with a final constant segment, noting that we can cut short the construction if we reach the desired height $K$ at any point. The generic structure of $V$ will be
\begin{equation}\label{eq:V_struct}
V(x) = \begin{cases}
    v_1(x) & , x<0 ,\\
    v_2(x) & , 0\leq x <L ,\\
    v_3(x) & , L\leq x < x_0 ,\\
    K  &, x\geq x_0 , 
\end{cases}
\end{equation}
where $[0,L]$ is the location of the heterogeneity, with $L$ to be specified in the construction. Let $h(v,x) := \min\{-f(v,x),f(-v,x)\}$, let $h_0(v) := \min\{-f_0(v),f_0(-v)\}$, and observe that $h_0$ is continuous and differentiable at $v=0$, with $h_0'(0) = -f'_0(0)$. The quotient $b(x)/a(x)$ is piecewise constant, we denote it by $\sigma_1$ for $x<0$, $\sigma_2$ for $x\in(0,L)$, and $\sigma_3$ for $x>L$.
Throughout the proof, we repeatedly use the existence and properties of solutions to a specific family of constrained ODEs, as stated in the following Lemma.

\begin{lemma}\label{lem:generic_ODE}
    Let $f(v,x)$ be continuous, locally Lipschitz in $v$ uniformly in $x$, and locally bounded in $v$ uniformly in $x$ in the sense \eqref{local_uniform_boundedness}. Let $\sigma, s_0>0$, let $x_0, k\in \R$, and let $R>k$. Then there exist $L\in(x_0,\infty]$ and a unique maximal constrained solution $v: [x_0,L)\to [k, R)$ satisfying
    \begin{empheq}[left=\empheqlbrace]{equation}\label{eq:generic_case_lemma}
    \begin{aligned}
    &\sigma v'' = f(v,x),\quad x\in (x_0,L),\\
    &v(x_0) = k, \quad v'(x_0) = s_0,\\
    &v(x)<R,\quad \text{ and }\quad v'(x)>0, \quad x\in (x_0,L).
    \end{aligned}
    \end{empheq}
Moreover, $v(L^-)$ and $v'(L^-)$ exist and are finite. If $L<\infty$, then either $v(L^-)=R$ or $v'(L^-)=0$. If $L=\infty$, then $v'(x)\to0$ as $x\to\infty$. If, additionally, $v'(x)\geq \delta$  for some $\delta>0$ and all $x\in(x_0,L)$, then $L$ is finite and $v(L^-) = R$. Also, $v$ satisfies
\begin{equation}\label{eq:generic_energy}
    \frac{1}{2}v'(x)^2 = \frac{1}{2}s_0^2  + \frac{1}{\sigma} \int_{k}^{v(x)} f(\omega,v^{-1}(\omega)) d\omega, \quad \forall x\in(x_0,L).
\end{equation}
\end{lemma}
\begin{proof}
    Since $R>k$, $s_0>0$, and $f(v,x)$ is continuous and locally Lipschitz in $v$, uniformly in $x$, equation \eqref{eq:generic_case_lemma} admits a unique maximal twice differentiable solution $v$ on an interval of the form $[x_0,L)$, with $L>x_0$. We have constrained $v$ to be strictly increasing and bounded above by $R$, so $v(L^-)$ exists. Multiplying the ODE in 
    \eqref{eq:generic_case_lemma} by $v'$ and integrating, for $x\in(x_0,L)$ we obtain
\begin{equation*}
    \frac{1}{2}v'(x)^2 - \frac{1}{2}v'(x_0)^2=  \frac{1}{\sigma} \int_{x_0}^{x} f(v(s),s)v'(s) ds.
\end{equation*}
Since $v'>0$ on $(x_0,L)$, the function $v$ is strictly increasing and therefore invertible on its image. The change of variables $\omega=v(s)$ gives \eqref{eq:generic_energy}. 

Since $v(L^-)$ exists and $f$ is uniformly bounded on the strip $k\leq v\leq R$, taking the limit $x\to L^-$ in \eqref{eq:generic_energy} shows that $v'(L^-)$ exists and is finite. Since $v(x)$ remains bounded, the uniform local boundedness of $f$ implies  $|v''(x)|= |f(v(x),x)|/\sigma \leq C$ for all $x\in(x_0,L)$. Thus $v'$ is uniformly Lipschitz continuous.

If $L<\infty$ and both $v(L^-)<R$ and $v'(L^-)>0$, the local existence theorem extends the solution beyond $L$, and continuity preserves $v<R$ and $v'>0$ on a sufficiently short right neighborhood of $L$. This contradicts the maximality of $L$. Therefore, either $v(L^-)=R$ or $v'(L^-)=0$.

If $L=\infty$, then we write
\begin{equation*}
    0\leq  \int_{x_0}^M v'(x) dx = v(M) - v(x_0) \leq v(L^-) - v(x_0)<\infty. 
\end{equation*}
Letting $M\to\infty$ shows that $v'\in L^1(x_0,\infty)$. Since $v'\geq 0$ and uniformly continuous, this implies $v'(x) \to 0$ as $x\to \infty$.

Finally, if $v'(x)\geq \delta>0$ on $x\in(x_0,L)$, then $v(x) \geq v(x_0) + \delta(x-x_0)$. Thus $L \leq x_0 + (R- v(x_0))/\delta<\infty$. Moreover, $v'(L^-) \geq \delta$, so the preceding alternative implies $v(L^-) = R$.
\end{proof}

\noindent
\textbf{Step 1:} The approach. In this step, we construct the tail $v_1(x), x<0$. This shows that we can climb from an arbitrarily small level $V_-\geq 0$ to a strictly (and structurally defined) positive height $k>0$ before hitting the heterogeneity. 

By condition~\eqref{condition_H1}.(iv), \( h_0(v)-\frac{v}{\gamma}>0\) for every \(v\in(0,k^*]\). In case~\ref{thm_heterogeneity_type1}, let \(\eta\in[0,k^*)\) be as in the hypothesis of that case and fix \(k\in(\eta,k^*)\). In the remaining cases, fix \(k\in(0,k^*)\). We then fix \(\bar V\in(0,k)\), depending only on the data of the case, and choose \(V_-\in[0,\bar V]\).
Let $x_{\text{-}\infty}<0$ and $v_1:(x_{\text{-}\infty},0]\to(V_-,k]$ be the maximal solution of the initial value problem
\begin{empheq}[left=\empheqlbrace]{equation*}
\begin{aligned}
&\sigma_1 v_1'' = h_0(v_1) -v_1/\gamma,\quad x\in (x_{\text{-}\infty},0),\\
&v_1(0) = k, \quad v_1'(0) = \sqrt{2} \biggl(\frac{1}{\sigma_1}\int_{V_-}^k(h_0(v)-v/\gamma)dv\biggr)^{1/2},\\
&V_- <  v_1(x)< k,\quad \text{ and }\quad v_1'(x)>0, \quad \text{ for } x\in (x_{\text{-}\infty},0).
\end{aligned}
\end{empheq}
Since $h_0(v)-v/\gamma>0$ on $(0,k)$, the slope $v_1'(0)$ is nonincreasing in $V_-$, and hence
\[
v_1'(0)\geq s_*:=\biggl(\frac{2}{\sigma_1}\int_{\bar V}^{k}\Bigl(h_0(v)-\frac{v}{\gamma}\Bigr)dv\biggr)^{1/2}>0
\qquad\text{for every } V_-\in[0,\bar V].
\]
In Step~2, the thresholds on $p$ depend on $v_1'(0)$ only through this lower bound, so they can be chosen uniformly in $V_-\in[0,\bar V]$.

The existence of $x_{\text{-}\infty}<0$ and $v_1$ follows from the analogous backwards version of Lemma \ref{lem:generic_ODE}. The lemma also implies that $v_1$ is strictly increasing, that $v_1(x_{\text{-}\infty}^+)$ and $v_1'(x_{\text{-}\infty}^+)$ exist, and that either $x_{\text{-}\infty}=-\infty$, $v_1(x_{\text{-}\infty}^+)= V_-$ or $v_1'(x_{\text{-}\infty}^+)=0$. For $x\in (x_{\text{-}\infty},0)$ equation \eqref{eq:generic_energy} implies
\begin{equation*}
\frac{1}{2}v_1'(x)^2 = \frac{1}{2}v_1'(0)^2  -  \frac{1}{\sigma_1} \int_{v_1(x)}^{k} (h_0(v)-v/\gamma) dv,
\end{equation*}
and for the particular choice of $v_1'(0)$ this means 
\begin{equation}
\begin{split}
\frac{1}{2}v_1'(x)^2
    &=\frac{1}{\sigma_1}\int_{V_-}^k (h_0(v)-v/\gamma)dv - \frac{1}{\sigma_1} \int_{v_1(x)}^{k} (h_0(v)-v/\gamma) dv\\
&= \frac{1}{\sigma_1}\int_{V_-}^{v_1(x)} ( h_0(v)-v/\gamma)dv.\label{eq:energy_step1}
\end{split}
\end{equation}
Since $(h_0(v)-v/\gamma)>0$ for $v\in(V_-,k)$, equation \eqref{eq:energy_step1} implies that $v_1'(x_{\text{-}\infty}^+)=0$
if and only if $v_1(x_{\text{-}\infty}^+)=V_-$. Since $v_1'(x_{\text{-}\infty}^+)$ is the minimum of $v_1'(x), x\in(x_{\text{-}\infty},0)$, Lemma~\ref{lem:generic_ODE} implies that  $v_1'(x_{\text{-}\infty}^+)=0$
and $v_1(x_{\text{-}\infty}^+)=V_-$ is the only possibility.

If $V_->0$, then $v_1''=(h_0(v)-v/\gamma)/\sigma_1>\delta$, for all $v\in(V_-,k)$, for some $\delta>0$. This means that $v_1'(x)$ is strictly increasing at a rate of at least $\delta$. Since $x_{\text{-}\infty}$ is characterized by $v_1'(x_{\text{-}\infty}^+)=0$, this implies that $x_{\text{-}\infty}>-\infty$. Since $(h_0(V_{-})-V_{-}/\gamma)>0$, extending the profile constantly by $V_-$ to the left of $x_{\text{-}\infty}$ preserves the differential inequality \eqref{system_trapping_reduced}.

When $V_-=0$, condition \eqref{condition_H1} implies that there exist $\delta_2>\delta_1>0$ such that $\delta_1 v\leq (h_0(v) -v/\gamma)/\sigma_1\leq \delta_2 v$ for all $v\in [0,k]$. Since $v_1$ and $v_1'$ are positive, equation \eqref{eq:energy_step1} implies that 
\begin{equation*}
    \sqrt{\delta_1}v_1(x)\leq v_1'(x)\leq \sqrt{\delta_2} v_1(x) \quad \Rightarrow \quad v_1(0) e^{\sqrt{\delta_2} x} \leq v_1(x) \leq  v_1(0) e^{\sqrt{\delta_1} x}, \quad x\leq 0.
\end{equation*}
These estimates show that $v_1$ cannot reach zero at a finite left endpoint. Hence $v_1(x_{\text{-}\infty}^+)=v_1'(x_{\text{-}\infty}^+)=0$, and $v_1(x)$ decays to 0 exponentially as $x\to-\infty$. 


\noindent
\textbf{Step 2:} The climb. 
The presence of the heterogeneous region will allow us to construct a profile $V$ with the structure \eqref{eq:V_struct} that either climbs to the level $v_0$, or that exits the heterogeneity region $(0,L)$ with a sufficiently large slope. More precisely, we separate the five cases of Theorem~\ref{thm_barrier_heterogeneous} into two classes. For the first class, if $p>0$ is sufficiently large, we construct $L=L(p)>0$ and $v_2(x), x\in[0,L)$ such that 
$V(L^+)\geq v_0$, $V'(L^+)>0$. For the second class, given any $M>0$, we show that if $p>0$ is large enough, then we can construct $V$ with $V'(L^+)>M/\sqrt{\sigma_3}$.
\begin{enumerate}
\item[\ref{thm_heterogeneity_type1}] Additive heterogeneity distributed over the interval $[0,L]$. This case belongs to the second class. Let $k$ be as in Step 1 and $k^*$ as in \eqref{condition_H1}.(iv), so that \(\eta<k<k^*\). Choose $\delta \in (0,\min\{1, k^*-k\})$, so $h_0(v) - v/\gamma >0$ for $v\in [k,k+\delta]$. Since \(v>\eta\) and \(-v<-\eta\) on this interval, the sign condition on \(g\) gives \( -g(v)>0\) and \(g(-v)>0\).
Consequently, 
\[ 
\kappa:= \frac{1}{\sigma_2} \min_{v\in[k,k+\delta]} \{-g(v),g(-v)\}>0. 
\]
Given $p\geq 0$ let $L> 0$ and $v_2$ correspond to the maximal solution of
\begin{empheq}[left=\empheqlbrace]{equation}\label{equation_climb2}
\begin{aligned}
&\sigma_2 v_2'' =h_0(v_2) -v_2/\gamma + p \min\{- g(v_2), g(-v_2)\},\quad x\in (0,L),\\
&v_2(0) = v_1(0) , \quad v'_2(0) = v'_1(0)\\
&v_2< k+\delta, \quad \text{ and } \quad v_2'(x)>0, \quad \text{ for } x\in (0,L).
\end{aligned}
\end{empheq}
Here we used that $h(v,x)\geq h_0(v)+p\min\{-g(v),g(-v)\}$ on $[0,L]$, so a solution of \eqref{equation_climb2} satisfies the required differential inequality. Since $v_2''(x)\geq p\kappa>0$ for all $x\in(0,L)$, $v_2'$ is increasing and $v_2'(x)\geq v_2'(0)>0$ for all $x\in(0,L)$. Lemma~\ref{lem:generic_ODE} then implies that $L$ is finite and $v_2(L^-)=k+\delta$. Cauchy's mean value theorem then implies that, for some $x\in(0,L)$,
\[
\frac{v_2'(L^-)}{\delta}\geq \frac{v_2'(L^-)-v_2'(0)}{v_2(L^-)-v_2(0)}=\frac{v_2''(x)}{v_2'(x)}\geq \frac{\kappa p}{v_2'(L^-)}.
\]
Therefore, $v_2'(L^-) \geq \sqrt{\kappa p \delta}$. Given $M>0$, since \(\sigma_3\) is fixed in this case, for all sufficiently large $p>0$ the corresponding $L$ and $v_2$ satisfy $v_2'(L^-)>M/\sqrt{\sigma_3}$. Since $b(x)$ is constant, we can set $V(L^+)=v_3(L^+)=v_2(L^-)$ and $V'(L^+)=v'_3(L^+)=v'_2(L^-)$, ensuring that \(V\) is continuous, satisfies \eqref{transmission_condition_symmetric} at \(x=L\), and has \(V'(L^+)>M/\sqrt{\sigma_3}\).

For the affine choice \(g(v)=E_{\mathrm{rev}}-v\), one has \(\eta=\lvert E_{\mathrm{rev}}\rvert\) and \(\min\{-g(v),g(-v)\}=v-\lvert E_{\mathrm{rev}}\rvert>0\) on \([k,k+\delta]\).

\item[\ref{thm_heterogeneity_type2}] This case describes a segment with depressed dynamics and belongs to the first class.
Since  $h(v,x) - v/\gamma\geq 0$ if $x\in [0,L]$ and $v \geq 0$, 
the function $v_2(x)=v_1(0)+xv_1'(0), x\in[0,L]$ satisfies inequality \eqref{system_trapping_reduced}. Set $p=L$. If $p\geq p^*=(v_0-k)/s_*$, 
then $v_2(L)\geq v_0$ and $v_2'(L)=v_1'(0)>0$. Since $b(x)$ is constant in this case, we set $V(L^+)=v_2(L)$ and $V'(L^+)=v_1'(0)>0$. This completes the Step~2 construction with the properties required for the first class.

\item[\ref{thm_heterogeneity_type3}] This is a case of an abrupt change in the geometry from a given point onward and belongs to the second class. In this case $L=0$, there is no function $v_2$, and the construction of $V$ goes straight from $v_1$ to $v_3$. Given the discontinuity in $b(x)$ at $x=0$ the transmission inequality \eqref{transmission_condition_symmetric} at that point is
\begin{equation*}
r_0^2v_1'(0^-) \geq \biggl(\frac{r_0}{1+p}\biggr)^2 v_3'(0^+),
\end{equation*}
and for the construction of $v_3$ we impose equality, hence $V$ will
be constructed with 
\begin{equation*}
V'(0^+)=v_3'(0^+)=(1+p)^2 v_1'(0^-).
\end{equation*}
In this case, \(\sigma_3=r_0/(1+p)\). Given $M>0$, since $v_1'(0^-)\geq s_*>0$, the identity above implies that \(V'(0^+)>M/\sqrt{\sigma_3}\) for all sufficiently large \(p\). We also impose \(V(0^+)=v_1(0^-)\).

\item[\ref{thm_heterogeneity_type4}] This is also the case of an abrupt change in the geometry, but only during a compact interval, returning afterward to the original geometry. This case falls under the first class. At $x=0$ there is a discontinuity in $b(x)$. If we impose equality in the transmission inequality \eqref{transmission_condition_symmetric}, then $v_2$ must satisfy $v_2(0)=v_1(0)=k$ and $v_2'(0^+) = (1+p)^2 v_1'(0^-)$. Since $k<k^*\leq v_0\leq K$, we have $v_1(0)=k<K$, and we let $L$ and $v_2$ correspond to the maximal solution of
\begin{empheq}[left=\empheqlbrace]{equation*}
\begin{aligned}
&\sigma_2 v_2'' = h_0(v_2) -v_2/\gamma , &&\text{for } x\in(0,L),\\
&v_2(0) = v_1(0), \quad v_2'(0) = (1+p)^2 v_1'(0),&& \\
&v_2(x)<v_0,\quad \text{ and } \quad v_2'(x)>0,  &&\text{for } x\in(0,L).
\end{aligned}
\end{empheq}
From Lemma~\ref{lem:generic_ODE}, for $x\in(0,L)$,
\begin{equation}\label{formula_derivative_case4}
\begin{split}
\frac{1}{2}v_2'(x)^2 &=\frac{1}{2}v_2'(0)^2  + \frac{1}{\sigma_2}\int_0^{x} (h_0(v)-v/\gamma) v' dx \\
&= \frac{1}{2}(1+p)^4v_1'(0)^2 + \frac{1+p}{r_0}\int_{v_2(0)}^{v_2(x)} (h_0(v)-v/\gamma) dv.
\end{split}
\end{equation}
Condition \eqref{condition_H1} implies that the last integral term in \eqref{formula_derivative_case4} admits a uniform lower bound, namely 
\begin{equation*}
\int_{v_2(0)}^{v_2(x)} (h_0(v)-v/\gamma) dv \geq \inf_{0\leq z_1\leq z_2\leq v_0} \int_{z_1}^{z_2} (h_0(v)-v/\gamma) dv.
\end{equation*}
Here we used that any negative contribution to the integral is confined to $[0,v_0]$. 
Since $v'_1(0) \geq s_* >0$, equation \eqref{formula_derivative_case4} then implies that, for $p\geq p^*$, one has $v_2'(x) \geq 1$ for every $x\in(0,L)$, where $p^*$ can be chosen independent of $K_0$. Lemma~\ref{lem:generic_ODE} then implies that $v_2(L^-)=v_0$, while $v_2'(L^-)\geq 1>0$. For the construction of $V$, this means that $V(L^+)=v_0$, and imposing equality in the transmission inequality $\eqref{transmission_condition_symmetric}$ constructs $V$ with $V'(L^+)= v_2'(L^-)/(1+p)^2>0$, as required.

\item[\ref{thm_heterogeneity_type5}] This is a case of an abrupt change in the geometry on a bounded interval, in the alternate direction of case \ref{thm_heterogeneity_type4}. This case falls in the second class. The difference from the previous case is that the critical change here is a return to the original conductivity, rather than the initial change.

Let $\tilde{k}\in(k,k^*)$, where $k$ is as in Step 1 and $k^*$ as in \eqref{condition_H1}.(iv). Then $(h_0(v)-v/\gamma)>0$ for all $v\in[k,\tilde{k}]$. At $x=0$ there is a discontinuity in $b(x)$, if we impose equality in the transmission inequality \eqref{transmission_condition_symmetric} we require $v_2$ to satisfy $v_2(0)=v_1(0)=k$ and $v_2'(0^+) = (1+p)^{-2} v_1'(0^-)>0$. Since $k<k^*\leq v_0\leq K$, we have $v_1(0)=k<K$, and we let $L$ and $v_2$ correspond to the maximal solution of
\begin{empheq}[left=\empheqlbrace]{equation*}
\begin{aligned}
&\sigma_2 v_2'' = h_0(v_2) -v_2/\gamma, &&\text{for } x\in(0,L),\\
&v_2(0) = v_1(0), \quad v_2'(0) = (1+p)^{-2} v_1'(0), &&\\
&v_2(x)< \tilde{k},\quad \text{ and } \quad v_2'(x)>0, &&\text{for } x\in(0,L).
\end{aligned}
\end{empheq}
From Lemma~\ref{lem:generic_ODE},
\begin{equation}\label{formula_derivative_at_L_case6}
\begin{split}
\frac{1}{2}v_2'(x)^2 &= \frac{1}{2}v_2'(0)^2  + \frac{1}{\sigma_2}\int_0^{x} (h_0(v)-v/\gamma) v' dx\\
&= \frac{1}{2}\biggl(\frac{v_1'(0)}{(1+p)^2}\biggr)^2  +  \frac{1}{(1+p)r_0}\int_{k}^{v_2(x)} (h_0(v)-v/\gamma) dv.
\end{split}
\end{equation}
In particular $v_2'(x)\geq v'_2(0)>0$ for $x\in(0,L)$. Hence, Lemma~\ref{lem:generic_ODE} implies that $L$ is finite and $v_2(L^-)=\tilde{k}$. 
At $x=L$ there is another discontinuity in $b(x)$. Imposing equality in the transmission inequality \eqref{transmission_condition_symmetric} we require $v_3$ to satisfy $v_3(L)= v_2(L^-)$, and $v_3'(L^+) = (1+p)^2v_2'(L^-)$. Combined with \eqref{formula_derivative_at_L_case6} this implies that
\begin{equation*}
V'(L^+)=v_3'(L^+) = \biggl(v'_1(0)^2  + \frac{2(1+p)^3}{r_0} \int_{k}^{\tilde{k}} (h_0(v)-v/\gamma) dv\biggr)^{1/2}.
\end{equation*}
Given \(M>0\), since \(\sigma_3\) is fixed in this case, this formula shows that \(V'(L^+)>M/\sqrt{\sigma_3}\) for all sufficiently large \(p\). This profile is continuous at \(L\) because \(V(L^+)=v_2(L^-)\).
\end{enumerate}

\noindent
\textbf{Step 3:} The summit. At this stage, the construction of $V$ falls into one of three scenarios. Scenario 1: while constructing the pieces $v_1$ or $v_2$ we already reached $V(x_0)=K$ with $V'(x_0)\geq 0$, for some point $x_0$, which means that the construction of $V$ is completed by letting $V(x)=K$ for $x\geq x_0$. Scenario 2: we have constructed a valid increasing profile $V(x)$, for $x\in(-\infty,L]$, with $v_0\leq V(L^+)<K$ and $V'(L^+)>0$.
Scenario 3: we have constructed a valid increasing profile $V(x)$, for $x\in(-\infty,L]$, with $V(L^+)<K$ and \(V'(L^+)>M/\sqrt{\sigma_3}\), where \(M>0\) can be chosen arbitrarily large. 
Scenario 2 can arise from the two cases belonging to the first class in Step 2.
Scenario 3 can arise from the three cases belonging to the second class in Step 2. In Step 2 the values of $v_3(L^+)$ and $v_3'(L^+)$ have been prescribed in the construction of $V$.
We let $x_0$ and $v_3$ be given by the maximal solution of the following initial value problem with the prescribed data: 
\begin{empheq}[left=\empheqlbrace]{equation*}
\begin{aligned}
&\sigma_3 v_3'' = h_0(v_3) -v_3/\gamma,\quad &&\text{for } x\in(L,x_0),\\
&v_3(L) = V(L^+), \quad v_3'(L) = V'(L^+), &&\\
&v_3(x)< K,\quad \text{ and }\quad  v_3'(x)>0,  &&\text{ for } x\in(L,x_0).
\end{aligned}
\end{empheq}
Lemma~\ref{lem:generic_ODE} implies that, for $x\in(L,x_0)$,
\begin{equation}\label{eq:energy_Step3}
\frac{1}{2}v_3'(x)^2 = \frac{1}{2}V'(L^+)^2  + \frac{1}{\sigma_3} \int_{v_3(L)}^{v_3(x)} (h_0(v)-v/\gamma) dv.
\end{equation}
In Scenario 2, because $v_3(L)\geq v_0$, condition \eqref{condition_H1}.(iii) implies that the integral is non-negative, hence $v_3'(x)\geq V'(L^+)>0$ for all $x\in(L,x_0)$. Lemma~\ref{lem:generic_ODE} then implies that $v_3(x_0^-)=K$.

In Scenario 3, condition \eqref{condition_H1} implies that the integral term in \eqref{eq:energy_Step3} is uniformly bounded below, namely
\begin{equation*}
   \int_{v_3(L)}^{v_3(x)} (h_0(v)-v/\gamma)\, dv \geq I_* := \inf_{0\leq z_1\leq z_2\leq v_0} \int_{z_1}^{z_2} (h_0(v)-v/\gamma)\, dv,
\end{equation*}
where $I_*\leq0$. Since the construction allows \(M\) in \(V'(L^+)>M/\sqrt{\sigma_3}\) to be chosen arbitrarily large, we fix any \(M\) with \(M^2>2\lvert I_*\rvert\), independently of \(p\). Then, until \(v_3\) reaches \(v_0\),
\[
   v_3'(x)^2\geq \frac{M^2-2\lvert I_*\rvert}{\sigma_3}=:\delta_p^2>0,
\]
where \(\delta_p\) may depend on \(p\) through \(\sigma_3\).
Condition~\eqref{condition_H1}.(iii) then implies that \(v_3'\) remains positive. Lemma~\ref{lem:generic_ODE} therefore implies that \(x_0\) is finite and \(v_3(x_0^-)=K\).

\noindent
\textbf{Step 4:} The conclusion. In summary, in all the cases of Theorem \ref{thm_barrier_heterogeneous} we are able to construct a valid increasing profile $V$ that eventually reaches $V(x_0)=K$ at some point $x_0$, during the construction of $v_1$, $v_2$ or $v_3$. 
Since $K\geq v_0$, the constant extension of $V$ to the right of $x_0$ satisfies \eqref{system_trapping_reduced} by  condition \eqref{condition_H1} and, in cases \ref{thm_heterogeneity_type1} and \ref{thm_heterogeneity_type2}, by the corresponding hypothesis on $f(v,x)$. 
Moreover, $[bV']_{x_0}=-b(x_0^-)V'(x_0^-)\leq0$, so the transmission inequality holds at $x_0$. 
The thresholds used in the five cases depend only on the fixed data of the corresponding case and are independent of $K_0$ and of $V_-\in[0,\bar V]$.
At each step, the construction ensures that $V$ has all the properties required in Lemma \ref{lem:symmetric_stationary_barriers} and therefore defines a barrier as required in Theorem~\ref{thm_barrier_heterogeneous}.

\end{proof}

\section{Failure of propagation in the double-diffusive case: proof of Theorem \ref{thm_double_diffusive}}\label{section_proof_double_diffusive}

\begin{proof}
     We are assuming that $(V,V/\tilde{\gamma},-V,-V/\tilde{\gamma})$, with $0\leq V \leq K$, is a barrier of system \eqref{FHN_pde_general} with $D=0$ and $\tilde{\gamma}$ in place of $\gamma$. By Lemma \ref{lem:symmetric_stationary_barriers}, this assumption implies that $V$ satisfies equation \eqref{system_trapping_reduced}, with $\tilde{\gamma}$ in place of $\gamma$, namely
     \begin{equation}\label{eq:V_tildegamma}
     \frac{1}{a}\partial_x(b\partial_x  V) \leq \min \{-f(V,x),f(-V,x)\} -\frac{V}{\tilde{\gamma}},
    \end{equation}
    Moreover, $V$ satisfies $[b\partial_x V]_\xi \leq 0$ at every interior interface $\xi$.
    Let $W:=V/\tilde{\gamma}$. 
    Since $b=b_0$ is constant and $b_2\equiv1$ when $D=\delta>0$ (see Section~\ref{section_definitions_and_main_results}), the transmission inequality for $V$ implies
    \begin{equation*}
    [b\partial_x V]_\xi \leq 0, \quad \text{ and } \quad [\partial_x W]_\xi \leq 0, 
    \end{equation*}
     at every interior interface $\xi$ (defined by $V$ and $W$). Let \(L_K\) be a Lipschitz constant for \(f(\cdot,x)\) on \([-K,K]\), chosen uniformly in \(x\), and set \(C:=\max\{L_K,1\}\). Since \(f(0,x)=0\) and \(0\leq V\leq K\), it follows that
     \begin{equation*}
     \bigg(\min\{f(-V,x),-f(V,x)\} - \frac{V}{\tilde{\gamma}}\bigg)_+\leq C V \qquad  \text{for all }x\in\R.
     \end{equation*}
     Set $d^*:=(\gamma-\tilde{\gamma})/C$. If \(\delta(x)a(x)/b_0 \leq d^*\), then equation \eqref{eq:V_tildegamma} implies (recall $V\geq 0$)
     \begin{equation*}
     \begin{aligned}
             \delta(x)\partial_x^2W&=\frac{\delta(x)a(x)}{b_0\tilde{\gamma}}\frac{1}{a(x)}\partial_x( b_0\partial_x V) \\
             &\leq \frac{\gamma-\tilde{\gamma}}{C\tilde{\gamma}}\biggl( \min\{f(-V,x),-f(V,x)\} - \frac{V}{\tilde{\gamma}} \biggr)_+ \leq (\gamma-\tilde{\gamma})\frac{V}{\tilde{\gamma}}=-V+\gamma W.
     \end{aligned}
     \end{equation*}
     In summary, we have shown that $(V,W)$ satisfies
     \begin{equation*}
         \begin{split}
             &\frac{1}{a(x)}\partial_x(b_0\partial_x  V) \leq \min \{-f(V,x),f(-V,x)\} -W,\\
             &\delta(x)\partial_x^2W\leq -V+\gamma W,
         \end{split}
     \end{equation*}
     with transmission inequalities 
    \begin{equation*}
    [b_0\partial_x V]_\xi \leq 0, \quad \text{ and } \quad [\partial_x W]_\xi \leq 0, 
    \end{equation*}
at every interior interface $\xi$ in a common refinement of the partitions associated with $a$, $\delta$, and $V$. Analogously to Lemma 
\ref{lem:symmetric_stationary_barriers}, the preceding inequalities, together with the corresponding symmetric lower inequalities and the transmission inequalities, show directly that \( (V,V/\tilde{\gamma},-V,-V/\tilde{\gamma}) \) is a stationary barrier for system \eqref{FHN_pde_general}.
\end{proof}

\section{Almost stationary barriers and bounds on propagation}\label{section_moving_barrier}

In this section, we construct time-dependent upper and lower solutions that control the expansion of suitably trapped solutions. After suitable translations and reflections, these barriers bound the threshold-exceedance regions of solutions whose initial data lie between the barrier components and that satisfy the hypotheses of the comparison principle. 

We consider system \eqref{FHN_pde_general} with $D=0$ and with $a$ and $b$ constant and positive. 
Let $\varepsilon=\tau^{-1}$, fix $c>0$, and introduce the normalized moving coordinate 
\( \xi=\sqrt{a/b} x+ct. \) 
We seek symmetric barriers of the form 
\( \bar v(x,t)=v(\xi), \underaccent{\bar}{v}(x,t)=-v(\xi),\bar w(x,t)=w(\xi), \underaccent{\bar}{w}(x,t)=-w(\xi). \)
Writing $u=v'$, where the prime denotes differentiation with respect to $\xi$, the barrier inequalities \eqref{system_sub_super_solutions} reduce to 
\begin{empheq}[left=\empheqlbrace]{equation*}
\begin{aligned}
&v'=u,\\ 
&u'\leq \min\{-f(v),f(-v)\}+cu-w,\\ 
&w'\geq \dfrac{\varepsilon}{c}(v-\gamma w). 
\end{aligned} 
\end{empheq}
Indeed, \( \partial_t\bar v=cv', \partial_x\bar v=\sqrt{a/b} v', \frac{1}{a}\partial_x(b\partial_x\bar v)=v'',  \partial_t\bar w=cw'. \) 
The upper activator and recovery inequalities are therefore equivalent to the second and third inequalities above, and the corresponding lower inequalities follow by symmetry. 

If $v'$ has a downward jump at $\xi=q$, the corresponding moving interface in the original coordinate is 
\( x_q(t)=\sqrt{b/a}\bigl(q-ct\bigr). \)
Since \( b \partial_x\bar v=\sqrt{ab} v'\), the condition $[v']_q\leq0$ implies the required transmission inequality \( [b \partial_x\bar v]_{x_q(t)}\leq0. \) 

The corresponding lower-barrier inequality follows by symmetry. No transmission inequality is required for the recovery components because $b_2\equiv 0$. Finally, the level sets of $\xi$ move to the left with physical speed $\sqrt{b/a} c$.

\subsection{Proof of Theorem \ref{thm_barriers_conduction_block}}
\begin{proof}
The theorem considers the case $f(v,x) = v(1-v) (v-\alpha)$ with $\alpha\in(1/2,1)$. Since $f(-v)+f(v) = 2(1+\alpha) v^2$, $\min\{-f(v),f(-v)\} = -f(v)$ for $v\geq 0$. Our construction aims for a positive $v$; hence, the first-order system for the construction of the barrier becomes
\begin{empheq}[left=\empheqlbrace]{equation}\label{eq:uvw-first-order}
\begin{aligned}
     &v' = u,\\
     &u' \leq -f(v)+cu - w, \\
     &w' \geq \dfrac{\varepsilon}{c} (v-\gamma w).
\end{aligned}
\end{empheq}
with transmission inequality $[v']_q \leq 0$ whenever $v'$ is discontinuous at $\xi = q$. Throughout the construction, we frequently consider this system with equalities instead of inequalities, which we refer to as \eqref{eq:uvw-first-order}$^=$. As in the proof of Theorem~\ref{thm_barrier_heterogeneous}, we construct the monotonically increasing barrier from left to right in stages. In the first stage, we construct a tail for $u,v$ and $w$ that decays exponentially to the left, and reaches a uniformly strictly positive height to the right. Unlike in the stationary case, we do not have an explicit formula for a condition at $\xi = 0$ that yields a direct construction; hence, we need a non-explicit and slightly more involved approach to construct the tail. In the second stage, the climb, we show that the properties of the constructed tail ensure that, for small enough $\varepsilon$, $v$ can climb to the prescribed large value while $w$ remains small. In the final stage, the summit, we finish constructing $v$ and $w$ so they climb and plateau above the required levels. Along the entire proof, $c = c(\varepsilon)$.

\noindent
\textbf{Step 1:} The approach. We use a linearization argument to construct a solution of \eqref{eq:uvw-first-order}$^=$ that approaches $(v,u,w)=(0,0,0)$ as $\xi\to-\infty$ and grows as $\xi$ increases, with $v$ reaching a positive value bounded away from $0$, uniformly in $\varepsilon$. Denote $X= (v,u,w)$. Linearization around $X=(0,0,0)$, leads to the system $\frac{d}{d\xi}X=A~X$ with
\begin{equation*}
A=
\begin{pmatrix}
0 & 1 & 0\\[6pt]
\alpha & c & -1\\[8pt]
\dfrac{\varepsilon}{c} & 0 & -\dfrac{\varepsilon\gamma}{c}
\end{pmatrix},
\end{equation*}
with characteristic equation $\det(A-\lambda I)=
(\lambda^2-c\lambda-\alpha) (\lambda+\frac{\varepsilon\gamma}{c})+\frac{\varepsilon}{c}=0$. 
Set $\nu(\varepsilon)=\varepsilon/c(\varepsilon)$, which goes to 0 as $\varepsilon\to 0$. The eigenvalues of $A$ are represented implicitly via the equation $P(\lambda, c , \nu) = (\lambda^2-c \lambda-\alpha) (\lambda+\nu\gamma)+\nu = 0$.
Let $\lambda_k(c, \nu), k=1,2,3$, be the three eigenvalues of $A$ in decreasing order. Because $\alpha >0$, the eigenvalues $\lambda_k(0,0)$ are real-valued and simple, and the implicit function theorem implies that we can solve uniquely for the eigenvalues around $(\lambda, c,\nu) = (\lambda_k(0,0), 0,0)$. A Taylor expansion in $\nu$ provides the following asymptotics, which is uniform in $\lvert c\rvert \leq \sqrt{\alpha}$,
\begin{align*}
\lambda_1(c,\nu) &=
\frac{c+\sqrt{c^2+4\alpha}}{2}
- \nu \frac{2}{\sqrt{c^2+4\alpha} \bigl(\sqrt{c^2+4\alpha}+c\bigr)} +O\Bigl(\nu^2\Bigr),\\ 
\lambda_2(c,\nu) &=
\nu \frac{1-\alpha\gamma}{\alpha }+O\Bigl(\nu^2\Bigr),\\ 
\lambda_3(c,\nu) &=
\frac{c-\sqrt{c^2+4\alpha}}{2}
- \nu\frac{2}{\sqrt{c^2+4\alpha} \bigl(\sqrt{c^2+4\alpha}-c\bigr)} +O\Bigl(\nu^2\Bigr).
\end{align*}
Since we are interested in solutions that decay exponentially as $\xi\to-\infty$, we need to approach $(0,0,0)$ in the direction of an eigenvector corresponding to a positive eigenvalue. The sign of the small eigenvalue $\lambda_2$ depends on $1-\alpha\gamma$, but this eigenvalue is not used to construct the tail; the construction below uses the positive eigenvalue $\lambda_1$.

\begin{remark}
    The appropriate stationary comparison comes from the independently derived $D=0$ scalar barrier equation. Its linearization at the rest state yields \( \lambda^2 = (-f_0'(0)-1/\gamma)/\sigma\), where $\sigma=b/a$. Hence, a positive real spatial exponent in the stationary construction requires \( -f_0'(0)>1/\gamma\). For the cubic $f_0(v)=v(1-v)(v-\alpha)$, this becomes $\alpha>1/\gamma$. By contrast, the moving construction uses the positive eigenvalue bifurcating from $\sqrt{\alpha}$ as $(c,\varepsilon/c)\to(0,0)$ and therefore does not require this stationary damping condition.
\end{remark}

Define $\lambda_k(\varepsilon):=\lambda_k(c(\varepsilon),\nu(\varepsilon))$,
$k=1,2,3$. Since $(c(\varepsilon),\nu(\varepsilon))\to(0,0)$ as
$\varepsilon\to0$, there exists $\varepsilon_0>0$ such that
$\lambda_1(\varepsilon)>\lambda_2(\varepsilon)>\lambda_3(\varepsilon)$ are
real and $\lambda_1(\varepsilon)>0$ for every
$\varepsilon\in(0,\varepsilon_0]$. We assume from now on that
$0<\varepsilon\leq\varepsilon_0$. For the remainder of the proof, write
$\kappa=\kappa(\varepsilon):=\lambda_1(\varepsilon)>0$. To study the
linearized system at $\xi=-\infty$, we introduce the change of variables
\begin{equation}\label{eq:change-vars}
s=e^{\kappa \xi},\qquad v(\xi)=s V(s),\qquad u(\xi)=s U(s),\qquad w(\xi)=s W(s).
\end{equation}
Then $\dfrac{ds}{d\xi}=\kappa s$, so $\dfrac{d}{d\xi}=\kappa s \dfrac{d}{ds}$, and $\dfrac{d}{d\xi}u=\kappa s \frac{d}{ds}(s U) = \kappa s(U+sU')$. Similarly for $v$ and $w$.
With this change of variable, and dividing by $\kappa s$, equation \eqref{eq:uvw-first-order}$^=$ is transformed into the following equation for $s>0$,
\begin{equation}\label{eq:UVW-system}
\begin{cases}
s \frac{d}{ds}V = -V+\dfrac{1}{\kappa} U,\\[8pt]
s \frac{d}{ds}U = -\dfrac{1}{\kappa} V(1-sV)(sV-\alpha) + \Bigl(\dfrac{c}{\kappa}-1\Bigr)U
-\dfrac{1}{\kappa} W,\\[10pt]
s \frac{d}{ds}W = \dfrac{\varepsilon}{c\kappa} (V-\gamma W) -W.
\end{cases}
\end{equation}
Since $s=0$ corresponds to $\xi=-\infty$, a bounded solution of \eqref{eq:UVW-system} in a neighborhood of $s=0$, corresponds to a solution of \eqref{eq:uvw-first-order}$^=$ that decays as $e^{\kappa\xi}$ when $\xi\to -\infty$.
Since
\[
V(1-sV)(sV-\alpha)=-\alpha V+s(1+\alpha)V^2-s^2V^3,
\]
letting $y=(V,U,W)$, we can write the system \eqref{eq:UVW-system} as
\begin{equation}\label{eq:syMy}
s y'(s)=\mathcal{M} y(s)+s \mathcal{N}(s,y(s)),
\end{equation}
with $\mathcal{M}=\kappa^{-1} A-I$ and $\mathcal{N}(s,y) = \bigl(0,-\kappa^{-1}\bigl((1+\alpha)V^2-sV^3\bigr) , 0\bigr)$. Observe that both $A$ and $\mathcal{M}$ depend on $\varepsilon$ and $c(\varepsilon)$. We must keep this in mind, since we later require estimates uniform in $\varepsilon$. 
The eigenvalues $\mu_i$ of $\mathcal{M}$ and the eigenvalues $\lambda_i$ of $A$ are related by $\mu_i=\lambda_i/\kappa-1$. Since $\kappa=\lambda_1>\lambda_2>\lambda_3$, we have $\mu_1=0$ and $\mu_2,\mu_3<0$. A direct computation gives 
\begin{equation*}
    \textsf{Ker}(\mathcal{M}) = \{\theta y_* \mid \theta\in  \R\}, \quad\text{with }y_*=
\biggl(1, \kappa, \frac{\varepsilon}{c\kappa+\varepsilon\gamma} \biggr).
\end{equation*}
The next lemma and corollary establish an existence result for \eqref{eq:syMy}, and then provide estimates on the solution and interval of existence that are uniform in $\varepsilon$.

\begin{lemma}[Local existence of a Fuchsian type system]\label{lem:y_exists}
Let $y_*=(1,\kappa,\varepsilon/(c\kappa+\varepsilon\gamma))\in \textsf{Ker}(\mathcal{M})$. There exist $\tilde{S} >0$ and an analytic solution of \eqref{eq:syMy} in $(-\tilde{S},\tilde{S})$ that satisfies $y(0)=y_*$, in particular $V(0)=1$. Moreover, this solution is unique in the class of analytic solutions of \eqref{eq:syMy} satisfying
\(y(0)=y_*\).
\end{lemma}
\begin{proof}
Let us consider the ansatz $y=\sum_{i=0}^\infty y_i s^i$, with $y_0=y_*$ and with $y_i=(V_i, U_i, W_i)\in\mathbb{R}^3$, $i\geq 1$, to be determined.
Replacing the ansatz in \eqref{eq:syMy}, we formally obtain
\begin{equation}\label{eq:yeqsum}
\sum_{i\geq1}iy_is^i=\sum_{i\geq 0}\mathcal{M}y_i s^i+\sum_{i\geq 1}g_is^i,
\end{equation}
where only the second component of $g_i\in\mathbb{R}^3$ might be nonzero, namely
\begin{equation*}
g_i=
\Biggl( 0~,\quad 
-\dfrac{1+\alpha}{\kappa}\sum_{j+k=i-1}V_j V_k
+\dfrac{1}{\kappa}\sum_{j+k+l=i-2}V_j V_k V_l~,\quad 
0
\Biggr)
\qquad i\ge1.
\end{equation*}
Equating the coefficients in \eqref{eq:yeqsum} gives the system of equations
\begin{equation}\label{eq:sys_y_is}
\begin{split}
&\mathcal{M}y_0=0,\\
&y_i=(iI-\mathcal{M})^{-1}g_i, \quad i\geq 1.
\end{split}
\end{equation}
The equation $\mathcal{M}y_0=0$ is satisfied since $y_0=y_*\in\textsf{Ker}(\mathcal{M})$, while the equations for $i\geq1$ can be solved recursively because each $g_i$ depends only on $y_0,\dots,y_{i-1}$ and, since the eigenvalues of $\mathcal{M}$ are non-positive, $(iI-\mathcal{M})$ is invertible for all $i\geq 1$.

Let $y_i$, $i\geq 1$, be the solution of \eqref{eq:sys_y_is}. We now prove that the corresponding series $y$ has a positive radius of convergence. We will use a standard majorant series argument. Let $\beta_i=\lVert y_i\rVert_\infty\in\mathbb{R}$, $i\geq 0$, and let $m_*=\lVert y_0\rVert_\infty$. For $N\in\mathbb{N}$ define the polynomial $B_N(s)=\sum_{i=0}^N \beta_i s^i$.

For the fixed value of \(\varepsilon\) considered in this lemma, there exists \(K_\varepsilon>0\) such that
\(\lVert (iI-\mathcal M)^{-1}\rVert_\infty\leq K_\varepsilon/i\) for every
\(i\geq1\). Indeed, \((iI-\mathcal M)^{-1}=i^{-1}(I-\mathcal M/i)^{-1}\). The matrices \(I-\mathcal M/i\) are invertible because the eigenvalues of
\(\mathcal M\) are nonpositive, and \(\lVert (I-\mathcal M/i)^{-1}\rVert_\infty\leq2\) whenever \(i>2\lVert \mathcal M\rVert_\infty\), by the Neumann series. The remaining finitely many inverses are bounded, which proves the claim. Therefore,
equation \eqref{eq:sys_y_is} and the definition of $g_i$ imply that

\begin{equation*}
i \beta_i\leq K_\varepsilon\Bigg[ \dfrac{1+\alpha}{\kappa} \sum_{j+k=i-1}\beta_j \beta_k + \dfrac{1}{\kappa}\sum_{j+k+l=i-2}\beta_j \beta_k \beta_l\Bigg], \quad i\geq 1.
\end{equation*}
Here an empty sum is understood to be zero. Assume $s\geq 0$, multiply by $s^{i-1}$ and sum from $i=1,\dots,N$ to obtain, for $N\geq 1$,
\begin{equation*}
B_N'(s)\leq \dfrac{K_\varepsilon(1+\alpha)}{\kappa} B_{N-1}(s)^2+\dfrac{K_\varepsilon}{\kappa} s B_{N-2}(s)^3,
\end{equation*}
where $B_0 (s)= \beta_0$, $B_{-1}(s) =0$. Let $\bar{K_\varepsilon}=K_\varepsilon(1+\alpha)/\kappa$. Since $B_{N-2}(s),B_{N-1}(s)\leq B_N(s)$ for $s\geq 0$, we get 
\begin{equation}\label{eq:boundBNsimple}
B_N'(s)\leq \bar{K_\varepsilon} B_N(s)^2+\dfrac{K_\varepsilon}{\kappa} s B_N(s)^3.
\end{equation}
Let us consider the function $\omega(s)=m_*(1-2\bar{K_\varepsilon}m_* s)^{-1}$, for $s\in \bigl[0, 1/(2\bar{K_\varepsilon}m_*)\bigr)$, and let $\tilde{S}< (2\bar{K_\varepsilon}m_*)^{-1}$ be such that
\[
s \omega(s) \leq 1 < (1+\alpha),\qquad \text{for all } s\in[0,\tilde{S}].
\]
These quantities are independent of $N$. Then
\begin{equation}\label{eq:eqomega}
\omega'(s)=2\bar{K_\varepsilon}\omega(s)^2\geq \bar{K_\varepsilon}\omega(s)^2+\dfrac{K_\varepsilon}{\kappa} s \omega(s)^3,\qquad s\in[0,\tilde{S}].
\end{equation}
Since $B_N(0)=m_*=\omega(0)$, equations \eqref{eq:boundBNsimple} and \eqref{eq:eqomega} imply that
\[
B_N(s)\leq \omega(s),\qquad s\in[0,\tilde{S}].
\]
In particular,
\[
\beta_i \tilde{S}^i\leq B_N(\tilde{S})\leq \omega(\tilde{S}),\qquad \forall  i=1,\dots,N,
\]
but the right-hand side is valid for any $N\in\mathbb{N}$, hence $\lVert y_i\rVert_\infty=\beta_i\leq \omega(\tilde{S})/\tilde{S}^i$ for all $i\geq 0$. This implies that $y$ is analytic with a convergence radius of at least $\tilde{S}$, and therefore $y$ is an analytic solution of \eqref{eq:syMy} in $(-\tilde{S},\tilde{S})$ with $y(0)=y_*$. 

Any analytic solution with the prescribed value at $s=0$ has Taylor coefficients satisfying the same recursion; hence all coefficients coincide with those constructed above and the solution is unique among analytic ones.
\end{proof}

Lemma \ref{lem:y_exists} provides explicit information on the radius of convergence and on bounds on the solution $y$, which are presented in the following corollary. Reducing \(\varepsilon_0\) if necessary, we also assume that \(c(\varepsilon)\leq1\) for \(0<\varepsilon\leq\varepsilon_0\).

\begin{corollary}[A priori estimates for the solution]\label{cor:unif_overS}
    Let $\delta>0$ be such that, for every $0<\varepsilon \leq  \varepsilon_0$ we have
    \[
    \delta\leq \gamma \leq \frac{1}{\delta}, \,\, 
    \delta\leq \alpha  \leq \dfrac{1}{\delta},\,\,
    \delta\leq \kappa \leq \dfrac{1}{\delta}, \,\,
    \delta\leq \lVert y_*\rVert_\infty \leq \dfrac{1}{\delta}, \,\,
    0< \varepsilon \leq \dfrac{1}{\delta}, \,\,
    0< \frac{\varepsilon}{c} \leq \dfrac{1}{\delta}.
    \]
    Such a $\delta>0$ exists by the definitions of $\kappa$, $\lVert y_*\rVert_\infty$, and $\varepsilon_0$.
    There exist \(S_0=S_0(\delta)>0\), \(\bar S=\bar S(\delta)\in(0,S_0)\), and \(C=C(\delta)>0\) such that, for every \(0<\varepsilon\leq\varepsilon_0\), the solution provided by Lemma~\ref{lem:y_exists} extends analytically to \((-S_0,S_0)\) and, for every \(s\in[0,\bar S]\),
    \[
    \lVert y(s)\rVert_\infty\leq C, \quad 
    \lvert V(s) - V(0)\rvert \leq \frac{1}{2}\lvert V(0)\rvert , \quad 
    \text{and}\quad 
    \lvert U(s) - U(0)\rvert \leq \frac{1}{2}\lvert U(0)\rvert .
    \]
        
\end{corollary}
\begin{proof}
In the proof of Lemma~\ref{lem:y_exists}, convergence of the power series and control of its coefficients depend only on upper bounds for \(m_*=\lVert y_*\rVert_\infty\) and
\(\bar K_\varepsilon=K_\varepsilon(1+\alpha)/\kappa\).
The quantities \(\lVert y_*\rVert_\infty\), \(\alpha\), and \(\kappa\) are already bounded uniformly in terms of \(\delta\). It therefore remains to bound \(K_\varepsilon\) uniformly for \(0<\varepsilon\leq\varepsilon_0\).

Cramer's rule gives $(I-\mathcal{M}/i)^{-1}=\operatorname{adj}(I-\mathcal{M}/i)/\det(I-\mathcal{M}/i)$. Since $\mu_j \leq 0$ for $j=1, 2, 3$,  
\begin{equation*}
\det(I-\mathcal{M}/i) = \prod_{j=1}^3 \Bigl(1-\frac{\mu_j}{i}\Bigr)\geq 1.
\end{equation*}
Also, by Hadamard's inequality, each entry of $\operatorname{adj}(I-\mathcal{M}/i)$ is bounded by $\lVert I-\mathcal{M}/i\rVert_2^2$. Therefore 
\[
K_\varepsilon\leq 3 \sup_{i\geq 1}\lVert I-\mathcal{M}/i\rVert_2^2 \leq 3\Big(1+\sqrt{3}\lVert \mathcal{M}\rVert_\infty\Big)^2.
\]
All the entries of $\mathcal{M}=(\kappa^{-1}A-I)$ can be bounded uniformly in terms of $\delta$. 
By the preceding estimate, there exists \(K_0=K_0(\delta)>0\)
such that \(K_\varepsilon\leq K_0\) for every
\(0<\varepsilon\leq\varepsilon_0\). Moreover,
\(m_*=\lVert y_*\rVert_\infty\leq m_0:=\delta^{-1}\) and
\((1+\alpha)/\kappa\leq(1+\delta^{-1})/\delta\). Thus, with
\[
\Lambda_0:=K_0\frac{1+\delta^{-1}}{\delta},
\qquad
\omega_0(s):=\frac{m_0}{1-2\Lambda_0m_0s},
\]
the majorant argument in Lemma~\ref{lem:y_exists} applies uniformly
whenever \(s\geq0\), \(2\Lambda_0m_0s<1\), and
\(s\omega_0(s)\leq1\). We may therefore choose
\(S_0=S_0(\delta)>0\) so that these inequalities hold on
\([0,S_0]\). It follows that the power series defining \(y\) converges
for \(\lvert s\rvert <S_0\), uniformly with respect to
\(0<\varepsilon\leq\varepsilon_0\), and that
\[
\sum_{i\geq0}\lVert y_i\rVert_\infty \lvert s\rvert ^i\leq\omega_0(\lvert s\rvert)
\qquad (\lvert s\rvert \leq S_0).
\]
Set \(C_0:=\omega_0(S_0/2)\). From Cauchy's integral formula, it follows  that \(\lVert y'(s)\rVert_\infty\leq 4C_0/S_0\) for \(\lvert s\rvert \leq S_0/4\).
The mean value theorem then implies that
\[
\lvert V(s)-V(0)\rvert \leq\frac{1}{2}\lvert V(0)\rvert 
\quad\text{if}\quad
\lvert s\rvert \leq\frac{S_0}{4}
\min\biggl\{1,\frac{\lvert V(0)\rvert }{2C_0}\biggr\}.
\]
The analogous estimate holds for $U(s)$. Since \(V(0)=1\) and \(U(0)=\kappa\geq\delta\), the conclusions hold on \([0,\bar S]\), where
\[
\bar S:=\min\biggl\{\frac{S_0}{4},
\frac{S_0}{8C_0},\frac{\delta S_0}{8C_0}\biggr\}.
\]
Since \(\bar S\leq S_0/4\), the majorant estimate gives \(\lVert y(s)\rVert_\infty\leq\omega_0(s)\leq C_0\) for \(s\in[0,\bar S]\). Thus, the conclusions follow with \(C:=C_0\).
\end{proof}

Let us recap what we have achieved so far in constructing the tail.

For fixed $\alpha,\gamma>0$, the hypotheses of Lemma \ref{lem:y_exists} and Corollary \ref{cor:unif_overS} hold uniformly for $0<\varepsilon\leq \varepsilon_0$. Fix the resulting constants $C$ and $\bar{S}$, which depend only on the fixed parameters and the initial choice of $\varepsilon_0$, and define $\bar{\xi}=\bar{\xi}(\varepsilon):=\ln(\bar{S})/\kappa(\varepsilon)$. 
Since \(\bar S\) is independent of \(\varepsilon\) and \(\delta\leq\kappa(\varepsilon)\leq\delta^{-1}\), both \(\kappa(\varepsilon)\) and \(\bar\xi(\varepsilon)\) are uniformly bounded for \(0<\varepsilon\leq\varepsilon_0\). 
The transformation~\eqref{eq:change-vars} then defines a solution $X(\xi)=(v,u,w)$ of \eqref{eq:uvw-first-order}$^=$ on the interval $(-\infty,\bar{\xi}]$ such that
\begin{equation*}
(v,u,w)\sim e^{\kappa \xi}\biggl(1,\kappa,\frac{\varepsilon}{c}~\frac{1}{\kappa+\frac{\varepsilon}{c}\gamma}
\biggr),\quad \text{ as } \xi\to - \infty.
\end{equation*}
Moreover,
\begin{equation*}
\frac12e^{\kappa\xi}\leq v(\xi)\leq \frac32e^{\kappa\xi}, \qquad \frac{\kappa}{2}e^{\kappa\xi}\leq u(\xi)\leq \frac{3\kappa}{2}e^{\kappa\xi}, \qquad \xi\in(-\infty,\bar{\xi}]. 
\end{equation*}
In particular, since $e^{\kappa\bar{\xi}}=\bar S$, 
\[ 
\frac{\bar S}{2}\leq v(\bar{\xi})\leq\frac{3\bar S}{2}, \qquad \frac{\kappa\bar S}{2}\leq u(\bar{\xi}) \leq\frac{3\kappa\bar S}{2}. 
\]
The constants \(C\) and \(\bar S\) were chosen uniformly on the original interval \(0<\varepsilon\leq\varepsilon_0\) and therefore remain valid if \(\varepsilon_0\) is reduced later in the proof. Consequently,  the construction of the tail is such that $(v,u,w)\to(0,0,0)$ as $\xi\to-\infty$, while
 $v(\bar{\xi})$ and $u(\bar{\xi})$ are bounded away from zero, and bounded above, uniformly, for $0<\varepsilon\leq \varepsilon_0$. 

To finalize Step~1, we estimate $w$. For $\xi\leq\bar{\xi}$, integration and the equation for $w'$ in \eqref{eq:uvw-first-order}$^=$ give
\begin{equation*}
     e^{\frac{\varepsilon \gamma }{c} \xi} w(\xi) - 0 = \int_{-\infty}^\xi   \frac{d}{d\zeta}\Bigl( e^{\frac{\varepsilon \gamma }{c}\zeta} w\Bigr) d\zeta = \int_{-\infty}^\xi e^{\frac{\varepsilon \gamma}{c}\zeta} \Bigl(w'+\frac{\varepsilon\gamma}{c} w \Bigr) d\zeta
     =
     \frac{\varepsilon}{c} \int_{-\infty}^\xi e^{\frac{\varepsilon \gamma}{c}\zeta} v(\zeta) \,d\zeta .
\end{equation*}
Hence, for $\xi \in (-\infty,\bar{\xi}]$,
\begin{equation*}
     w(\xi) = \frac{\varepsilon}{c} \int_{-\infty}^{\xi} e^{\frac{\varepsilon \gamma}{c} (\zeta-\xi)} v(\zeta) d\zeta 
     \leq \frac{\varepsilon}{c} \int_{-\infty}^{\xi} v(\zeta) d\zeta \leq \frac{\varepsilon}{c} \int_{-\infty}^{\xi}  \frac{3}{2} e^{\kappa \zeta} d\zeta 
     \leq \frac{3\varepsilon}{\kappa c}~\frac{1}{2}  e^{\kappa \xi}
     \leq \frac{3\varepsilon}{\kappa c}v(\xi).
\end{equation*}
By reducing $\varepsilon_0$ if necessary (such that $3\gamma\varepsilon/\kappa c <1$ for all $\varepsilon<\varepsilon_0$) we also conclude from \eqref{eq:uvw-first-order}$^=$ that $w'>0$ and $w>0$ for all $\xi \in (-\infty,\bar{\xi})$.

The tail constructed in $(-\infty,\bar{\xi})$ is such that $v(\xi)$ reaches a level between $\bar{S}/2$ and $3\bar{S}/2$ at $\bar{\xi}$, which is uniformly bounded away from zero for $0<\varepsilon\leq \varepsilon_0$. By contrast, $w(\bar{\xi})\lesssim \frac{\varepsilon}{c}v(\bar{\xi})$ can be made arbitrarily small by reducing the value of $\varepsilon$. This is a key ingredient in the next step.

\noindent
\textbf{Step 2: }The climb. We show that we can extend the solution of \eqref{eq:uvw-first-order}$^=$ to the right of $\bar{\xi}$ so that $v$ reaches an arbitrarily large value $K>v(\bar{\xi})$, after reducing $\varepsilon_0$ if necessary.

Let $F(r) = -\int_0^r f(\rho) d\rho, r\geq 0$. Since \(F'(r)=-f(r)=r(r-\alpha)(r-1)\) and $\alpha\in(1/2,1)$, the function \(F\) is strictly increasing on \([0,\alpha)\), decreasing on \((\alpha,1)\), and increasing on \((1,\infty)\). At the local minimum $r=1$, we have $F(1)=(2\alpha-1)/12>0$, hence, by reducing $\bar{S}$ if necessary, we can assume that 
$3\bar{S}/2<\alpha$ and $F(3\bar{S}/2)\leq F(1)/2$. Consequently,
\begin{equation*}
    \min_{r\geq \frac{3\bar{S}}{2}
    }F(r) = F\biggl(\frac{3\bar{S}}{2}\biggr)
    \quad \Longrightarrow\quad F(\bar{r})<F(r) \text{ for all } r>\bar{r} \text{ and $\bar{r}\in\frac{\bar{S}}{2}[1,3]$}.
\end{equation*}
Let
\[
H(\bar{r},r) := 
\begin{cases}
(F(r)-F(\bar{r}))/\sum_{j=1}^4 (r-\bar{r})^j, & \text{for } \bar{r}\in \frac{\bar{S}}{2}[1,3] \text{ and }r\in (\bar{r},\infty),\\[2pt]
F'(\bar{r}), &  \text{for } \bar{r}\in \frac{\bar{S}}{2}[1,3] \text{ and } r=\bar{r},\\[2pt]
1/4, & \text{for } \bar{r}\in \frac{\bar{S}}{2}[1,3] \text{ and } r=\infty.
\end{cases}
\]
Then $H$ is continuous and strictly positive on its compact extended domain. Hence,
\[
0< \eta_F :=\min \{ H(\bar{r},r): {\bar{r}\in \frac{\bar{S}}{2}[1,3], r\in [\bar{r},\infty]} \}
\]
and therefore $F(\bar{r})+\eta_F\sum_{j=1}^4 (r-\bar{r})^j \leq F(r)$ for every $\bar{r}\in \frac{\bar{S}}{2}[1,3]$ and $r\geq \bar{r}$. Note that $\eta_F$ depends only on $\bar{S}$ and $\alpha\in(1/2,1)$. In particular, we have
\begin{equation}\label{eq:F_quarticbound}
    F(v(\bar{\xi}))+\eta_F\sum_{j=1}^4 (r-v(\bar{\xi}))^j\leq F(r),\quad \text{ for all } r\geq v(\bar{\xi}).
\end{equation}
Let $K>v(\bar{\xi}^-)$ be fixed. Now consider $\xi_0$ and $(v,u,w)$ as the unique maximal solution of the following constrained initial value problem, which continues the tail solution of \eqref{eq:uvw-first-order}$^=$ beyond $\bar{\xi}$. For $\xi\in(\bar{\xi},\xi_0)$,
\begin{empheq}[left=\empheqlbrace]{equation}\label{eq:uvw-step2_version2}
\begin{aligned}
     &v' = u,\\
     &u' = -f(v)+ cu - w,\\
     &w' = \frac{\varepsilon}{c}(v-\gamma w),\\
     &v(\bar{\xi})=v(\bar{\xi}^-),\quad  u(\bar{\xi})=u(\bar{\xi}^-),\quad  w(\bar{\xi})=w(\bar{\xi}^-),\\ 
     &v(\xi)< K,\quad \text{ and }\quad  v'(\xi)=u(\xi)>u(\bar{\xi}^-)/2>0.
\end{aligned}
\end{empheq}
where $v(\bar{\xi}^-), u(\bar{\xi}^-),  w(\bar{\xi}^-)$ are the values from the tail. Using arguments similar to those in Lemma \ref{lem:generic_ODE}, it follows that such constrained maximal solution exists, that $\xi_0<\infty$, and that either $v(\xi_0^-)=K$, or $\liminf_{\xi\to \xi_0^-}u(\xi)=u(\bar{\xi}^-)/2$.

Observe that Step~1 gives $v(\bar{\xi})-\gamma w(\bar{\xi})>0$. Set $z=v-\gamma w$. If $z$ had a first zero $\xi_\ast\in(\bar{\xi},\xi_0)$, then \( w'(\xi_\ast)=\frac{\varepsilon}{c}z(\xi_\ast)=0 \) and hence \( z'(\xi_\ast) =u(\xi_\ast)-\gamma w'(\xi_\ast) =u(\xi_\ast)>0\), contradicting the fact that $\xi_\ast$ is the first point at which $z$ reaches zero from above. Therefore $v-\gamma w>0$ on $(\bar{\xi},\xi_0)$, and consequently $w'>0$ there.
Equation \eqref{eq:uvw-step2_version2} also implies  $v'>u(\bar{\xi}^-)/2>0$, and therefore $v$ is strictly increasing and positive for all $\xi\in (\bar{\xi},\xi_0)$. One consequence is that $w'\leq (\varepsilon/c)v$ for $\xi\in(\bar{\xi},\xi_0)$, and therefore 
\begin{equation*}
\begin{split}
    \int_{\bar{\xi}}^\xi w(s) v'(s)\, ds
    &\leq \int_{\bar{\xi}}^\xi \biggl(w(\bar{\xi})+\int_{\bar{\xi}}^s \frac{\varepsilon}{c}v(t)\, dt\biggr) v'(s)\, ds \\
    &= w(\bar{\xi})(v(\xi)-v(\bar{\xi}))
    + \frac{\varepsilon}{c}\int_{\bar{\xi}}^\xi \biggl(\int_{t}^\xi v'(s) \, ds\biggr)v(t) \, dt \\
    &= w(\bar{\xi})(v(\xi)-v(\bar{\xi}))
    + \frac{\varepsilon}{c}\int_{\bar{\xi}}^\xi (v(\xi)-v(t))v(t)\, dt\\
    &\leq  \frac{\varepsilon}{c}\frac{3\bar{S}}{2\kappa}(v(\xi)-v(\bar{\xi})) + \frac{\varepsilon}{c}(\xi-\bar{\xi})(v(\xi)-v(\bar{\xi}))v(\xi)\\
    &\leq  \frac{\varepsilon}{c}\biggl(\frac{3\bar{S}}{2\kappa}(v(\xi)-v(\bar{\xi})) + \frac{2}{u(\bar{\xi}^-)}(v(\xi)-v(\bar{\xi}))^2v(\xi)\biggr)\\    
    &\leq  \frac{\varepsilon}{c}\biggl(\frac{3\bar{S}}{2\kappa}(v(\xi)-v(\bar{\xi})) + \frac{4}{\kappa\bar{S}}(v(\xi)-v(\bar{\xi}))^2v(\bar{\xi}) +\frac{4}{\kappa\bar{S}}(v(\xi)-v(\bar{\xi}))^3\biggr)\\    
    &\leq  \frac{\varepsilon}{c}\biggl(\frac{3\bar{S}}{2\kappa}(v(\xi)-v(\bar{\xi})) + \frac{6}{\kappa}(v(\xi)-v(\bar{\xi}))^2 +\frac{4}{\kappa\bar{S}}(v(\xi)-v(\bar{\xi}))^3\biggr),
\end{split}
\end{equation*}
where we used that $v$ is increasing and that $v(\xi)-v(\bar{\xi})\geq (\xi-\bar{\xi})u(\bar{\xi}^-)/2$. Since $\varepsilon/c\to0$ as $\varepsilon\to0$, the estimate above and \eqref{eq:F_quarticbound} imply that there exists $\varepsilon_1\leq \varepsilon_0$ depending only on the uniform Step~1 constants and $\alpha$, such that if $0<\varepsilon\leq \varepsilon_1$ then
\begin{equation}\label{eq:intwu_leq_F}
\begin{split}
    \int_{\bar{\xi}}^\xi w(s) v'(s)\, ds
    &\leq F(v(\xi))-F(v(\bar{\xi})) \quad \text{ for all } \quad \xi\in [ \bar{\xi},\xi_0).
\end{split}
\end{equation}
Next, multiply the equation for $u'$ in \eqref{eq:uvw-step2_version2} by $u=v'$, integrate, and make the change of variables $\rho=v(s)$. For any $\xi\in(\bar{\xi},\xi_0)$ and $0<\varepsilon\leq \varepsilon_1$, this gives
\begin{equation}\label{eq:vprime_estim_step2}
\begin{split}
\frac12\bigl(u(\xi)\bigr)^2
&=\frac12\bigl(u(\bar{\xi})\bigr)^2
+c\int_{\bar{\xi}}^\xi\bigl(u(s)\bigr)^2\,ds
-\int_{v(\bar{\xi})}^{v(\xi)} f(\rho)\,d\rho 
-\int_{\bar{\xi}}^\xi w(s) u(s)\,ds\\
&\geq \frac12\bigl(u(\bar{\xi})\bigr)^2
+(F(v(\xi))-F(v(\bar{\xi})))
-\int_{\bar{\xi}}^\xi w(s) v'(s)\,ds\\
&\geq \frac12\bigl(u(\bar{\xi})\bigr)^2,
\end{split}
\end{equation}
where the last inequality follows from \eqref{eq:intwu_leq_F}. Estimate \eqref{eq:vprime_estim_step2} implies that $u(\xi) \geq u(\bar{\xi})$ for all  $\xi\in(\bar{\xi},\xi_0)$. Among the two alternatives at $\xi_0^-$ for the maximal solution of \eqref{eq:uvw-step2_version2}, only $v(\xi_0^-)=K$ is feasible, concluding the construction with the desired property. Observe that $\varepsilon_1$ in this proof does not depend on the value of $K$.

\noindent
\textbf{Step 3:} The summit. In this step, we will show that, given $K_1, K_2 >0$, the construction can be continued to the right of $\xi_0$, while preserving \eqref{eq:uvw-first-order} and the transmission inequality $[v']_q \leq 0$. The function $v$ becomes constant at $\xi_0$, so that $v'(\xi) = 0$ for $\xi \geq  \xi_0$. We then show that there exists $\xi_1\geq \xi_0$ such that $v(\xi) \geq K_1$ and $w(\xi) \geq K_2$ for $\xi \geq \xi_1$.

Choose $K > \max\{K_1,\gamma K_2\}$ sufficiently large that $-f(K)\geq K/\gamma$, and perform the construction in Step~2 with this $K$. Together with the tail constructed in Step~1, this gives a smooth solution \((v,u,w)\) of \eqref{eq:uvw-first-order}$^=$ for $\xi\in(-\infty,\xi_0)$ that decays exponentially as $\xi\to-\infty$. Moreover, $v$, $u$, and $w$ are strictly positive there, $v$ and $w$ are strictly increasing, and $v(\xi_0^-)=K$, $v'(\xi_0^-)>0$, and $K=v(\xi_0^-)\geq \gamma w(\xi_0^-)$.
To complete the construction, we set $v(\xi):=K$ and $u(\xi):=0$ for $\xi\geq \xi_0$, and define $w$ on $[\xi_0, \infty)$ as the solution of
\begin{equation*}
\begin{cases}
    w'=(\varepsilon/c)(K-\gamma w), \quad \xi>\xi_0,\\
    w(\xi_0)=w(\xi_0^-).
\end{cases}
\end{equation*}
Then $w$ is continuously differentiable through $\xi_0$, nondecreasing, and satisfies $w(\xi)\leq K/\gamma$ and $\lim_{\xi\to\infty} w(\xi)=K/\gamma$. The inequality $-f(K)\geq K/\gamma \geq w(\xi)$ implies that $-f(K) -w(\xi)\geq 0$ for $\xi\geq \xi_0$. Thus, this choice of $(v,u,w)$ satisfies \eqref{eq:uvw-first-order} in $(\xi_0,\infty)$. At $\xi_0$, we have $[\partial_\xi v]_{\xi_0} =u(\xi_0^+)-u(\xi_0^-)=-u(\xi_0^-) < 0$, so the transmission inequality is satisfied. Since $K>\gamma K_2$ and $w(\xi)$ converges to $K/\gamma$, there exists $\xi_1\geq \xi_0$ such that $w(\xi)> K_2$ for $\xi\geq \xi_1$.

The preceding construction gives a globally bounded triple $(v,u,w)$ satisfying \eqref{eq:uvw-first-order}, with $v$ and $w$ positive and nondecreasing. All three components decay exponentially as $\xi\to-\infty$, and there exists $\xi_1\in\mathbb R$ such that 
\[ v(\xi)\geq K_1,\qquad w(\xi)\geq K_2 \qquad \text{for } \xi\geq\xi_1. \] 
Moreover, the only interface introduced by the construction is the plateau interface at $\xi=\xi_0$, where 
\[ [v']_{\xi_0}=-v'(\xi_0^-)\leq0. \] 
Thus, the ansatz and the transformation introduced at the beginning of the section define a barrier for \eqref{FHN_pde_general}. The interface \(\xi=\xi_0\) corresponds in the original variables to 
\[x_0(t):=\sqrt{b/a}(\xi_0-c(\varepsilon)t),\qquad [\partial_x\bar v]_{x_0(t)} =\sqrt{a/b}[v']_{\xi_0}\leq0.\]
Hence, the transmission inequality is preserved, and the barrier moves to the left with speed
\(\sqrt{b/a}c(\varepsilon)\).
The symmetric lower components have the same boundedness, monotonicity, and decay properties. Since $\varepsilon/c(\varepsilon)\to0$ as $\varepsilon\to0$, all the smallness conditions imposed in the construction hold whenever $\tau>\tau_0$, after increasing $\tau_0$ if necessary. This proves Theorem~\ref{thm_barriers_conduction_block}.
\end{proof}

\section{Numerical experiments}\label{section_numerical_experiments}
In this section, we present numerical simulations that illustrate the analytical barrier constructions and explore the dynamics beyond the parameter regimes covered by the theory. The heterogeneous experiments examine localized suppression of the activator dynamics, its interaction with spatial diffusion of the recovery variable, and abrupt changes in the geometry of the medium. Despite these differences, the heterogeneous simulations display a similar transition from pulse transmission to attenuation and finite-time blocking. We also examine almost stationary barriers in a homogeneous medium, which provide threshold-region containment bounds whose speed tends to zero as the recovery timescale increases. 
The analytical results are stated on $\mathbb R$, whereas the simulations use bounded intervals with homogeneous Neumann conditions; in Experiments~1--3 the barriers are constant near both endpoints and therefore remain barriers on the bounded domain, while Experiment~4 is purely illustrative.
The observation of these effects with a straightforward numerical discretization suggests that the barrier mechanism is robust and allows us to compare the sufficient analytical conditions with the apparent dynamical transition between propagation and blocking.

All simulations were performed using Wolfram Mathematica, Version 15. We document the spatial discretization, time integration, treatment of coefficient discontinuities, and the full code used to generate each figure in a GitHub repository%
\footnote{\url{https://github.com/estebanpaduro/fhn_barrier_heterogeneous_media}}.

Unless stated otherwise, we use the compactly supported initial data
\begin{equation}\label{initial_data}
v^{\mathrm{in}}(x)
=
\frac{A}{2}
\Bigl(1+\cos\Bigl(\pi\frac{x-x_0}{h}\Bigr)\Bigr)
\mathbf 1_{[x_0-h,x_0+h]}(x),
\qquad
w^{\mathrm{in}}(x)=0.
\end{equation}
We classify a simulation as propagating if the solution observed at \(x=x_{\min}\) develops a peak of height at least \(0.5\) before the final simulation time \(T\), and as exhibiting finite-time blocking otherwise. This computational criterion differs from the all-time directional blocking introduced in
Definition~\ref{defi_directional_blocking}.

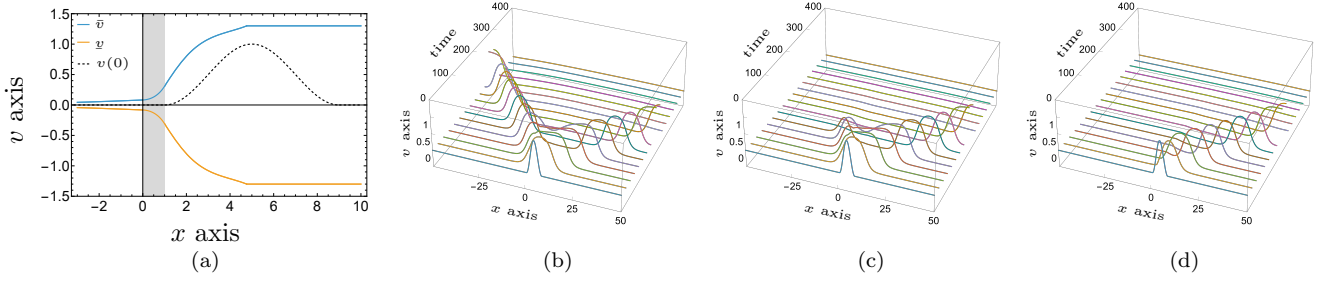
\begin{figure}
    \centering

    \subfloat[\label{fig2a}]{%
        \begin{minipage}[c][3.4cm][c]{0.28\textwidth}
            \centering
            \begin{tikzpicture}
                \node[anchor=south west, inner sep=0cm] (image1)
                at (0,0.5cm) {%
                    \includegraphics[width=0.9\linewidth]
                    {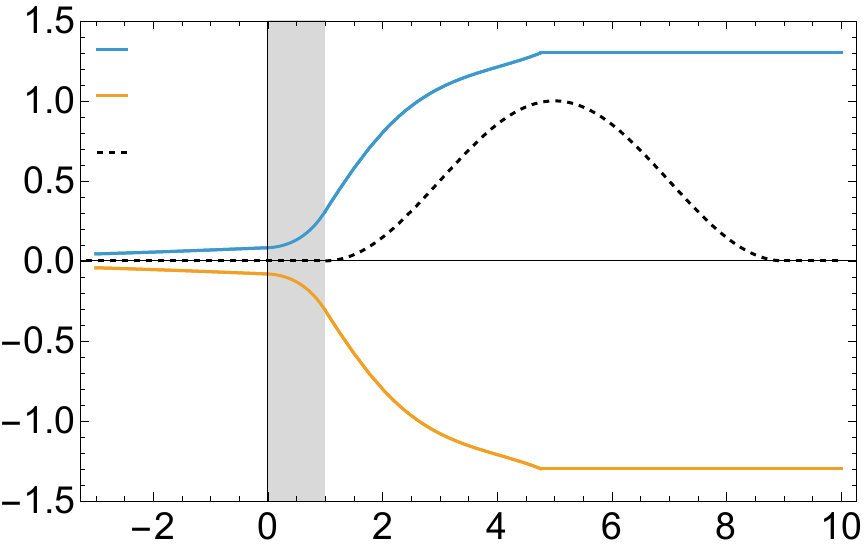}%
                };
                \begin{scope}[
                    overlay,
                    shift={(image1.south west)},
                    x={(image1.south east)},
                    y={(image1.north west)}
                ]
                    \node at (0.18,0.91)
                        {\tiny $\bar{v}$};
                    \node at (0.18,0.81)
                        {\tiny $\underaccent{\bar}{v}$};
                    \node at (0.22,0.72)
                        {\tiny $v^\mathrm{in}$};
                    \node at (0.5,-0.08)
                        {$x$ axis};
                    \node[rotate=90] at (-0.08,0.5)
                        {$v$ axis};
                \end{scope}
            \end{tikzpicture}
        \end{minipage}%
    }%
    \hfill
    \subfloat[\label{fig2b}]{%
        \begin{minipage}[c][3cm][c]{0.22\textwidth}
            \centering
            \begin{tikzpicture}
                \node[anchor=south west, inner sep=0cm] (image2)
                at (0,0) {%
                    \includegraphics[width=\linewidth]
                    {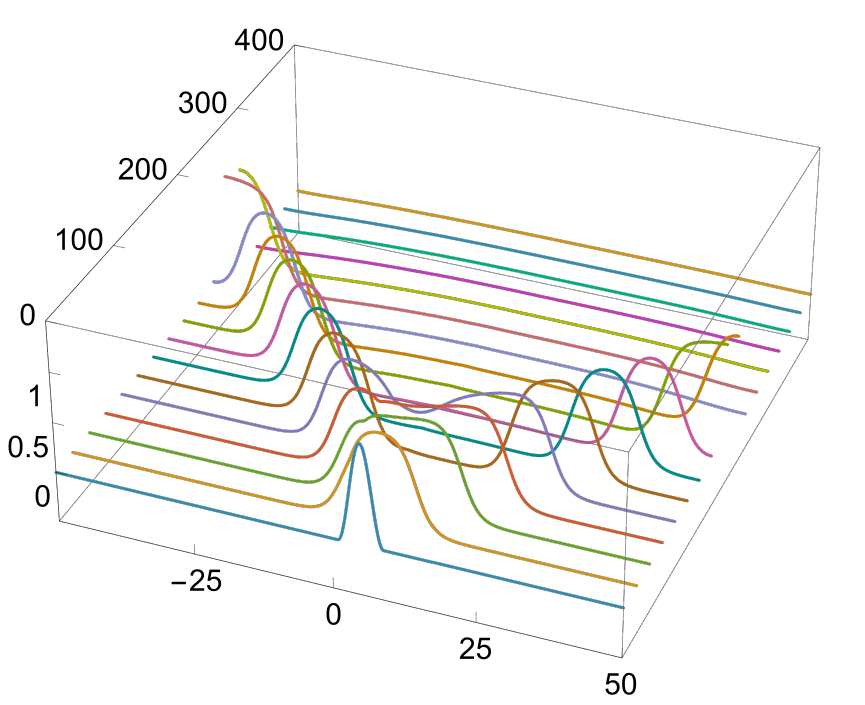}%
                };
                \begin{scope}[
                    overlay,
                    shift={(image2.south west)},
                    x={(image2.south east)},
                    y={(image2.north west)}
                ]
                    \node[rotate=-12] at (0.35,0.07)
                        {\tiny $x$ axis};
                    \node[rotate=90] at (-0.03,0.4)
                        {\tiny $v$ axis};
                    \node[rotate=50] at (0.1,0.79)
                        {\tiny time};
                \end{scope}
            \end{tikzpicture}
        \end{minipage}%
    }%
    \hfill
    \subfloat[\label{fig2c}]{%
        \begin{minipage}[c][3cm][c]{0.22\textwidth}
            \centering
            \begin{tikzpicture}
                \node[anchor=south west, inner sep=0cm] (image3)
                at (0,0) {%
                    \includegraphics[width=\linewidth]
                    {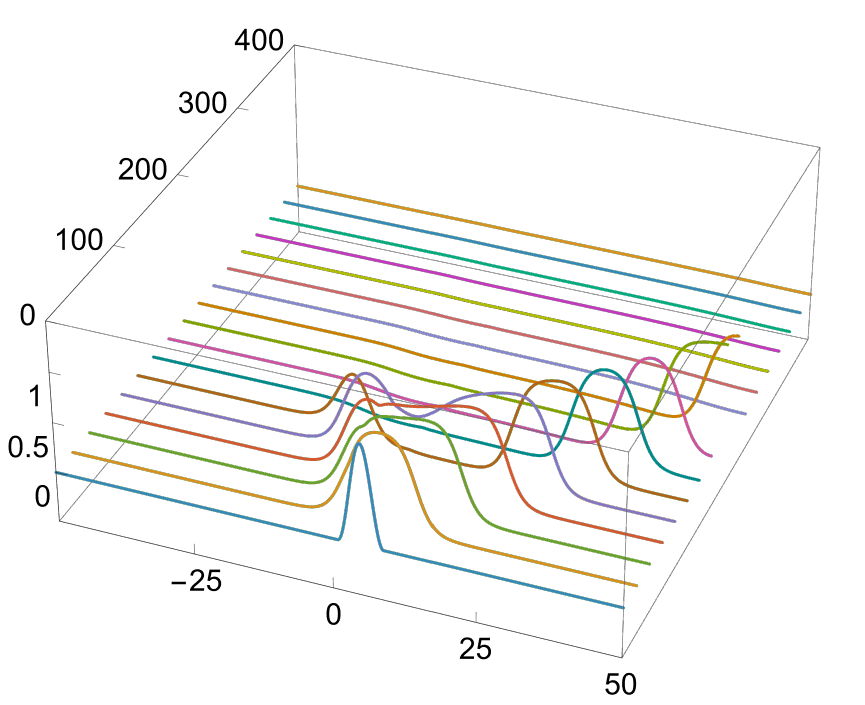}%
                };
                \begin{scope}[
                    overlay,
                    shift={(image3.south west)},
                    x={(image3.south east)},
                    y={(image3.north west)}
                ]
                    \node[rotate=-12] at (0.35,0.07)
                        {\tiny $x$ axis};
                    \node[rotate=90] at (-0.03,0.4)
                        {\tiny $v$ axis};
                    \node[rotate=50] at (0.1,0.79)
                        {\tiny time};
                \end{scope}
            \end{tikzpicture}
        \end{minipage}%
    }%
    \hfill
    \subfloat[\label{fig2d}]{%
        \begin{minipage}[c][3cm][c]{0.22\textwidth}
            \centering
            \begin{tikzpicture}
                \node[anchor=south west, inner sep=0cm] (image4)
                at (0,0) {%
                    \includegraphics[width=\linewidth]
                    {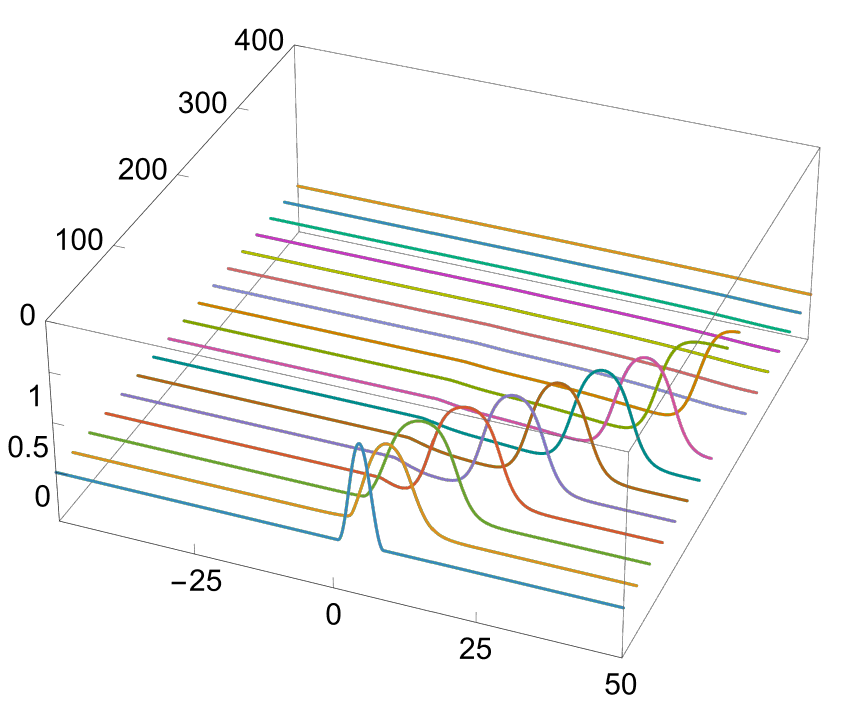}%
                };
                \begin{scope}[
                    overlay,
                    shift={(image4.south west)},
                    x={(image4.south east)},
                    y={(image4.north west)}
                ]
                    \node[rotate=-12] at (0.35,0.07)
                        {\tiny $x$ axis};
                    \node[rotate=90] at (-0.03,0.4)
                        {\tiny $v$ axis};
                    \node[rotate=50] at (0.1,0.79)
                        {\tiny time};
                \end{scope}
            \end{tikzpicture}
        \end{minipage}%
    }

    \caption{Localized depressed dynamics generated by the additive heterogeneity in 
    Theorem~\ref{thm_barrier_heterogeneous},
    case~\ref{thm_heterogeneity_type1}.
    (a) Barrier constructed with $p=3.7$ and
    $V_{-}=0.01$; the gray region marks the interval of depressed excitability, and the dashed curve is the initial activator profile.
    (b)--(d) Activator evolution for
    $p=0.11$, $p=0.12$, and $p=3.7$, respectively.}
    \label{Fig2}
\end{figure}

\noindent
\textbf{Experiment 1: Localized depressed dynamics.} We first consider the additive heterogeneity in Theorem~\ref{thm_barrier_heterogeneous}, case~\ref{thm_heterogeneity_type1}. With the choice \(g(v)=-v\), the heterogeneous reaction term \(p g(v)=-pv\) locally suppresses the activator dynamics. This configuration models a finite region of depressed excitability, whose intensity and spatial extent are determined by \(p\) and \(L\), respectively. It also corresponds to a shunting light-gated conductance with reversal potential \(E_{\mathrm{rev}}=0\).

The simulations are performed on \( [x_{\min},x_{\max}]=[-40,50] \) up to time \(T=400\), with homogeneous Neumann boundary conditions on \(v\). We take 
\[ a=b=1,\qquad D=0,\qquad f_0(v)=v(1-v)(v-\alpha),\qquad \alpha=0.3,\qquad \gamma=5, \] 
together with \(g(v)=-v\), \(L=1\), and \(\tau=500\). In particular, \(\alpha\gamma>1\), so condition~\eqref{condition_H1} is satisfied. 

The bounded-domain barrier is constructed with 
\[ K=1.3,\qquad V_{-}=0.01,\qquad p=3.7. \] 
Since \(V_{-}>0\), the profile is constant near the left endpoint and is compatible with the homogeneous Neumann condition. The initial data are given by \eqref{initial_data} with \(A=1\), \(x_0=5\), and \(h=4\),
and are trapped between the barrier components, as shown in Figure~\ref{Fig2}(a). 

Figure~\ref{fig2b}--\ref{fig2d} compares the solution dynamics for \(p=0.11\), \(p=0.12\), and \(p=3.7\). For \(p=0.11\), two pulses propagate in opposite directions, and the left-moving pulse crosses the region of depressed excitability. For \(p=0.12\), the attenuation accumulated while crossing this region prevents the left-moving pulse from emerging as a propagating pulse. For \(p=3.7\), the suppression is considerably stronger, and the numerical solution remains between the displayed activator barriers, consistent with Theorem~\ref{thm_barrier_heterogeneous}.

To examine the combined influence of the intensity and spatial extent of the heterogeneity, we next vary 
\[ L\in[0,1],\qquad \Delta L=0.05, \] 
and 
\[ p\in[0,1],\qquad \Delta p=0.05.\] 
Figure~\ref{fig4a} indicates that shorter heterogeneous regions require larger values of \(p\) to block the left-moving pulse. The dark and light regions indicate propagation and finite-time blocking, respectively. Thus, within the tested parameter range, blocking depends on the combined intensity and spatial extent of the region of depressed excitability: increasing either \(p\) or \(L\) strengthens the cumulative suppression experienced by the pulse.

\begin{figure}
    \centering

    \subfloat[\label{fig3a}]{%
        \begin{minipage}[c][3.4cm][c]{0.28\textwidth}
            \centering
            \begin{tikzpicture}
                \node[anchor=south west, inner sep=0cm] (image1)
                at (0,0.5cm) {%
                    \includegraphics[width=0.9\linewidth]
                    {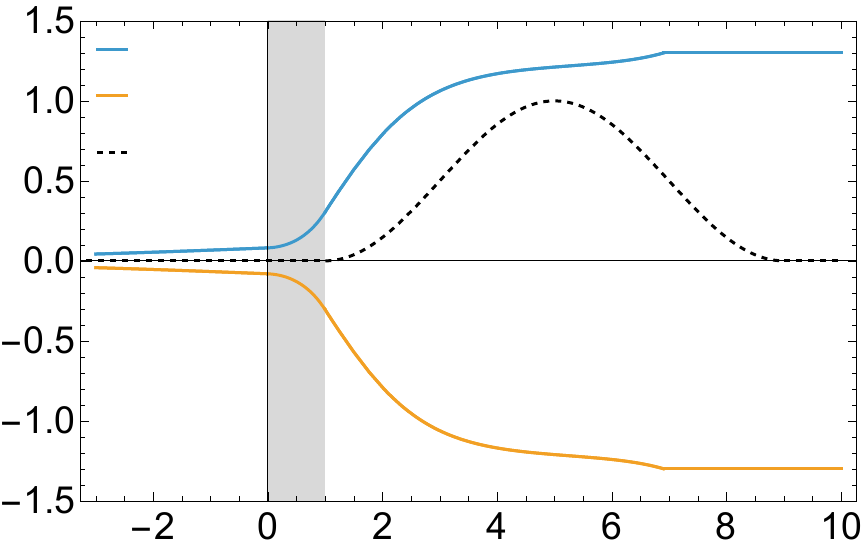}%
                };
                \begin{scope}[
                    overlay,
                    shift={(image1.south west)},
                    x={(image1.south east)},
                    y={(image1.north west)}
                ]
                    \node at (0.18,0.91)
                        {\tiny $\bar{v}$};
                    \node at (0.18,0.81)
                        {\tiny $\underaccent{\bar}{v}$};
                    \node at (0.22,0.72)
                        {\tiny $v^\mathrm{in}$};
                    \node at (0.5,-0.08)
                        {$x$ axis};
                    \node[rotate=90] at (-0.08,0.5)
                        {$v$ axis};
                \end{scope}
            \end{tikzpicture}
        \end{minipage}%
    }%
    \hfill
    \subfloat[\label{fig3b}]{%
        \begin{minipage}[c][3cm][c]{0.22\textwidth}
            \centering
            \begin{tikzpicture}
                \node[anchor=south west, inner sep=0cm] (image2)
                at (0,0) {%
                    \includegraphics[width=\linewidth]
                    {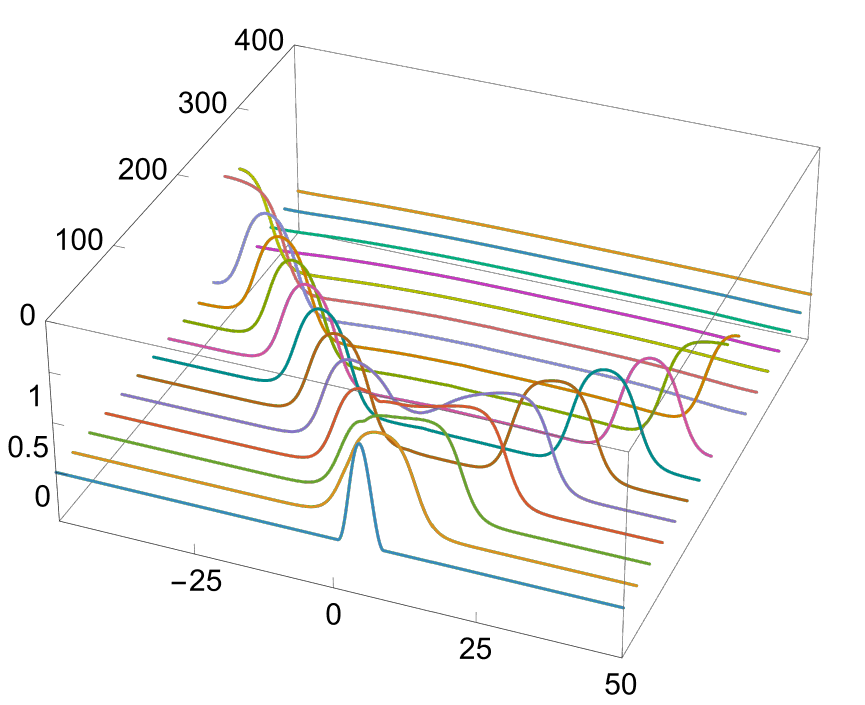}%
                };
                \begin{scope}[
                    overlay,
                    shift={(image2.south west)},
                    x={(image2.south east)},
                    y={(image2.north west)}
                ]
                    \node[rotate=-12] at (0.35,0.07)
                        {\tiny $x$ axis};
                    \node[rotate=90] at (-0.03,0.4)
                        {\tiny $v$ axis};
                    \node[rotate=50] at (0.1,0.79)
                        {\tiny time};
                \end{scope}
            \end{tikzpicture}
        \end{minipage}%
    }%
    \hfill
    \subfloat[\label{fig3c}]{%
        \begin{minipage}[c][3cm][c]{0.22\textwidth}
            \centering
            \begin{tikzpicture}
                \node[anchor=south west, inner sep=0cm] (image3)
                at (0,0) {%
                    \includegraphics[width=\linewidth]
                    {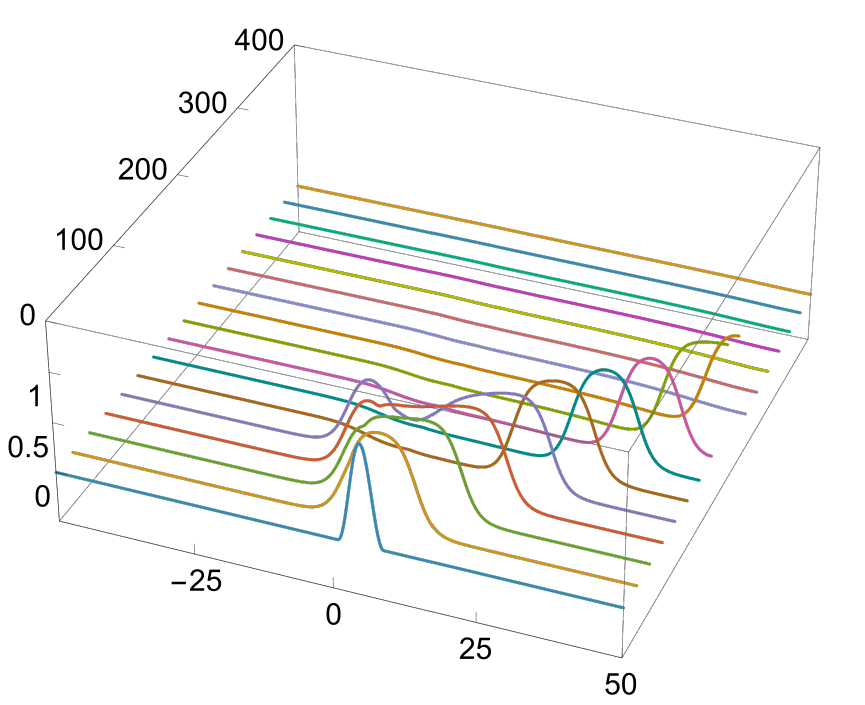}%
                };
                \begin{scope}[
                    overlay,
                    shift={(image3.south west)},
                    x={(image3.south east)},
                    y={(image3.north west)}
                ]
                    \node[rotate=-12] at (0.35,0.07)
                        {\tiny $x$ axis};
                    \node[rotate=90] at (-0.03,0.4)
                        {\tiny $v$ axis};
                    \node[rotate=50] at (0.1,0.79)
                        {\tiny time};
                \end{scope}
            \end{tikzpicture}
        \end{minipage}%
    }%
    \hfill
    \subfloat[\label{fig3d}]{%
        \begin{minipage}[c][3cm][c]{0.22\textwidth}
            \centering
            \begin{tikzpicture}
                \node[anchor=south west, inner sep=0cm] (image4)
                at (0,0) {%
                    \includegraphics[width=\linewidth]
                    {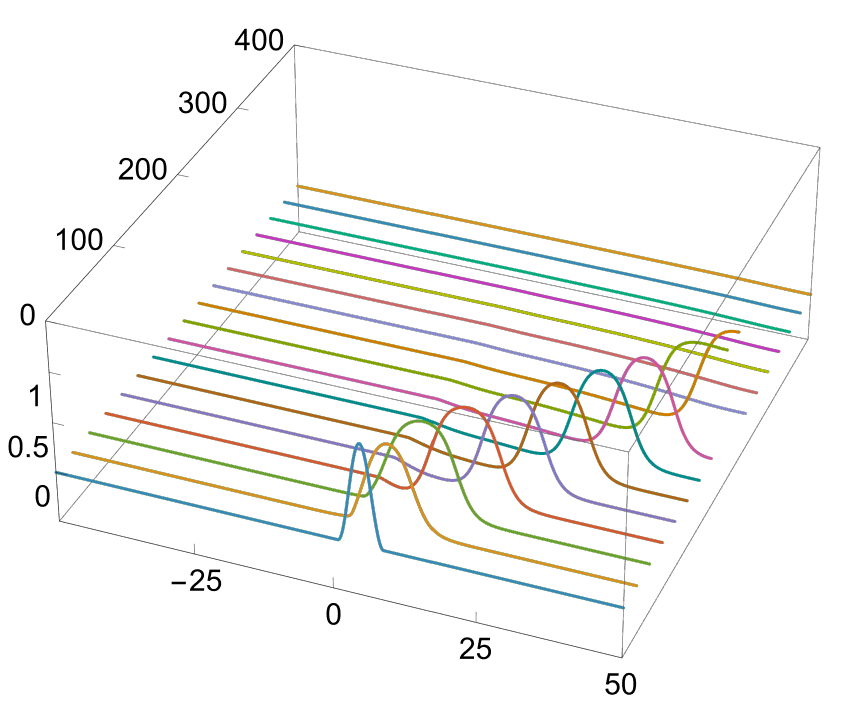}%
                };
                \begin{scope}[
                    overlay,
                    shift={(image4.south west)},
                    x={(image4.south east)},
                    y={(image4.north west)}
                ]
                    \node[rotate=-12] at (0.35,0.07)
                        {\tiny $x$ axis};
                    \node[rotate=90] at (-0.03,0.4)
                        {\tiny $v$ axis};
                    \node[rotate=50] at (0.1,0.79)
                        {\tiny time};
                \end{scope}
            \end{tikzpicture}
        \end{minipage}%
    }
    
    \caption{Localized depressed dynamics with recovery diffusion \(D=1\). (a) Barrier constructed for \(D=0\) using \(\tilde\gamma=4.97\), \(p=3.7\), and \(V_{-}=0.01\); the gray region marks the heterogeneous interval, and the dashed curve is the initial activator profile. (b)--(d) Activator evolution for \(p=0.11\), \(p=0.12\), and \(p=3.7\), respectively.
    }
    \label{Fig3}
\end{figure}
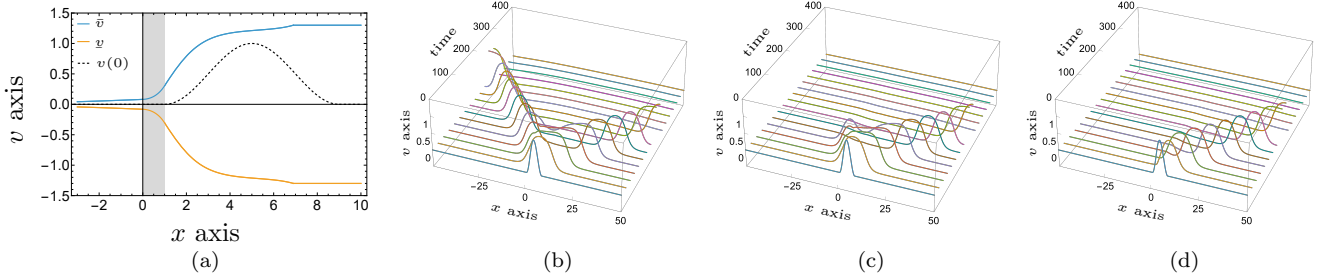

\noindent
\textbf{Experiment 2: Recovery diffusion.} 
We repeat the experiment on localized depressed dynamics with recovery diffusion \(D=1\) and homogeneous Neumann boundary conditions on both components. The activator encounters the same finite region of depressed excitability as in Experiment~1, but the recovery variable can now redistribute spatially. This configuration allows us to examine how recovery diffusion modifies the interaction between the propagating pulse and the region of depressed excitability. The domain, final time, nonlinearity, heterogeneity, $\tau$, and initial data are otherwise the same as in Experiment~1.

Motivated by Theorem~\ref{thm_double_diffusive}, we first construct the barrier for \(D=0\) using the modified recovery parameter
\[ 
\tilde\gamma=4.97<\gamma=5, \qquad K=1.3, \qquad V_{-}=0.01, \qquad p=3.7. 
\] 
Figure~\ref {fig3a} shows the resulting barrier along with the initial data. 

Figure~\ref{fig3b}--\ref{fig3d} displays the dynamics for \(p=0.11\), \(p=0.12\), and \(p=3.7\). Despite the additional spatial coupling through the recovery component, the simulations exhibit the same qualitative transition as in Experiment~1.
The left-moving pulse crosses the heterogeneous region for \(p=0.11\), whereas finite-time blocking occurs for \(p=0.12\). For \(p=3.7\), the suppression is considerably stronger, and the solution remains between the displayed activator barriers over the simulated time interval. Since \(D=1\) is not known to lie in the small-\(D\) range provided by Theorem~\ref{thm_double_diffusive}, this last observation should be interpreted as numerical evidence rather than as a direct consequence of the theorem. 

We also investigate how recovery diffusion affects the intensity of the localized reaction defect required for finite-time blocking. We fix \(L=0.5\) and vary
\[ p\in[0,0.3],\qquad \Delta p=0.02, \] 
and 
\[ D\in[1,131],\qquad \Delta D=10. \] 
Figure~\ref{fig4b} shows that, within the tested parameter range, the value of \(p\) required to produce finite-time blocking decreases as \(D\) increases. Thus, in these simulations, stronger recovery diffusion enhances the effectiveness of the localized reaction defect in obstructing pulse propagation. This trend is numerical and lies beyond the small-recovery-diffusion regime established by Theorem~\ref{thm_double_diffusive}.

\begin{figure}
    \centering

    \subfloat[\label{fig4a}]{%
        \begin{minipage}[c][6.5cm][c]{0.35\textwidth}
            \centering
            \begin{tikzpicture}
                \node[anchor=south west, inner sep=0cm] (image1)
                at (0,0) {%
                    \includegraphics[width=\linewidth]
                    {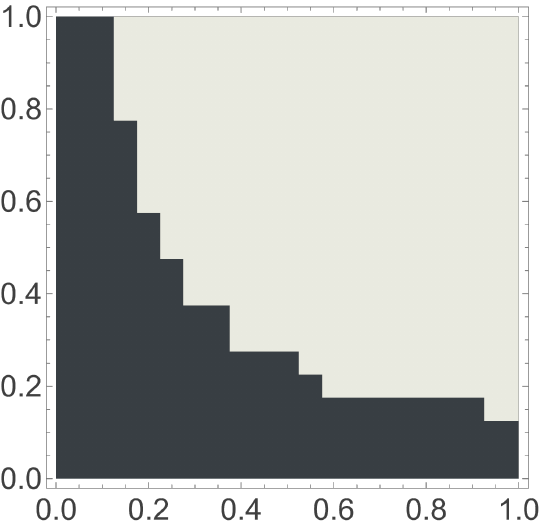}%
                };
                \begin{scope}[
                    overlay,
                    shift={(image1.south west)},
                    x={(image1.south east)},
                    y={(image1.north west)}
                ]
                    \node at (0.5,-0.04)
                        {$p$ axis};
                    \node[rotate=90] at (-0.08,0.5)
                        {$L$ axis};
                \end{scope}
            \end{tikzpicture}
        \end{minipage}
    }
    \hspace{1cm}
    \subfloat[\label{fig4b}]{%
        \begin{minipage}[c][6.5cm][c]{0.35\textwidth}
            \centering
            \begin{tikzpicture}
                \node[anchor=south west, inner sep=0cm] (image1)
                at (0,0) {%
                    \includegraphics[width=\linewidth]
                    {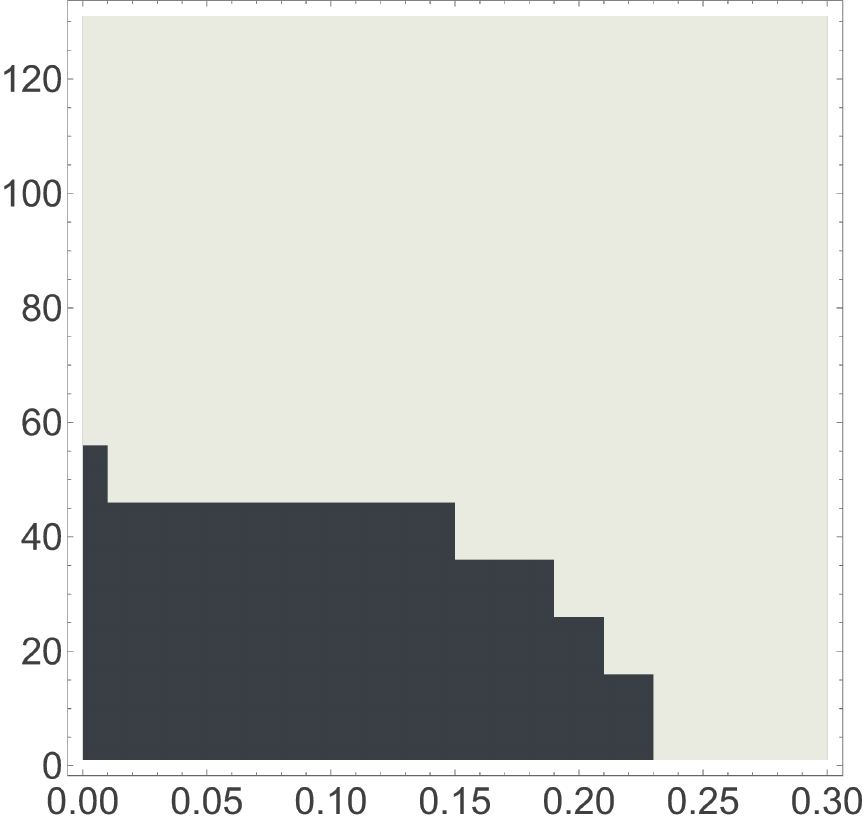}%
                };
                \begin{scope}[
                    overlay,
                    shift={(image1.south west)},
                    x={(image1.south east)},
                    y={(image1.north west)}
                ]
                    \node at (0.5,-0.04)
                        {$p$ axis};
                    \node[rotate=90] at (-0.08,0.5)
                        {$D$ axis};
                \end{scope}
            \end{tikzpicture}
        \end{minipage}
    }
    \caption{Numerical classification of propagation for the additive heterogeneity. Dark regions indicate propagation, while light regions indicate finite-time blocking according to the finite-time threshold criterion. (a) Dependence on the intensity \(p\) and spatial extent \(L\) of the region of depressed excitability for \(D=0\). (b) Dependence on the recovery diffusion \(D\) and heterogeneity intensity \(p\) for \(L=0.5\).
    }
    \label{fig4dependenceParametersBarrier}
\end{figure}
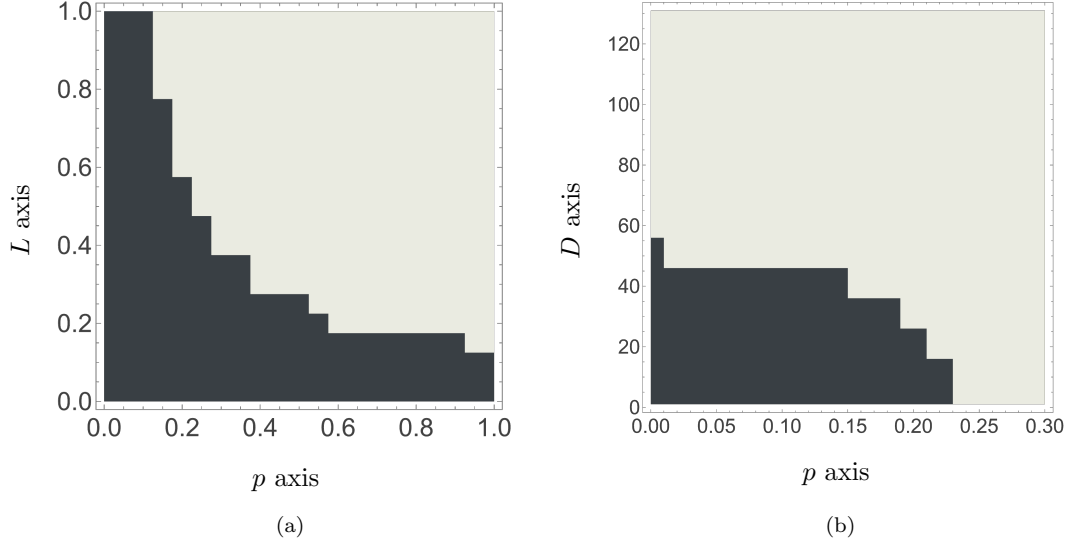

\noindent 
\textbf{Experiment 3: Abrupt geometric changes.} We next examine the conductivity heterogeneities in Theorem~\ref{thm_barrier_heterogeneous}, cases~\ref{thm_heterogeneity_type3} and \ref{thm_heterogeneity_type4}. 
In these configurations, the coefficients \(a=r\) and \(b=r^2\) represent abrupt changes in the geometry of the axon. Unlike the localized reaction heterogeneities in Experiments~1 and 2, these heterogeneities leave the reaction terms unchanged and instead modify the activator flux across the interfaces.
The domain, final time, nonlinearity, parameters $D$, $\gamma$, and $\tau$, and initial data are the same as in Experiment~1. 

For case~\ref{thm_heterogeneity_type3}, let \( a(x)=r(x), b(x)=r(x)^2, \) where 
\[ 
r(x)= \begin{cases} 
1, & x\leq0,\\[2mm] 
(1+p)^{-1}, & x>0. 
\end{cases} 
\] 
For the left-moving pulse considered here, this configuration represents a single abrupt transition from the reduced-\(r\) region \(x>0\) into the reference region \(x\leq0\). The pulse therefore encounters one conductivity interface rather than a finite heterogeneous segment. The barrier in Figure~\ref{fig5a} is constructed with \(K=1.3\), \(V_{-}=0.005\), and \(p=10.5\).

Figure~\ref{fig5b}--\ref{fig5d} compares the dynamics for \(p=1.37\), \(p=1.38\), and \(p=10.5\). For \(p=1.37\), the left-moving pulse crosses the conductivity interface and continues to propagate into the region \(x<0\), whereas finite-time blocking occurs for \(p=1.38\). For \(p=10.5\), the pulse is more strongly attenuated at the interface, and the numerical solution remains between the displayed activator barriers, consistent with Theorem~\ref{thm_barrier_heterogeneous}. The finite-time blocking observed for \(p=1.38\), well below the value used in the explicit barrier construction, illustrates that the analytical condition is sufficient but not expected to be sharp. 

The observed blocking is directional because the coefficient configuration is not invariant under spatial reflection. Consequently, the barrier predicting failure of propagation toward the left does not imply that propagation toward the right is also impeded.

\begin{figure}
    \centering

    \subfloat[\label{fig5a}]{%
        \begin{minipage}[c][3.5cm][c]{0.28\textwidth}
            \centering
            \begin{tikzpicture}
                \node[anchor=south west, inner sep=0cm] (image1)
                at (0,0.5cm) {%
                    \includegraphics[width=0.9\linewidth]
                    {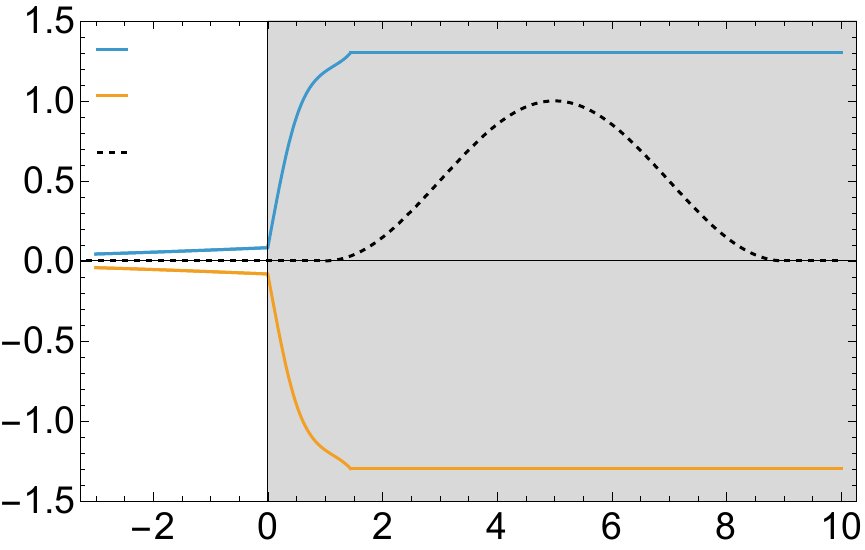}%
                };
                \begin{scope}[
                    overlay,
                    shift={(image1.south west)},
                    x={(image1.south east)},
                    y={(image1.north west)}
                ]
                    \node at (0.18,0.91)
                        {\tiny $\bar{v}$};
                    \node at (0.18,0.81)
                        {\tiny $\underaccent{\bar}{v}$};
                    \node at (0.22,0.72)
                        {\tiny $v^\mathrm{in}$};
                    \node at (0.5,-0.08)
                        {$x$ axis};
                    \node[rotate=90] at (-0.08,0.5)
                        {$v$ axis};
                \end{scope}
            \end{tikzpicture}
        \end{minipage}%
    }%
    \hfill
    \subfloat[\label{fig5b}]{%
        \begin{minipage}[c][3.5cm][c]{0.22\textwidth}
            \centering
            \begin{tikzpicture}
                \node[anchor=south west, inner sep=0cm] (image2)
                at (0,0) {%
                    \includegraphics[width=\linewidth]
                    {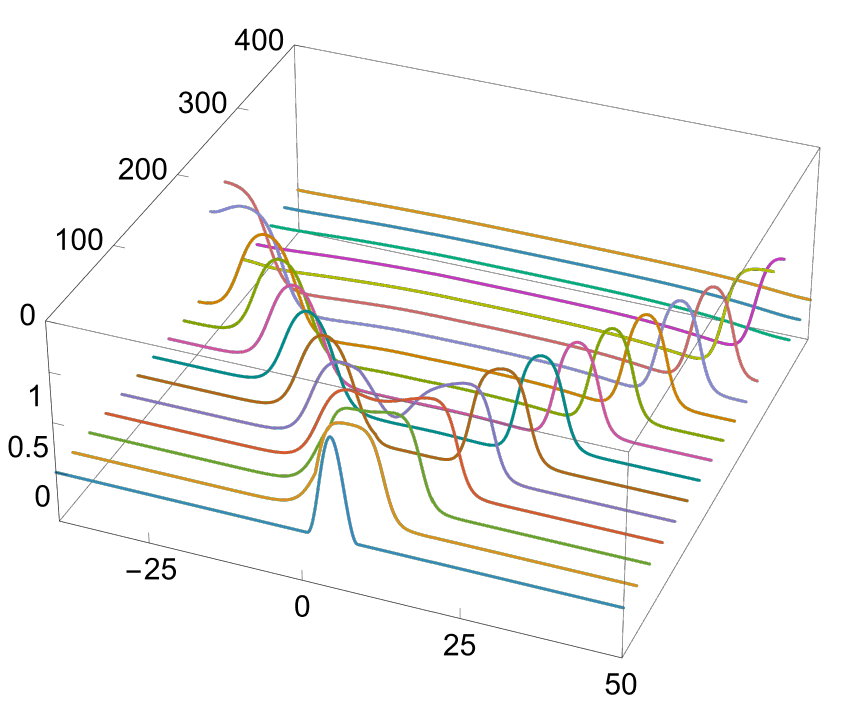}%
                };
                \begin{scope}[
                    overlay,
                    shift={(image2.south west)},
                    x={(image2.south east)},
                    y={(image2.north west)}
                ]
                    \node[rotate=-12] at (0.35,0.07)
                        {\tiny $x$ axis};
                    \node[rotate=90] at (-0.03,0.4)
                        {\tiny $v$ axis};
                    \node[rotate=50] at (0.1,0.79)
                        {\tiny time};
                \end{scope}
            \end{tikzpicture}
        \end{minipage}%
    }%
    \hfill
    \subfloat[\label{fig5c}]{%
        \begin{minipage}[c][3.5cm][c]{0.22\textwidth}
            \centering
            \begin{tikzpicture}
                \node[anchor=south west, inner sep=0cm] (image3)
                at (0,0) {%
                    \includegraphics[width=\linewidth]
                    {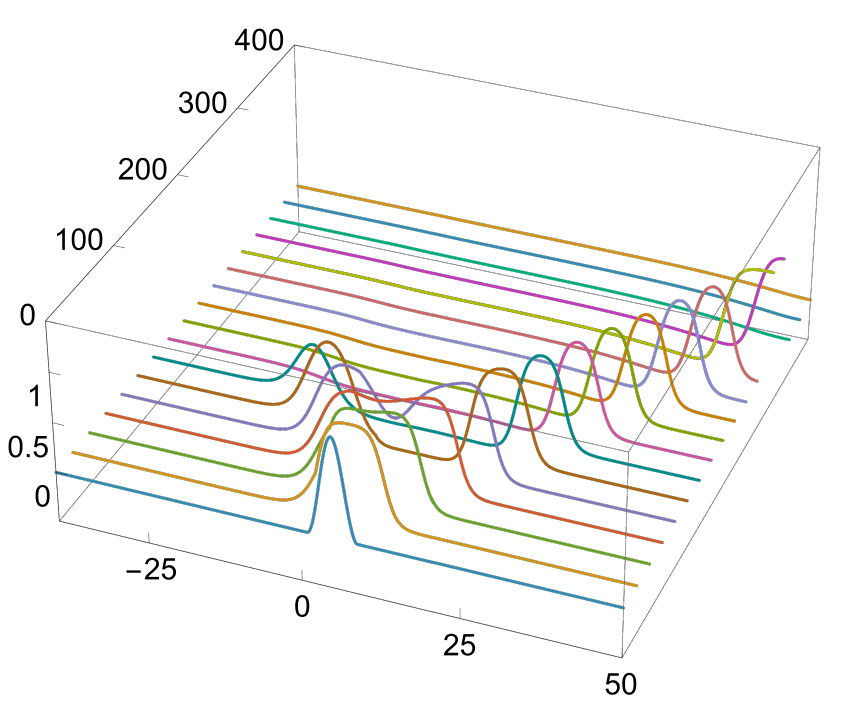}%
                };
                \begin{scope}[
                    overlay,
                    shift={(image3.south west)},
                    x={(image3.south east)},
                    y={(image3.north west)}
                ]
                    \node[rotate=-12] at (0.35,0.07)
                        {\tiny $x$ axis};
                    \node[rotate=90] at (-0.03,0.4)
                        {\tiny $v$ axis};
                    \node[rotate=50] at (0.1,0.79)
                        {\tiny time};
                \end{scope}
            \end{tikzpicture}
        \end{minipage}%
    }%
    \hfill
    \subfloat[\label{fig5d}]{%
        \begin{minipage}[c][3.5cm][c]{0.22\textwidth}
            \centering
            \begin{tikzpicture}
                \node[anchor=south west, inner sep=0cm] (image4)
                at (0,0) {%
                    \includegraphics[width=\linewidth]
                    {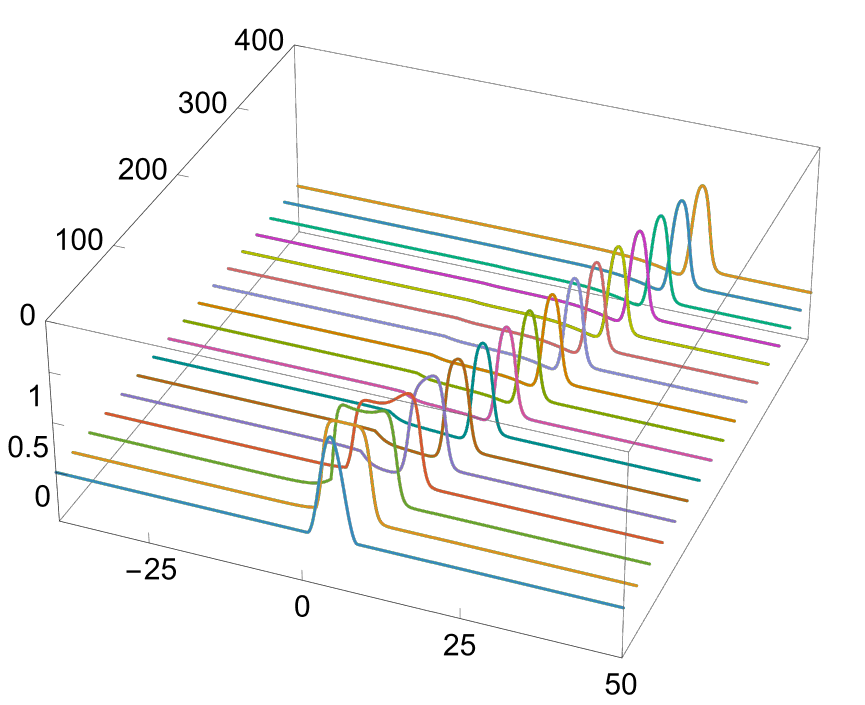}%
                };
                \begin{scope}[
                    overlay,
                    shift={(image4.south west)},
                    x={(image4.south east)},
                    y={(image4.north west)}
                ]
                    \node[rotate=-12] at (0.35,0.07)
                        {\tiny $x$ axis};
                    \node[rotate=90] at (-0.03,0.4)
                        {\tiny $v$ axis};
                    \node[rotate=50] at (0.1,0.79)
                        {\tiny time};
                \end{scope}
            \end{tikzpicture}
        \end{minipage}%
    }

    \caption{Pulse interaction with a single abrupt conductivity transition in Theorem~\ref{thm_barrier_heterogeneous}, case~\ref{thm_heterogeneity_type3}. (a) Barrier constructed with \(p=10.5\) and \(V_{-}=0.005\), together with the initial activator profile; the interface is located at \(x=0\). (b)--(d) Activator evolution for \(p=1.37\), \(p=1.38\), and \(p=10.5\), respectively.
    }
    \label{fig5}
\end{figure}
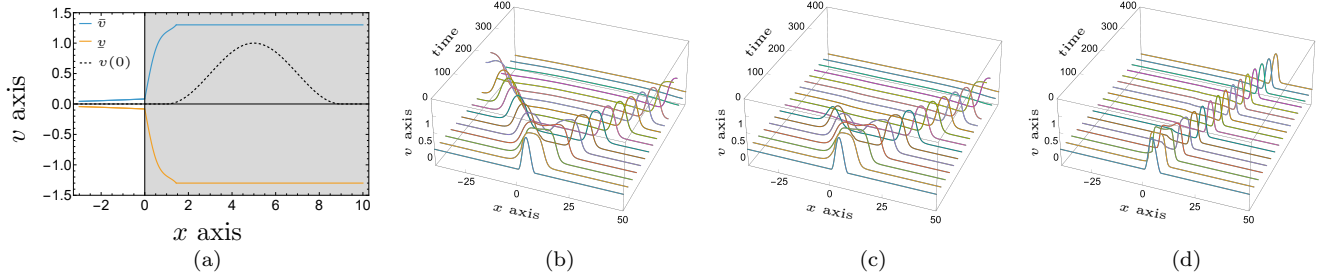

For case~\ref{thm_heterogeneity_type4}, the conductivity reduction is confined to a finite interval:
\[ 
r(x)= 
\begin{cases} 
1, & x<0,\\[1mm] 
(1+p)^{-1}, & 0\leq x\leq L,\\[1mm] 
1, & x>L, 
\end{cases} \qquad L=0.8. 
\] 
The left-moving pulse therefore enters a finite region with reduced \(r\) through the interface at \(x=L\) and exits it through the interface at \(x=0\). In contrast with the preceding configuration, the conductivity reduction is localized to a finite interval bounded by two interfaces. The barrier shown in Figure~\ref{fig6a} is constructed with \(K=1.3\), \(V_{-}=0.005\), and \(p=11\).

Figure~\ref{fig6b}--\ref{fig6d} displays the dynamics for \(p=1.37\), \(p=1.38\), and \(p=11\), respectively. Although the pulse now encounters a localized conductivity reduction bounded by two interfaces, the simulations exhibit the same qualitative transition as in the single-interface configuration. For \(p=1.37\), the left-moving pulse traverses the region of reduced conductivity and continues to propagate into the homogeneous region \(x<0\), whereas finite-time blocking occurs for \(p=1.38\). For \(p=11\), the pulse is more strongly attenuated, and the solution remains between the displayed activator barriers. 

Because the medium is homogeneous for \(x<0\), propagation into this region in Figure~\ref{fig6b} confirms that the initial pulse can continue to propagate after crossing the region of reduced conductivity. The attenuation observed in panels (c) and (d) can therefore be attributed to its interaction with the localized conductivity heterogeneity.

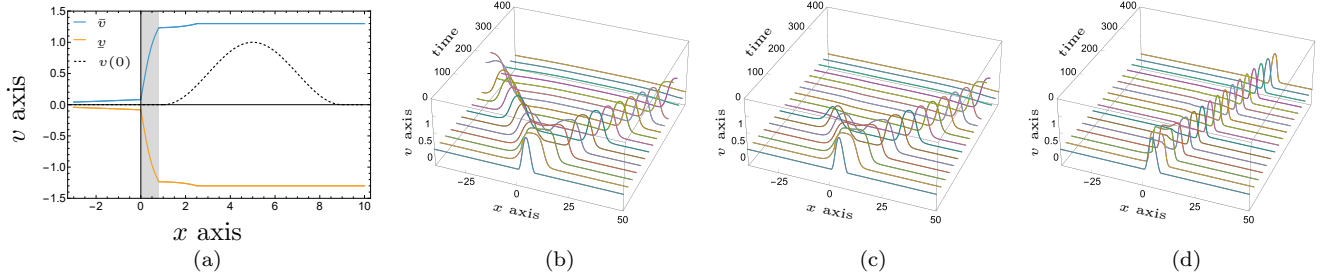
\begin{figure}
    \centering

    \subfloat[\label{fig6a}]{%
        \begin{minipage}[c][3.5cm][c]{0.28\textwidth}
            \centering
            \begin{tikzpicture}
                \node[anchor=south west, inner sep=0cm] (image1)
                at (0,0.5cm) {%
                    \includegraphics[width=0.9\linewidth]
                    {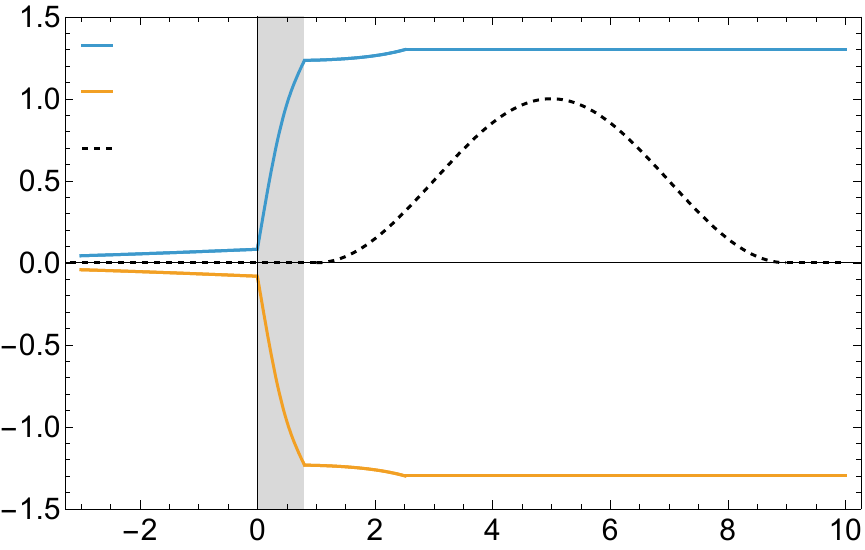}%
                };
                \begin{scope}[
                    overlay,
                    shift={(image1.south west)},
                    x={(image1.south east)},
                    y={(image1.north west)}
                ]
                    \node at (0.18,0.91)
                        {\tiny $\bar{v}$};
                    \node at (0.18,0.81)
                        {\tiny $\underaccent{\bar}{v}$};
                    \node at (0.22,0.72)
                        {\tiny $v^\mathrm{in}$};
                    \node at (0.5,-0.08)
                        {$x$ axis};
                    \node[rotate=90] at (-0.08,0.5)
                        {$v$ axis};
                \end{scope}
            \end{tikzpicture}
        \end{minipage}%
    }%
    \hfill
    \subfloat[\label{fig6b}]{%
        \begin{minipage}[c][3.5cm][c]{0.22\textwidth}
            \centering
            \begin{tikzpicture}
                \node[anchor=south west, inner sep=0cm] (image2)
                at (0,0) {%
                    \includegraphics[width=\linewidth]
                    {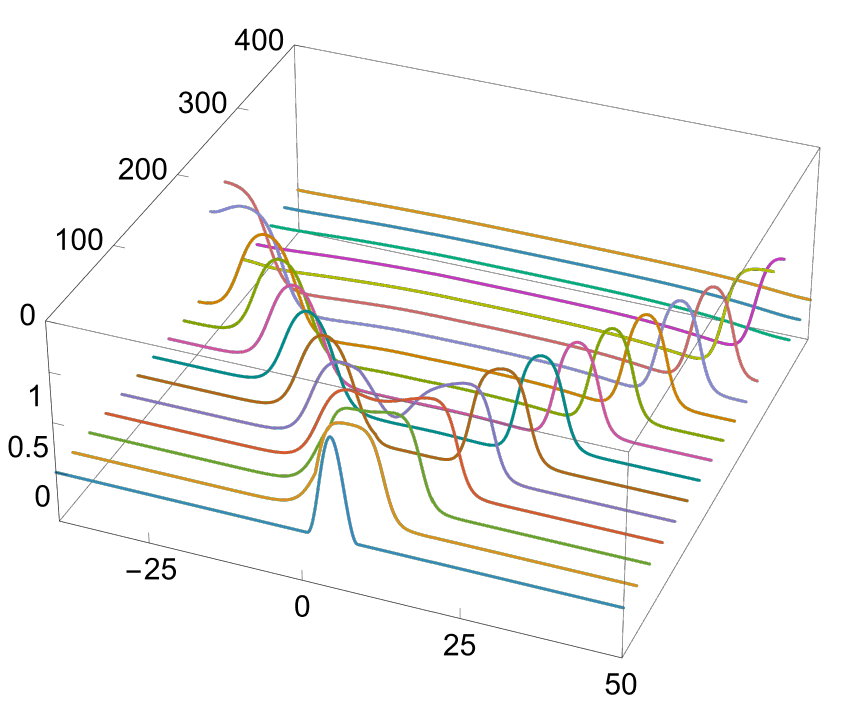}%
                };
                \begin{scope}[
                    overlay,
                    shift={(image2.south west)},
                    x={(image2.south east)},
                    y={(image2.north west)}
                ]
                    \node[rotate=-12] at (0.35,0.07)
                        {\tiny $x$ axis};
                    \node[rotate=90] at (-0.03,0.4)
                        {\tiny $v$ axis};
                    \node[rotate=50] at (0.1,0.79)
                        {\tiny time};
                \end{scope}
            \end{tikzpicture}
        \end{minipage}%
    }%
    \hfill
    \subfloat[\label{fig6c}]{%
        \begin{minipage}[c][3.5cm][c]{0.22\textwidth}
            \centering
            \begin{tikzpicture}
                \node[anchor=south west, inner sep=0cm] (image3)
                at (0,0) {%
                    \includegraphics[width=\linewidth]
                    {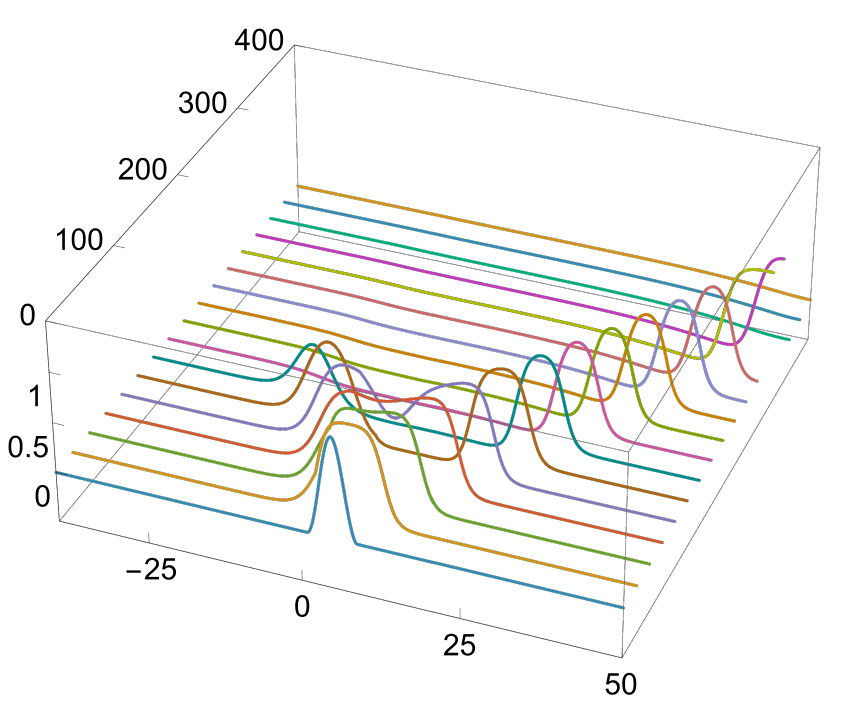}%
                };
                \begin{scope}[
                    overlay,
                    shift={(image3.south west)},
                    x={(image3.south east)},
                    y={(image3.north west)}
                ]
                    \node[rotate=-12] at (0.35,0.07)
                        {\tiny $x$ axis};
                    \node[rotate=90] at (-0.03,0.4)
                        {\tiny $v$ axis};
                    \node[rotate=50] at (0.1,0.79)
                        {\tiny time};
                \end{scope}
            \end{tikzpicture}
        \end{minipage}%
    }%
    \hfill
    \subfloat[\label{fig6d}]{%
        \begin{minipage}[c][3.5cm][c]{0.22\textwidth}
            \centering
            \begin{tikzpicture}
                \node[anchor=south west, inner sep=0cm] (image4)
                at (0,0) {%
                    \includegraphics[width=\linewidth]
                    {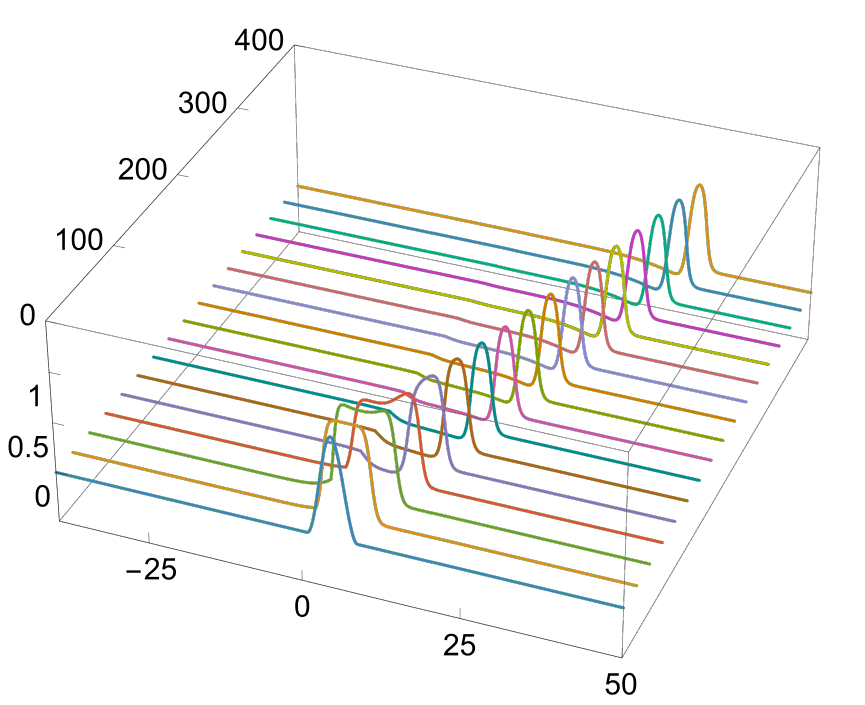}%
                };
                \begin{scope}[
                    overlay,
                    shift={(image4.south west)},
                    x={(image4.south east)},
                    y={(image4.north west)}
                ]
                    \node[rotate=-12] at (0.35,0.07)
                        {\tiny $x$ axis};
                    \node[rotate=90] at (-0.03,0.4)
                        {\tiny $v$ axis};
                    \node[rotate=50] at (0.1,0.79)
                        {\tiny time};
                \end{scope}
            \end{tikzpicture}
        \end{minipage}%
    }

    \caption{Pulse interaction with a localized conductivity reduction in Theorem~\ref{thm_barrier_heterogeneous}, case~\ref{thm_heterogeneity_type4}, with \(L=0.8\). (a) Barrier constructed with \(p=11\) and \(V_-=0.005\), together with the initial activator profile; the constriction occupies \(0\leq x\leq L\). (b)--(d) Activator evolution for \(p=1.37\), \(p=1.38\), and \(p=11\), respectively. 
    }
    \label{Fig6}
\end{figure}

\noindent 
\textbf{Experiment 4: Almost stationary barriers.} Finally, we illustrate the almost stationary barriers of Theorem~\ref{thm_barriers_conduction_block}. 
Unlike the preceding experiments, the medium is homogeneous, and the barriers provide threshold-region containment bounds whose speed tends to zero as the recovery timescale increases.
The simulations are performed on \([-10,15]\) up to time \(T=8\), with homogeneous Neumann boundary conditions on \(v\). We take 
\[ a=b=1,\qquad D=0,\qquad f(v)=v(1-v)(v-\alpha), \] 
and consider two parameter configurations. 

\begin{itemize}
\item For Figure~\ref{fig7a}--\ref{fig7b}, we use 
\( \alpha=0.7, \gamma=4, \varepsilon=\tau^{-1}=0.001, \)
and construct the barrier with 
\( K_1=3, K_2=\frac{K_1}{\gamma}\), and \(c(\varepsilon)=\varepsilon^{1/2}\). 
The two initial activator profiles are given by \eqref{initial_data} with 
\[ (A,x_0,h)=(1.2,3,6) \quad\text{and}\quad (A,x_0,h)=(3,6.7,3.2). \] 
\item For Figure~\ref{fig7c}--\ref{fig7d}, we take 
\( \alpha=0.55, \gamma=1.2, \varepsilon=0.0005, \)
so that \(\alpha\gamma=0.66<1\), and use 
\( K_1=2.8, K_2=\frac{K_1}{\gamma}\), and \(c(\varepsilon)=\varepsilon^{0.4}\). 
The corresponding initial-data parameters are 
\[ (A,x_0,h)=(0.93,3,6.2) \quad\text{and}\quad (A,x_0,h)=(2.8,8.4,3.5). \] 
This second configuration illustrates that the almost stationary construction remains applicable when condition~\eqref{condition_H1}.(ii) is not satisfied.
\end{itemize}
In both configurations, the barrier profiles are displayed at time increments \(\Delta t=5\), while the solutions are shown at 
\[ t\in\{0,0.15,0.8,2,4,8\}. \] 
The increasing opacity indicates increasing time. The barriers move slowly to the left, while the numerical solutions initiated from two distinct profiles remain below their upper activator components over the displayed time interval. Thus, the simulations illustrate the comparison property of the almost stationary barriers for different parameter regimes and different initial profiles. By translation and spatial reflection invariance of the homogeneous problem, reflected barriers provide the analogous constraint on the expansion of threshold-exceedance regions toward the right.

\begin{figure}
    \centering

\subfloat[\label{fig7a}]{%
    \begin{minipage}{0.24\linewidth}
        \centering
        \begin{tikzpicture}
            \node[anchor=south west, inner sep=0cm] (image1) at (0,0) {%
                \includegraphics[width=\linewidth]{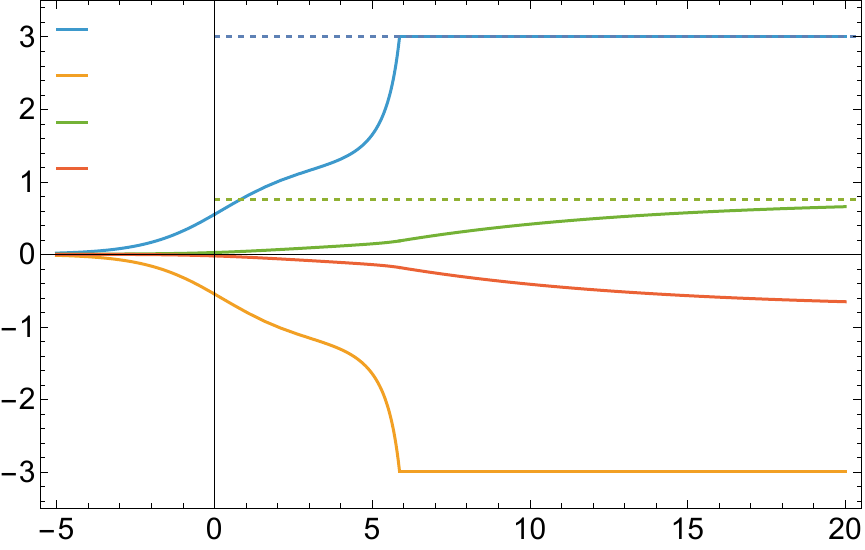}%
            };
            \begin{scope}[
                shift={(image1.south west)},
                x={(image1.south east)},
                y={(image1.north west)}
            ]
                \draw[-{Latex[length=2.5mm,width=1.5mm]}, thick]
                    (0.4,0.75) -- (0.25,0.75);
                \node at (0.13,0.94) {\tiny $\bar{v}$};
                \node at (0.13,0.85)
                    {\tiny $\underaccent{\bar}{v}$};
                \node at (0.13,0.76) {\tiny $\bar{w}$};
                \node at (0.13,0.67)
                    {\tiny $\underaccent{\bar}{w}$};
                \node at (0.21,0.92) {\tiny $K_1$};
                \node at (0.21,0.65) {\tiny $K_2$};
                \node at (0.5,-0.08) {$x$ axis};
            \end{scope}
        \end{tikzpicture}
    \end{minipage}%
}%
\hfill
\subfloat[\label{fig7b}]{%
    \begin{minipage}{0.24\linewidth}
        \centering
        \begin{tikzpicture}
            \node[anchor=south west, inner sep=0cm] (image2) at (0,0) {%
                \includegraphics[width=\linewidth]{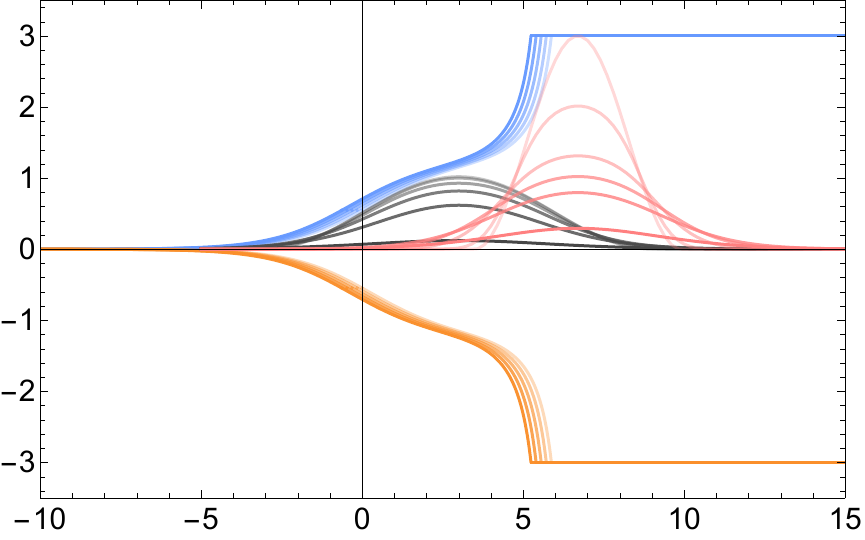}%
            };
            \begin{scope}[
                shift={(image2.south west)},
                x={(image2.south east)},
                y={(image2.north west)}
            ]
                \node at (0.5,-0.08) {$x$ axis};
            \end{scope}
        \end{tikzpicture}
    \end{minipage}%
}%
\hfill
\subfloat[\label{fig7c}]{%
    \begin{minipage}{0.24\linewidth}
        \centering
        \begin{tikzpicture}
            \node[anchor=south west, inner sep=0cm] (image3) at (0,0) {%
                \includegraphics[width=\linewidth]{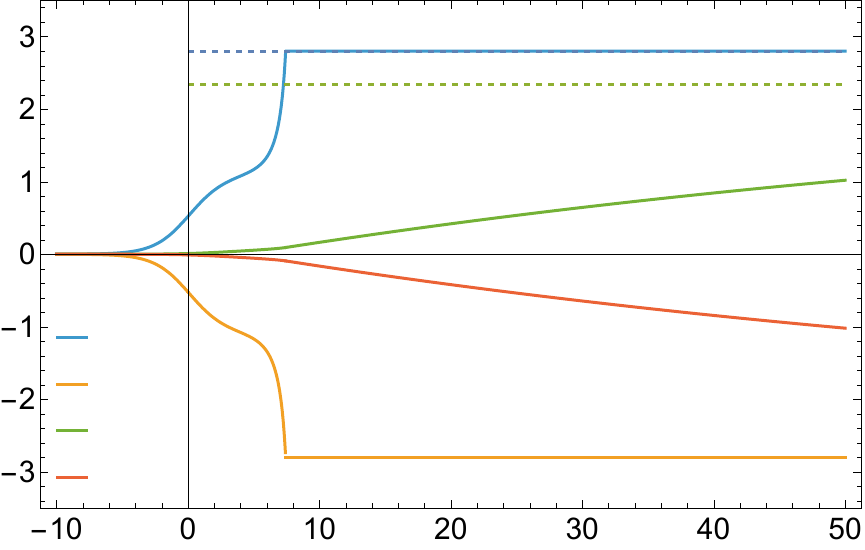}%
            };
            \begin{scope}[
                shift={(image3.south west)},
                x={(image3.south east)},
                y={(image3.north west)}
            ]
                \draw[-{Latex[length=2.5mm,width=1.5mm]}, thick]
                    (0.25,0.68) -- (0.12,0.68);
                \node at (0.13,0.39) {\tiny $\bar{v}$};
                \node at (0.13,0.30)
                    {\tiny $\underaccent{\bar}{v}$};
                \node at (0.13,0.21) {\tiny $\bar{w}$};
                \node at (0.13,0.13)
                    {\tiny $\underaccent{\bar}{w}$};
                \node at (0.17,0.92) {\tiny $K_1$};
                \node at (0.17,0.83) {\tiny $K_2$};
                \node at (0.5,-0.08) {$x$ axis};
            \end{scope}
        \end{tikzpicture}
    \end{minipage}%
}%
\hfill
\subfloat[\label{fig7d}]{%
    \begin{minipage}{0.24\linewidth}
        \centering
        \begin{tikzpicture}
            \node[anchor=south west, inner sep=0cm] (image4) at (0,0) {%
                \includegraphics[width=\linewidth]{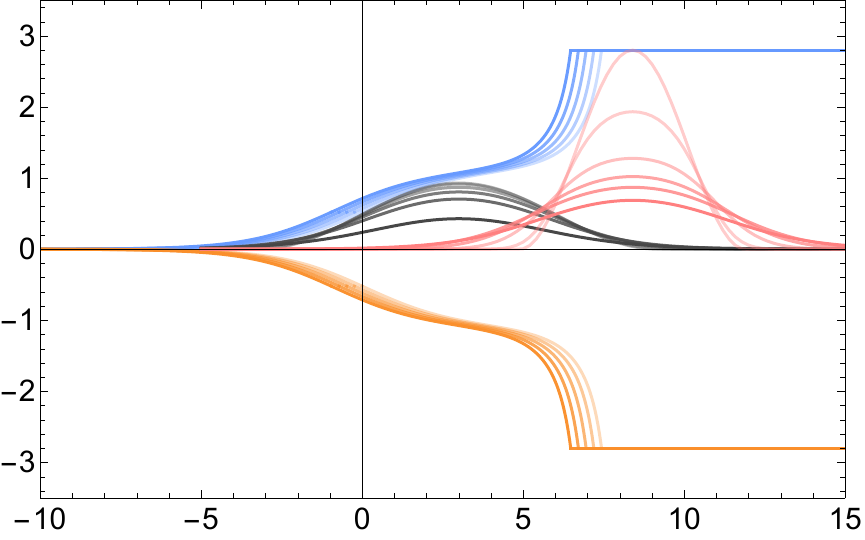}%
            };
            \begin{scope}[
                shift={(image4.south west)},
                x={(image4.south east)},
                y={(image4.north west)}
            ]
                \node at (0.5,-0.08) {$x$ axis};
            \end{scope}
        \end{tikzpicture}
    \end{minipage}%
}
    
    \caption{Almost stationary barriers and solution evolution for the homogeneous problem. (a)--(b) Parameters \(\alpha=0.7\), \(\gamma=4\), \(\varepsilon=0.001\), \(K_1=3\), and \(c(\varepsilon)=\varepsilon^{1/2}\). (c)--(d) Parameters \(\alpha=0.55\), \(\gamma=1.2\), \(\varepsilon=0.0005\), \(K_1=2.8\), and \(c(\varepsilon)=\varepsilon^{0.4}\). Panels (a) and (c) show the corresponding barriers, while panels (b) and (d) show the activator barriers together with solutions from two distinct profiles. In both configurations, \(K_2=K_1/\gamma\), and increasing opacity indicates increasing time.
    }
    \label{Fig7}
\end{figure}
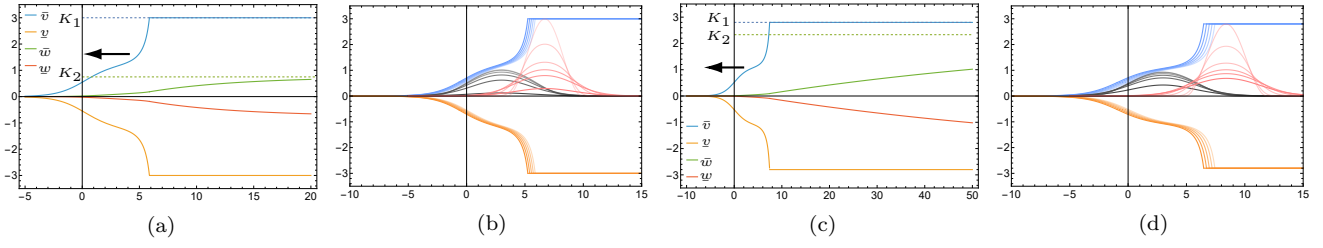

\section{Discussion and conclusions}\label{sec_conclusions}

The main contribution of this work is an explicit coupled-barrier method for guaranteeing failure of propagation in heterogeneous FitzHugh--Nagumo systems. The construction yields front-like upper and lower solutions for the nonlinear system that incorporate localized reaction defects, localized affine conductance terms, abrupt conductivity changes, and the corresponding transmission inequalities. When the initial data are confined between the upper and lower components of a stationary barrier, the comparison principle implies directional blocking. The resulting directional blocking criteria are directly verifiable for families of initial data and do not require an exact stationary solution or traveling wave of the original system.

Coupled upper and lower solutions for mixed quasimonotone systems have a classical foundation; the comparison formulation developed here is adapted to the analytical requirements of the explicit barrier constructions. In particular, it accommodates piecewise coefficients, transmission inequalities, moving barrier interfaces, unbounded domains, and nonnegative conductivities that may vanish, under the stated regularity hypotheses. The comparison principle then converts the constructed profiles into persistent bounds on the solutions.

The five stationary heterogeneous configurations in Theorem \ref{thm_barrier_heterogeneous} illustrate a unified construction strategy. In each instance, heterogeneity either places the barrier profile directly at a height where the homogeneous scalar inequality is favorable, or increases its outgoing slope sufficiently so that the profile subsequently enters that regime. Once the profile attains a positive slope at a favorable height, it can reach any prescribed height and be extended constantly while preserving the barrier and transmission inequalities. This mechanism explains why the required heterogeneity threshold can be chosen independently of the imposed minimum barrier height. 

The comparison component of this strategy applies to other mixed quasimonotone reaction--diffusion systems, while the explicit barrier profiles depend on additional model-specific structures. In the stationary symmetric construction, these include the particular activator--recovery coupling of the FitzHugh--Nagumo system; in the almost stationary construction, they also involve properties of the cubic reaction term. Therefore, extending the method to other coupled systems would require tailored barrier constructions.

The stationary barriers provide an obstruction to propagation in the blocked direction; however, the long-term dynamics in other regions of the spatial domain are not determined by the present analysis. In particular, the comparison result does not establish whether solutions confined by the barrier converge to the rest state, approach a localized stationary structure, or exhibit more complex long-term behavior. Characterizing the long-term dynamics of blocked solutions in these heterogeneous media remains an open problem.

The positive-recovery-diffusion result indicates that the stationary construction persists under small perturbations. Specifically, a barrier obtained for zero recovery diffusion remains valid for sufficiently small positive recovery diffusion, 
provided a gap is introduced between the coefficients $\tilde\gamma$ and $\gamma$. Numerical experiments suggest that larger recovery diffusion may increase susceptibility to heterogeneity-induced blocking, indicating that regimes beyond the perturbative result may involve effects not captured by the present barrier estimate.

The almost stationary barriers in Theorem \ref{thm_barriers_conduction_block} introduce a different approach that does not require the condition $\alpha\gamma>1$ and provides a quantitative result in homogeneous media.
These barriers yield conditional bounds on the expansion of regions where the solution exceeds the prescribed thresholds, establishing an $O(\tau^{-\beta})$ maximum expansion rate for any $\beta\in(0,1)$.
An open question is whether the propagation bounds can be refined or related to an asymptotic spreading speed under additional assumptions.

The numerical experiments illustrate the behaviors predicted by the analytical constructions and examine parameter regimes beyond those addressed by the theory. When the initial data are confined by an explicit barrier, the computed solutions remain within the predicted bounds throughout the simulated time interval. Parameter sweeps reveal qualitative transitions between propagation and finite-time blocking, demonstrate the conservativeness of the sufficient analytical conditions, and, in the recovery-diffusion simulations, indicate blocking behavior beyond the small-diffusion regime covered by the analysis. Understanding the gap between the theoretically sufficient heterogeneity thresholds established by the barriers and the dynamically observed transitions remains an open question.

\appendix
\section{A weighted positive-part chain rule}
\label{app:weighted_truncation}

The standard quadratic chain rule in an evolution triple is classical; see, for example, \cite[Chapter~3, Section~1]{lionsNonHomogeneousBoundaryValue1972}. The following positive-part version follows by temporal regularization and the Sobolev truncation theorem. Related nonlinear time-calculus arguments appear in \cite[Lemma~1.5]{altLuckhausQuasilinear1983}.

\begin{lemma}[Weighted positive-part chain rule]
\label{lem:weighted_truncation_identity}
Let $U\subset\mathbb R$ be a bounded interval, $S>0$, and $a\in L^\infty(U)$ satisfy $0<m\leq a\leq M$ almost everywhere. If $q\in L^2(0,S;H^1(U))$ and $\partial_t(aq)\in L^2(0,S;H^1(U)^*)$, then $q_+\in L^2(0,S;H^1(U))$ and $t\mapsto \frac12\int_U a(x)q_+(x,t)^2\,dx$ has an absolutely continuous representative. Moreover, for every $0\leq s\leq t\leq S$,
\begin{equation}
\label{eq:weighted_truncation_identity}
\frac12\int_U a q_+(t)^2\,dx-\frac12\int_U a q_+(s)^2\,dx
=
\int_s^t\bigl\langle\partial_\tau(aq)(\tau),q_+(\tau)\bigr\rangle_{H^1(U)^*,H^1(U)}\,d\tau.
\end{equation}
In particular, the corresponding differential identity holds for almost every $t\in(0,S)$.
\end{lemma}

\begin{proof}
Set $V=H^1(U)$ and $H_a=L^2(U,a(x)\,dx)$. Since $0<m\leq a\leq M$, one has the evolution triple $V\hookrightarrow H_a\hookrightarrow V^*$, where the embedding $H_a\hookrightarrow V^*$ is given by $r\mapsto(h\mapsto\int_U arh\,dx)$. Under this embedding, the weak time derivative of $q$ is identified with $\partial_t(aq)\in V^*$. The standard Lions--Magenes result therefore gives $q\in C([0,S];H_a)$. Moreover, the Sobolev truncation theorem gives $q_+\in L^2(0,S;V)$.

Let $\beta_\varepsilon\in C^1(\mathbb R)$ be nondecreasing, satisfy $\beta_\varepsilon(0)=0$ and $0\leq\beta_\varepsilon'\leq1$, and be chosen so that $\beta_\varepsilon(r)\to r_+$ and $\beta_\varepsilon'(r)\to\mathbf 1_{\{r>0\}}$ for $r\neq0$. Set $B_\varepsilon(r):=\int_0^r\beta_\varepsilon(\xi)\,d\xi$. For a temporal mollification $q_h$ of $q$, the classical chain rule gives
\[
\frac{d}{dt}\int_U a(x)B_\varepsilon(q_h(x,t))\,dx
=
\bigl\langle\partial_t(aq_h)(t),\beta_\varepsilon(q_h(t))\bigr\rangle_{V^*,V}.
\]
Since $q_h\to q$ in $L^2_{\mathrm{loc}}(0,S;V)$ and, because $a$ is independent of time, $\partial_t(aq_h)\to\partial_t(aq)$ in $L^2_{\mathrm{loc}}(0,S;V^*)$, while $\beta_\varepsilon(q_h)\to\beta_\varepsilon(q)$ in $L^2_{\mathrm{loc}}(0,S;V)$, letting $h\to0$ gives the same identity with $q$ in place of $q_h$ in $\mathcal D'(0,S)$.

As $\varepsilon\to0$, pointwise convergence and domination give $B_\varepsilon(q)\to\frac12(q_+)^2$ in $L^1(U\times(0,S))$, while the Sobolev truncation argument gives $\beta_\varepsilon(q)\to q_+$ in $L^2(0,S;V)$. Hence, passing to the limit yields
\[
\frac{d}{dt}\left(\frac12\int_U a(x)q_+(x,t)^2\,dx\right)
=
\bigl\langle\partial_t(aq)(t),q_+(t)\bigr\rangle_{V^*,V}
\]
in $\mathcal D'(0,S)$. The right-hand side belongs to $L^1(0,S)$, so the weighted energy has an absolutely continuous representative. Since $q\in C([0,S];H_a)$ and the positive-part map is Lipschitz on $H_a$, this representative agrees at every time with $\frac12\int_U a q_+^2\,dx$. Integration over $(s,t)$ now gives \eqref{eq:weighted_truncation_identity}.
\end{proof}

\section*{Declaration of generative AI and AI-assisted technologies in the man\-u\-script preparation process}
During the preparation of this work, the authors used Microsoft 365 Copilot to assist with literature searches and to improve the readability of the manuscript. Claude (Anthropic) was used for proofreading. After using these tools, the authors reviewed and edited the content as needed and take full responsibility for the content of the published article.

\section*{Funding}
This work has been partially supported by the grants ANID Fondecyt regular 1262107, ANID Fondecyt regular 1240200, and Fondecyt postdoctorado 3240512.


\end{document}